\documentclass[12pt]{article}
\pdfoutput=1

\usepackage[english]{babel}
\usepackage[T1]{fontenc}
\usepackage{ucs}
\usepackage[utf8x]{inputenc}
\usepackage[active]{srcltx}
\usepackage[refpage]{nomencl}
\usepackage{ifthen}
\usepackage{enumerate}
\usepackage{psfrag}
\usepackage{stmaryrd}
\usepackage{graphicx}
\usepackage[textwidth=25mm]{todonotes}

\RequirePackage{xcolor} 

\RequirePackage{amsmath,amsthm,amssymb,amsfonts,latexsym,mathrsfs}

\newcommand{\bd}{\mathrm{BD}}
\newcommand{\sis}{{\sigma_{\mathbf{S}}}}
\newcommand{\siv}{{\sigma_{\mathbf{V}}}}
\newcommand{\sie}{{\sigma_{\mathbf{E}}}}
\newcommand{\sif}{{\sigma_{\mathbf{F}}}}
\newcommand{\sS}{\mathbf{S}}
\newcommand{\sV}{\mathbf{V}}
\newcommand{\sE}{\mathbf{E}}
\newcommand{\sF}{\mathbf{F}}
\newcommand{\A}{\mathcal{E}}

\newcommand{\ee}{\mathbf{e}}

\newcommand{\CM}{\mathcal{M}}

\newcommand{\CZ}{\mathcal{Z}}
\newcommand{\SSL}{\mathscr{S}^\sS_L}

\newcommand{\bq}{\mathbf{q}}
 \newcommand{\bp}{\mathbf{p}}

\newcommand{\bm}{\mathbf{m}}
\newcommand{\bt}{\mathbf{t}}
\newcommand{\bff}{\mathbf{f}}

\newcommand{\bQ}{{\bf Q}}

\newcommand{\bB}{\mathbf{B}}
\newcommand{\bQns}{{\bf Q}_{l,n}}
\newcommand{\bQnsn}{{\bf Q}_{l_n,n}}

\newcommand{\Qn}{Q_n}

\newcommand{\ton}{\xrightarrow[n\to\infty]{}}

\newcommand{\R}{\mathbb{R}}
\newcommand{\N}{\mathbb{N}} 
\newcommand{\M}{\mathbb{M}}
\newcommand{\E}{\mathbb{E}}

\newcommand{\Z}{\mathbb{Z}}

\def\build#1_#2^#3{\mathrel{ \mathop{\kern 0pt#1}\limits_{#2}^{#3}}}

\def\un{\underline}
\def\d{\mathrm{d}}

\def\W{\mathbb{W}}
\newcommand{\WW}{\mathcal{W}}

\def\eps{\varepsilon}

\def\cov{{\rm Cov\,}}

\renewcommand{\P}{\mathbb{P}}

\newcommand{\ind}{\mathbf{1}}

\newcommand{\Img}{\mathrm{Im}}
\newcommand{\Zo}{Z^{\mathbf 0}}

\newtheorem{clm}{Claim}
\newtheorem{thm}{Theorem}
\newtheorem{lmm}[thm]{Lemma}
\newtheorem{prp}[thm]{Proposition}
\newtheorem{defn}[thm]{Definition}
\newtheorem{crl}[thm]{Corollary}

\newenvironment{proofc}{\begin{proof}}{\end{proof}}
\DeclareMathAlphabet{\mathpzc}{OT1}{pzc}{m}{it}

\theoremstyle{definition}
\newtheorem*{note}{Note}
\newtheorem{rem}{Remark}

\usepackage{fancyhdr}

\usepackage[outerbars]{changebar}
\usepackage{caption}
\usepackage[colorlinks=true,pdftex,urlcolor=blue,pdfstartview=FitH]{hyperref}

\title{Compact {B}rownian surfaces {I}. {B}rownian disks}

\author{J\'er\'emie Bettinelli\thanks{CNRS \& Institut {\'E}lie Cartan de Lorraine}
\and
Gr\'egory Miermont\thanks{ENS de Lyon \& Institut Universitaire de France}}

\renewcommand{\d}{\mathrm{d}}

\begin{document}

\selectlanguage{english}

\maketitle

\begin{abstract}
  We show that, under certain natural assumptions, large
  random plane bipartite maps with a boundary converge after rescaling to a
  one-parameter family $(\bd_L,\, 0<L<\infty)$ of random
  metric spaces homeomorphic to the closed unit disk of~$\R^2$, the
  space~$\bd_L$ being called the \emph{Brownian disk of
    perimeter~$L$} and unit area. These results can be seen as an extension of
  the convergence of uniform plane quadrangulations to the Brownian
  map, which intuitively corresponds to the limit case where
  $L=0$. Similar results are obtained for maps following a Boltzmann
  distribution, in which the perimeter is fixed but the area is
  random. 
\end{abstract}

\tableofcontents

\section{Introduction}\label{sec:introduction}

\subsection{Motivation}\label{sec:main-results}

Random maps are a natural discrete version of random surfaces. It has
been shown in recent years that their scaling limits can provide
``canonical'' models of random metric spaces homeomorphic to a surface
of a given topology. More precisely, given a random map $M$,
one can consider it as a random finite metric space by endowing its
vertex set with the usual graph metric, and multiply this graph metric
by a suitable renormalizing factor that converges to $0$ as the size
of the map $M$ is sent to infinity. One is then interested in the
convergence in distribution of the resulting sequence of rescaled
maps, in the Gromov--Hausdorff topology~\cite{gromov99} (or pointed
Gromov--Hausdorff topology if one is interested in non-compact
topologies), to some limiting random metric space. 

Until now, the topology for which this program has been carried out
completely is that of the sphere, for a large (and still growing)
family of different random maps models, see
\cite{legall11,miermont11,BeLG,addarioalbenque2013simple,BeJaMi14,abr14},
including for instance the case of uniform triangulations of the
sphere with $n$ faces, or uniform random maps of the sphere with $n$
edges. The limiting metric space, called the Brownian map, turns out
not only to have the topology of the sphere \cite{lgp,miermontsph}, as
can be expected, but also to be independent (up to a scale constant)
of the model of random maps that one chooses, provided it is, in some
sense, ``reasonable.'' See however \cite{AlMa,LGMi09} for natural
models of random maps that converge to qualitatively different
metric spaces.  These two facts indeed qualify the Brownian map as
being a canonical random geometry on the sphere.  Note that a
non-compact variant of the Brownian map, called the \emph{Brownian
  plane}, has been introduced in~\cite{CuLG12Bplane} and shown to be
the scaling limit of some natural models of random quadrangulations.

However, for other topologies allowing higher genera and
boundary components, only partial results are known
\cite{bettinelli10,bettinelli11,bettinelli11b,bettinelli14gbs}. Although
subsequential convergence results have been obtained for rescaled
random maps in general topologies, it has not been shown that the limit is
uniquely defined and independent of the choice of the extraction. The
goal of this paper and its companion~\cite{BeMi15b} is to fill in this gap
by showing convergence of a natural model of random maps on a given
compact surface $S$ to a random metric space with same topology, which
one naturally can call the ``Brownian~$S$.''

This paper will focus exclusively on the particular case of the disk
topology, which requires quite specific arguments, and indeed serves
as a building block to construct the boundaries of general compact
Brownian surfaces in~\cite{BeMi15b}.

\subsection{Maps}\label{sec:results}

To state our results, let us recall some important definitions and set
some notation. We first define the objects that will serve as discrete
models for a metric space with the disk topology. 

A \emph{plane map} is an embedding of a finite connected multigraph
into the $2$-dimensional sphere, and considered up to
orientation-preserving homeomorphisms of the latter. The \emph{faces}
of the map are the connected components of the complement of edges,
and can be then shown to be homeomorphic to $2$-dimensional open
disks. For every oriented edge $e$, with origin vertex $v$, we can
consider the oriented edge $e'$ that follows $e$ in counterclockwise
order around $v$, and define the {\em corner} incident to $e$ as a
small open angular sector between~$e$ and~$e'$. It does not matter how
we choose these regions as long as they are pairwise disjoint.  The
number of corners contained in a given face~$f$ is called the
\emph{degree} of that face; equivalently, it is the number of oriented
edges to the left of which~$f$ lies --- we say that $f$ is {\em
  incident} to these oriented edges, or to the corresponding corners.
We let $\sV(\bm)$, $\sE(\bm)$, $\sF(\bm)$ denote the sets of vertices, edges
and faces of a map $\bm$, or simply $\sV$, $\sE$, $\sF$ when the mention of
$\bm$ is clear from the context. 

If $\bm$ is a map, we can view it as a metric space
$(\sV(\bm),d_{\bm})$, where $d_\bm$ is the graph metric on the set
$\sV(\bm)$ of vertices of~$\bm$. For simplicity, we will sometimes
denote this metric space by $\bm$ as well and, if $a>0$, we denote by
$a\bm$ the metric space $(\sV(\bm),ad_{\bm})$.

For technical reasons, the maps we consider will always implicitly be
\emph{rooted}, which means that one of the corners (equivalently, one
of the oriented edges) is distinguished and called the \emph{root}. The face $f_*$ incident to the root is
called the {\em root face}. Since we want to consider objects with the
topology of a disk, we insist that the root face is an \emph{external
face} to the map, whose incident edges forms the {\em boundary} of
the map, and call its degree the {\em perimeter} of the map.  By
contrast, the non-root faces are called \emph{internal faces}. Note that
the boundary of the external face is in general not a simple curve
(see Figure~\ref{exdisk}). As a result, the topological space obtained
by removing the external face from the surface in which the map is
embedded is not necessarily a surface with a boundary, in the sense
that every point does not have a neighborhood homeomorphic to some
open set of $\R\times \R_+$. However, removing any Jordan domain from
the external face does of course result in a surface with a boundary,
which is homeomorphic to the $2$-dimensional disk.

\subsection{The case of quadrangulations}\label{sec:case-quadr}

The first part of the paper is concerned exclusively with a particular
family of maps, for which the results are the simplest to obtain and
to state.  A \emph{quadrangulation with a boundary} is a rooted plane
map whose internal faces all have degree~$4$. It is a simple
exercise to see that this implies in fact that the perimeter is
necessarily an even number.  For $l$, $n\in \N$, we let $\bQns$ be the
set of quadrangulations with a boundary having~$n$ internal faces and
perimeter~$2l$.

\begin{figure}[htb!]
	\centering\includegraphics[width=88mm]{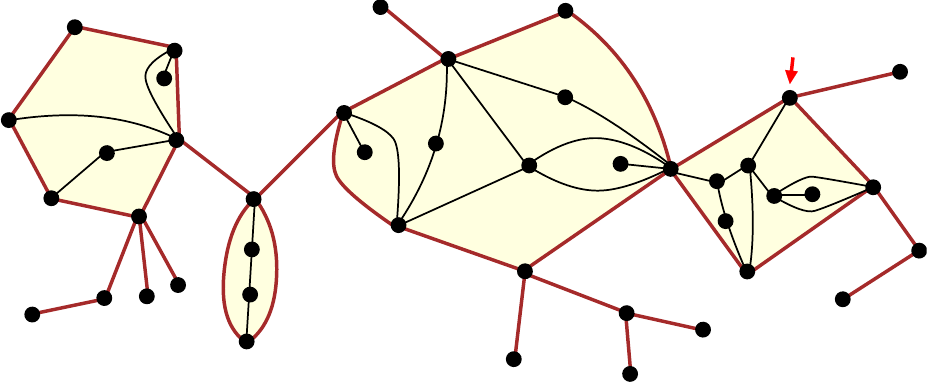}
	\caption{A quadrangulation from $\bQ_{23,19}$. The root is the corner indicated by the red arrow.}
	\label{exdisk}
\end{figure}

Our main result in the context of random quadrangulations is the following. 
\begin{thm}\label{THMDISK}
  Let $L\in [0,\infty)$ be fixed, and $(l_n,n\geq 1)$ be a
  sequence of integers such that $l_n/{\sqrt{2n}} \to L$
  as $n\to\infty$. Let $Q_n$ be uniformly distributed
  over~$\bQnsn$. There exists a random compact metric space
  $\bd_L$ such that
$$\Big(\frac{9}{8n}\Big)^{1/4}Q_n\build\longrightarrow_{n\to\infty}^{(d)}\bd_L$$
where the convergence holds in distribution for the Gromov--Hausdorff topology.
\end{thm}

The random metric space $\bd_L$ is called the {\em Brownian disk with
  perimeter $L$} and unit area. We will give in
Section~\ref{sec:defin-main-prop} an explicit description of $\bd_L$
(as well as versions with general areas, see also Section
\ref{sec:conv-boltzm-maps}) in terms of certain stochastic processes,
and the convention for the scaling constant $(9/8)^{1/4}$ is here to
make the description of these processes simpler. The main properties
of $\bd_L$ are the following; they follow from
\cite[Theorems~1--3]{bettinelli11b}.

\begin{prp}
  \label{sec:results-1}
  Let $L>0$ be fixed. Almost surely, the space $\bd_L$ is
  homeomorphic to the closed unit disk of~$\R^2$. Moreover, almost surely, the Hausdorff
  dimension of $\bd_L$ is~$4$, while that of its boundary $\partial \bd_L$ is~$2$.
\end{prp}

We stress that the case $L=0$, corresponding to the situation
where $l_n=o(\sqrt{n})$, is the statement of \cite[Theorem~4]{bettinelli11b}, which says that $\bd_0$ is the so-called Brownian
map. Since the Brownian map is a.s.\ homeomorphic to the sphere
\cite{lgp}, this means that the boundaries of the approximating random
maps are too small to be seen in the limit. This particular case
generalizes the convergence of uniform random quadrangulations,
obtained in \cite{legall11,miermont11}, corresponding to the case
where $l_n=2$ for every $n\geq 1$. 

The case where $l_n/\sqrt{n}\to\infty$ is also of interest, and
is the object of \cite[Theorem~5]{bettinelli11b}, showing that, in
this case, $(2l_n)^{-1/2}Q_n$ converges to the so-called Brownian
Continuum Random Tree~\cite{aldouscrt91,aldouscrt93}. This means that the boundary takes over the
planar geometry and folds the map into a tree-shaped object.

We will prove our result by using the already studied case of plane
maps without boundary, together with some surgical
methods. Heuristically, we will cut $\Qn$ along certain geodesics into
elementary pieces of planar topology, to which we can apply a variant
of the convergence of random spherical quadrangulations to the
Brownian map. The idea of cutting into {\em slices} quadrangulations
with a boundary along geodesics appears in Bouttier and Guitter
\cite{BoGu09,BoGu12}. The use of these slices (also called {\em maps
  with a piecewise geodesic boundary}) plays an important role in Le
Gall's approach \cite{legall11} to the uniqueness of the Brownian map
in the planar case, which requires to introduce the scaling limits of
these slices. The previously cited works are influential to our
approach.  It however requires to glue an infinite number of metric
spaces along geodesic boundaries, which could create potential
problems when passing to the limit.

\subsection{Universal aspects of the limit}\label{sec:universality}

Another important aspect is that of {\em universality} of the spaces
$\bd_L$. Indeed, we expect these spaces to be the scaling limit of
many other models of random maps with a boundary, as in the case of
the Brownian map, which corresponds to $L=0$.  In the latter case, it
has indeed been proved, starting in Le~Gall's work \cite{legall11},
that the Brownian map is the unique scaling limit for a large family
of natural models of discrete random maps, see
\cite{BeLG,addarioalbenque2013simple,BeJaMi14,abr14}. The now
classical approach to universality developed in \cite{legall11} can be
generalized to our context, as we illustrate in the case of {\em
  critical bipartite Boltzmann maps.}

\subsubsection{Boltzmann random maps}\label{secboltz}

Let $\bB$ be the set of bipartite rooted plane maps, that is, the set
of rooted plane maps with faces all having even degrees (equivalently,
this is the set of maps whose internal faces all have even degrees).
For $l\in \Z_+$, let~$\bB_l$ be the set of bipartite maps $\bm\in \bB$
with perimeter\footnote{By convention, the vertex map $\circ$
  consisting of no edges and only one vertex, ``bounding'' a face of
  degree~$0$, is considered as an element of $\bB$, so that
  $\bB_0=\{\circ\}$. It will only appear incidentally in the analysis.}~$2l$. 
Note that when $l=1$, meaning that the root face has degree $2$,
there is a natural bijection between $\bB_1$ and $\bB\setminus \bB_0$,
consisting in gluing together the two edges of the root face into one
edge.

Let $q=(q_1,q_2,\ldots)$ be a sequence of non-negative {\em
  weights}. We assume throughout that $q_i>0$ for at least one index
$i\geq 2$.  The Boltzmann measure associated with the sequence $q$ is
the measure $W(q;\cdot)$ on~$\bB$ defined by
$$W(q;\bm)=\prod_{f\in \sF(\bm)\setminus\{f_*\}}q_{\deg(f)/2}\, .$$
This defines a non-negative, $\sigma$-finite measure, and by
convention the vertex-map receives a weight $W(q,\circ)=1$. In what 
follows, the weight sequence $q$ is considered fixed and its mention 
will be implicit, so that we denote for example $W(\bm)=W(q;\bm)$,
and likewise for the variants of~$W$ to be defined below.

We aim at understanding various probability measures obtained by
conditioning $W$ with respect to certain specific subsets of $\bB$.
It is a simple exercise to check that $W(\bB_l)$ is non-zero for every
$l\in \N$, and that $W(\bB_l)$ is finite for one value of $l>0$ if and
only if it is finite for all values of $l>0$. In this case, it makes
sense to define the Boltzmann probability measures
$$\W_l=W(\cdot\, | \, \bB_l)=\frac{W(\cdot \cap \bB_l)}{W(\bB_l)}\, ,\qquad l\geq 0\, .$$
A random map with distribution~$\W_l$ has a
root face of fixed degree $2l$, but a random number of vertices, edges
and faces. 

Likewise, we can consider conditioned versions of~$W$ given both the perimeter and the
``size'' of the map, where the size can be alternatively the number of
vertices, edges or internal faces\footnote{We could also consider
other ways to measure the size of a map $\bm$, e.g.\ considering combinations of
the form $x_\sV|\sV(\bm)|+x_\sE|\sE(\bm)|+x_\sF|\sF(\bm)|$ for some
$x_\sV$, $x_\sE$, $x_\sF\geq 0$ with sum~$1$ as is done for instance
in~\cite{stephenson14} (in fact, due to the Euler formula, there is
really only one degree of freedom rather than two). We will not
address this here but we expect our results to hold in this context as
well.}. We let
$\bB^{\sV}_{l,n}$, $\bB^{\sE}_{l,n}$, $\bB^{\sF}_{l,n}$ be the subsets of
$\bB_l$ consisting of maps with respectively $n+1$ vertices, $n$ edges
and~$n$ internal faces. (The choice of $n+1$ vertices instead of a more natural choice of~$n$ vertices is technical and will make the statements simpler.)

In all the statements involving a given weight sequence $q$ and a
symbol $\sS\in \{\sV,\sE,\sF\}$ (for ``size''), it will always be tacitly imposed
that $(l,n)$ belongs to the set
$$\A^\sS(q)=\{(l,n)\in \Z_+^2:W(\bB^\sS_{l,n})>0\}\, .$$
Note that for $(l,n)\in \A^\sS(q)$, it holds that
$W(\bB^{\sS}_{l,n})<\infty$ since $W(\bB_l)<\infty$. 
In this way, we can define the distribution
$$\W^{\sS}_{l,n}(\cdot)=W(\cdot\, |\,
\bB^{\sS}_{l,n})\, .$$ 
It will be useful in the following to know what the set $\A^\sS(q)$ looks
like. More precisely, let
\begin{equation}\label{Asl}
\A^\sS_l(q)=\{n\geq
  0:(l,n)\in \A^\sS(q)\}\, .
\end{equation}
As above, when the weight sequence~$q$ is unequivocally fixed, we will drop the mention of it from the notation and write~$\A^\sS$ and~$\A^\sS_l$.

Define three numbers $h^\sV$, $h^\sE$, $h^\sF$ by 
\begin{equation}\label{hS}
h^\sV=\mathrm{gcd}(\{n\geq 2:q_{2n}>0\})\, ,\qquad 
h^\sE=\mathrm{gcd}(\{n\geq 1:q_{2n}>0\})\,,\qquad h^\sF=1\,.
\end{equation}
Then we have the following lemma, which is a slight generalization of
\cite[Section 6.3.1]{stephenson14}. 
\begin{lmm}
  \label{sec:boltzm-rand-maps}
  Let $q$ be a weight sequence, and let $\sS$ be one of the three
  symbols $\sV$, $\sE$, $\sF$. There exists an integer
  $\beta^\sS\geq 0$ such that for every $l\geq 1$, there exists a 
  set $R^\sS_l\subseteq \{0,1,\ldots,l\beta^\sS-1\}$ such that 
$$\A^\sS_l(q)=R^\sS_l\cup (l\beta^\sS+h^\sS\Z_+)\, .$$ 
\end{lmm}

In fact, note that $\A^\sF_l(q)=\Z_+$,
which amounts to the fact that, for any~$q$ and any $n\geq 0$, $l\geq 1$, there is at least one
map $\bm$ with $n$ internal faces and perimeter $2l$ such that $W(q;\bm)>0$. As a consequence, we can always take $\beta^\sF=0$.

\subsubsection{Admissible, regular critical weight sequences}\label{secZq}

Let us introduce some terminology taken from
\cite{MaMi07}. Let 
$$f_q(x)=\sum_{k\geq 0}x^k\binom{2k+1}{k}q_{k+1}\, ,\qquad x\geq
0\, .$$ This defines a totally monotone function with values in
$[0,\infty]$. 

\begin{defn}
  We say that $q$ is {\em admissible} if the equation 
\begin{equation}
  \label{eq:2} 
f_q(z)=1-\frac{1}{z}
\end{equation}
admits a solution $z>1$. We
also say that $q$ is {\em regular critical} if moreover this
solution satisfies 
$$z^2 f_q'(z)=1$$ and if there exists $\eps>0$ such
that $f_q(z+\eps)<\infty$. 
\end{defn}

Note that $q$ being regular critical means that the graphs of~$f_q$
and of $x\mapsto 1-1/x$ are tangent at the point of abscissa $z$, and
in particular, by convexity of~$f_q$, the solution~$z$ to~\eqref{eq:2}
is unique. We denote by
$$\CZ_q=z$$
this solution, which will play an important role in the discussion to
come.

To give a little more insight into this definition, let us introduce
at this point a measure on maps that looks less natural at first sight
than the Boltzmann measure~$\W_l$, but which will turn out to be
better-behaved from the bijective point of view on which this work
relies. Let $\bB^\bullet$ be the set of pairs $(\bm,v_*)$ where $\bm\in
\bB$ is a rooted bipartite map and $v_*\in \sV(\bm)$ is a
distinguished vertex. We also let $\bB^\bullet_l$ be the subset of
$\bB^\bullet$ consisting of the maps having perimeter~$2l$.  We let
$W^{\bullet}(q;\cdot)$ be the measure on $\bB^\bullet$ defined by
\begin{equation}\label{Wpoint}
W^{\bullet}(q;\{(\bm,v_*)\})=W(\bm)\,,\qquad (\bm,v_*)\in \bB^\bullet \, ,
\end{equation}
as well as the probability measures $\W^{\bullet}$ and
$\W^{\bullet}_l$, defined by conditioning $W^{\bullet}$ respectively
on $\bB^{\bullet}$ and $\bB^{\bullet}_l$.  Note that, if
$\phi(\bm,v_*)=\bm$ denotes the map from $\bB^\bullet$ to $\bB$ that
forgets the marked point, then $\W_l$ is absolutely continuous with
respect to $\phi_*\W^\bullet_l$, with density function given by
\begin{equation}
  \label{eq:16}
  \d \W_l(\bm)=\frac{K_l}{|\sV|} \d (\phi_*\W^{\bullet}_l)(\bm)\, , 
\end{equation}
where $|\sV|$ should be understood as the random variable $\bm\mapsto
|\sV(\bm)|$ giving the number of vertices of the map, and
$K_l=\W^\bullet_l[1/|\sV|]^{-1}$. This fact will be useful later. 

Proposition~1 in \cite{MaMi07} shows that the sequence~$q$ of
non-negative weights is admissible if and only if
$W^\bullet(q;\bB^\bullet_1)<\infty$ (this is in fact the defining
condition of admissibility in \cite{MaMi07}). We see that this clearly
implies that $W(q;\bB_1)<\infty$, and even that $W(q;\bB_l)<\infty$
for every $l\geq 1$.  Moreover, in this case, the constant $\CZ_q$ has
a nice interpretation in terms of the pointed measures. Namely, it
holds that
\begin{equation}
  \label{eq:1}
  \CZ_q=1+W^\bullet(\bB^{\bullet}_1)/2\, .
\end{equation}

From now on, our attention will be exclusively focused on regular
critical weight sequences. It is not obvious at this point how to
interpret the definition, which will become clearer when we see how to
code maps with decorated trees. However, let us explain now in which
context this property typically intervenes, and refer the reader to
the upcoming Subsection~\ref{sec:applications} for two applications. For
instance, if one wants to study uniform random quadrangulations with a
boundary and with $n$ faces as we did in the first part of this paper,
it is natural to consider the sequence $q^\circ=\delta_{2}=(0,1,0,0,\ldots)$ and to
note that $\W^{\sF}_{l,n}(q^\circ;\cdot)$ is the uniform distribution
on $\bQ_{l,n}$. Here, note that the sequence $q^\circ$ is not
admissible, but the probability measure
$\W^{\sF}_{l,n}(q^\circ;\cdot)$ does make sense because
$0<W(\bB^{\sF}_{l,n})<\infty$, due to the fact that there are finitely
many quadrangulations with a boundary of perimeter $2l$, and with~$n$
internal faces. Now, it can be checked that
$q=12^{-1}q^\circ$ is admissible and regular
critical, and that
$\W^{\sF}_{l,n}(q;\cdot)=\W^{\sF}_{l,n}(q^\circ;\cdot)$ is still the
uniform distribution on $\bQ_{l,n}$. This way of transforming a
``naturally given'' weight sequence $q^\circ$ into a regular weight
sequence while leaving $\W^{\sS}_{l,n}$ invariant is common and very
useful.

The main result is the following. Let~$q$ be a regular critical
weight sequence. Define
$\rho_q=2+\CZ_q^3f''_q(\CZ_q)$ and let
$\sie$, $\siv$, $\sif$ be the non-negative numbers with squares
\begin{equation}
  \label{eq:3}
  \sie^2=\frac{\rho_q}{\CZ_q}\, ,\quad
\siv^2=\rho_q \, ,\quad \sif^2=\frac{\rho_q}{\CZ_q-1}\,.
\end{equation}
For $L>0$, we denote by~$\SSL$ the set of sequences $(l_k,n_k)_{k\geq 0}\in (\A^\sS)^\N$ such that $l_k$, $n_k\to \infty$ with $l_k\sim L\sis\sqrt{n_k}$ as $k\to\infty$.

\begin{thm}\label{sec:admiss-regul-crit}
  Let~$\sS$ denote one of the symbols $\sV$, $\sE$, $\sF$, and $(l_k,n_k)_{k\geq 0}\in \SSL$ for some $L>0$. For $k\ge 0$, denote by~$M_k$ a random map with distribution $\W^{\sS}_{l_k,n_k}$. Then
$$\left(\frac{4\sis^2}{9}\, n_k\right)^{-1/4}M_k\build\longrightarrow_{k\to\infty}^{(d)}\bd_{L}$$
in distribution for the Gromov--Hausdorff topology. 
\end{thm}

\begin{rem}
The intuitive meaning for these
renormalization constants is the following: in a large random map with
Boltzmann distribution, it can be checked that the numbers $|\sV|$ and $|\sF|$ of
vertices and faces are of order $|\sE|/\CZ_q$ and $|\sE|(1-1/\CZ_q)$
respectively, where $|\sE|$ is the number of edges, and that
conditioning on having $n$ edges is asymptotically the same as
conditioning on having (approximately) $n/\CZ_q$ vertices, or
$n\,(1-1/\CZ_q)$ faces. 
\end{rem}

\begin{rem}
  In fact, the above result is also valid in the case where $L=0$,
  with the interpretation that $\bd_0$ is the Brownian map. The proof
  of this claim can be obtained by following ideas similar
  to~\cite[Section 6.1]{bettinelli11b}. However, a full proof requires the
  convergence of a map with law $\W_{1,n}^\sS$, rescaled by
  $(4\sis^2n/9)^{1/4}$, to the Brownian map, and this has been
  explicitly done only in the case where $\sS=\sV$ in \cite[Section
  9]{legall11}. In fact, building on the existing literature
  \cite{MaMi07,miergwmulti}, it is easy to adapt the argument to work
  for $\sS=\sF$ in the same way, while the case $\sS=\sE$, which is
  slightly different, can be tackled by the methods of
  \cite{abr14}. Writing all the details would add a consequent
  number of pages to this already lengthy paper, so we will omit the
  proof. 
\end{rem}

\subsubsection{Applications}\label{sec:applications}

Let us give two interesting specializations of Theorem~\ref{sec:admiss-regul-crit}. 
If $p\geq 2$ is an integer, a
$2p$-angulation with a boundary is a map whose internal faces all have
degree $2p$. The
computations of the various constants appearing in the statement of
Theorem \ref{sec:admiss-regul-crit}
have been performed in Section~1.5.1 of \cite{MaMi07}. These show that
the weight sequence 
$$q=\frac{(p-1)^{p-1}}{p^p\binom{2p-1}{p}}\delta_p$$
is regular critical, that $\W^{\sF}_{l,n}$ is the uniform law on
the set of $2p$-angulations with~$n$ faces and perimeter~$2l$ in this
case, and that the constants are
$$\CZ_q=\frac{p}{p-1}\, ,\quad \rho_q=p\, ,\quad \sie^2=p-1\,
,\quad \siv^2=p\, ,\quad \sif^2=p(p-1)\, .$$ Therefore, in
this situation, Theorem \ref{sec:admiss-regul-crit} for $\sS=\sF$
gives the following result, that clearly generalizes Theorem
\ref{THMDISK}. 

\begin{crl}\label{thmuniv}
  Let $L\in (0,\infty)$ be fixed, $(l_n,n\geq 1)$ be a sequence of
  integers such that $l_n\sim L\sqrt{p(p-1)n}$ as $n\to\infty$, and
  $M_n$ be uniformly distributed over the set of $2p$-angulations with
  $n$ internal faces and with perimeter $2l_n$. Then the following convergence
  holds in distribution for the Gromov--Hausdorff topology:
$$\Big(\frac{9}{4p(p-1)\,n}\Big)^{1/4}M_n\build\longrightarrow_{n\to\infty}^{(d)}\bd_L.$$
\end{crl}

Next, consider the case where $q_k=a^{-k}$, $k\geq 1$ for some $a>0$. In this case, for
every $\bm\in \bB$, a simple computation shows that
$$W(\bm)=a^{-|\sE(\bm)|+l}$$
so that $\W^{\sE}_{l,n}$ is the uniform distribution over bipartite
maps with~$n$ edges and a perimeter~$2l$. It was shown in
\cite[Section~1.5.2]{MaMi07} (and implicitly recovered in
\cite[Proposition~2]{abr14}) that choosing $a=1/8$ makes $q$ regular
critical and that, in this case, 
$$\CZ_q=\frac{3}{2}\, ,\qquad \rho_q=\frac{27}{4}\, ,\qquad
\sie^2=\frac{9}{2}\, .$$
Thus, one deduces the following statement, that should be compared to
\cite[Theorem 1]{abr14}.

\begin{crl}
  \label{sec:applications-1}
Let $M_n$ be a uniform random bipartite map with~$n$ edges and with
perimeter~$2l_n$, where $l_n\sim 3L\sqrt{n/2}$ for some $L>0$. Then  the following
  convergence holds in distribution for the Gromov--Hausdorff
  topology:
$$(2n)^{-1/4}M_n\build\longrightarrow_{n\to\infty}^{(d)}\bd_L.$$
\end{crl}

\subsection{Convergence of Boltzmann maps}\label{sec:conv-boltzm-maps}

The models we have presented so far consist in taking a random map
with a fixed size and perimeter and letting both these quantities go
to infinity in an appropriate regime. However, it is legitimate to ask
about the behavior of a typical random map with law $\W_l$ or
$\W^\bullet_l$ when $l\to\infty$, so that the perimeter is fixed and
large, while the total size is left free.

For every $L\geq 0$ and $A>0$, we define a random metric space
$\bd_{L,A}$, which we interpret as the Brownian disk with area~$A$ and
perimeter~$L$. For concreteness, the space $\bd_{L,A}$ has same
distribution as $A^{1/4}\,\bd_{A^{-1/2}L}$. To motivate the
definition, note that $\bd_{L,1}$ has same distribution as $\bd_L$ and
that if $Q_n$ is a uniform random element in $\bQ_{\lfloor
  L\sqrt{2n}\rfloor , \lfloor An\rfloor }$, then $(8n/9)^{-1/4}Q_n$
converges in distribution for the Gromov--Hausdorff topology to
$\bd_{L,A}$ by virtue of Theorem \ref{THMDISK}. See also
Remark~\ref{sec:brownian-disks} in Section~\ref{sec:br-disks} below.

Let $\mathcal{A}^\bullet$ be a stable random variable with index $1/2$, with distribution
given by 
$$\P(\mathcal{A}^\bullet\in \d A)=\frac{1}{\sqrt{2\pi
    A^3}}\exp\left(-\frac{1}{2 A}\right)\d A\, \ind_{\{A>0\}}\, .$$
Note that $\E[1/\mathcal{A}^\bullet]=1$, so that the formula
$$\frac{\P(\mathcal{A}^\bullet\in
  \d A)}{ A}=\frac{1}{\sqrt{2\pi A^5}}\exp\left(-\frac{1}{2A}\right)\d
A\, \ind_{\{A>0\}}$$ also defines a probability distribution, and we
let $\mathcal{A}$ be a random variable with this distribution. We
define the \emph{free Brownian disk} with perimeter~$1$ to be a space with
same law as $\bd_{1,\mathcal{A}}$, where this notation means that
conditionally given $\mathcal{A}=A$, it has same distribution as
$\bd_{1,A}$.  Likewise, the \emph{free pointed Brownian disk} with perimeter $1$ has
same distribution as $\bd_{1,\mathcal{A}^\bullet}$. 

For future reference, for $L>0$, it is natural to define 
the law of the free Brownian disk (resp.\ free pointed Brownian disk)
with perimeter~$L$ by scaling, setting it to be the law of
$\sqrt{L}\,\bd_{1,\mathcal{A}}$ or equivalently of
$\bd_{L,L^2\mathcal{A}}$ (resp.\ 
$\sqrt{L}\,\bd_{1,\mathcal{A}^\bullet}=^{(d)}\bd_{L,L^2\mathcal{A}^\bullet}$).
We let $\mathrm{FBD}_L$ (resp.\ $\mathrm{FBD}^\bullet_L$) stand for
the free Brownian disk (resp.\ free pointed Brownian disk) with
perimeter $L$.

\begin{thm}\label{thmboltz_nfree}
  Let $q$ be a regular critical weight sequence. For $l\in\N$, let
  $B_l$ (resp.\ $B^\bullet_l$) be distributed according to~$\W_l(q;\cdot)$
  (resp.\ $\W^\bullet_l(q;\cdot)$). Then
$$\left(\frac{2l}{3}\right)^{-1/2}B_l\build\longrightarrow_{l\to\infty}^{(d)}\mathrm{FBD}_1$$
and respectively
$$\left(\frac{2l}{3}\right)^{-1/2}B^\bullet_l\build\longrightarrow_{l\to\infty}^{(d)}\mathrm{FBD}^\bullet_1$$
in distribution for the the Gromov--Hausdorff topology.
\end{thm}

It is remarkable that the renormalization in this theorem does not
involve $q$ whatsoever!

\subsection{Further comments and organization of the paper}\label{sec:furth-comm-organ}

The very recent preprint \cite{MiSh15charac} by Miller and Sheffield aims at
providing an axiomatic characterization of the Brownian map in terms
of elementary properties. In this work, certain measures on random
disks play a central role. We expect that these measures, denoted by
$\mu_{\mathrm{DISK}}^{k,L}$ for $k\in \{0,1\}$ and $L>0$, are
respectively the laws of the free Brownian disk ($k=0$) and the pointed
free Brownian disk ($k=1$) with perimeter $L>0$. Miller and Sheffield
define these measures directly in terms of the metric balls in certain
versions of the Brownian map, and it is not immediate, though it is
arguably very likely, that this definition matches the one given in
the present paper. Establishing such a connection would be interesting
from the perspective of \cite{MiSh15charac} since, for example, it is not
established that $\mu_{\mathrm{DISK}}^{k,L}$ is
supported on compact metric spaces, due to the possibly wild behavior
of the boundary from a metric point of view. We hope to address such
questions in future work. 

Note also that \cite{MiSh15charac} introduces
another measure on metric spaces, called $\mu_{\mathrm{MET}}^L$, which
intuitively corresponds to the law of a variant of a metric ball in
the Brownian map, with a given boundary length. A description of this
measure in terms of slices is given in \cite{MiSh15charac}, which is very
much similar to the one we describe in the current work. However,
there is a fundamental difference, which is that
$\mu_{\mathrm{MET}}^L$ does not satisfy the invariance under
re-rooting that is essential to our study of random disks. In a few words,
in a random disk with distribution $\mu_{\mathrm{MET}}^L$, all points
of the boundary are equidistant from some special point (the center of
the ball), while it is very likely that no such point exists a.s.\ in
$\bd_{L,A}$, or under the law $\mu_{\mathrm{DISK}}^{k,L}$.

It would be natural to consider the operation that consists in gluing
Brownian disks, say with same perimeter, along their boundaries, hence
constructing what should intuitively be a random sphere with a
self-avoiding loop. However, this operation is in general badly
behaved from a metric point of view (in the sense of \cite[Chapter~3]{burago01} say), and it is not clear that the resulting space has
the same topology as the topological gluing. The reason for this
difficulty is that we require to glue along curves that are not
Lipschitz, since the boundaries of the spaces $\bd_L$ have Hausdorff
dimension~$2$ (by contrast, the gluings considered in Section~\ref{sec:conv-brown-disk} of the present paper are all along
geodesics.)  At present, such questions remain to be investigated.

The rest of the paper is organized as follows. In
Section~\ref{sec:defin-main-prop}, we give a self-standing definition
of the limiting objects. As in many papers on random maps, we rely on
bijective tools, and Section~\ref{sec:scha-biject-first} introduces
these tools. Section~\ref{sec:known-scaling-limit} gives a technical
result of convergence of slices, which are the elementary pieces from which the
Brownian disks are constructed.
Section~\ref{sec:conv-brown-disk} is dedicated to the proof of
Theorem~\ref{THMDISK}. In
Sections~\ref{sec:univ}--\ref{sec:conv-brown-disk-4}, we address the
question of universality and prove
Theorems~\ref{sec:admiss-regul-crit} and~\ref{thmboltz_nfree}.

\bigskip

\noindent{\bf Acknowledgments. }
This work is partly supported by the GRAAL grant ANR-14-CE25-0014.  We
also acknowledge partial support from the Isaac Newton Institute for
Mathematical Sciences where part of this work was conducted, and where
G.M.\ benefited from a Rothschild Visiting Professor position during
January 2015.

We thank Erich Baur, Timothy Budd, Guillaume Chapuy, Nicolas Curien,
Igor Kortchemski, Jean-Fran\c{c}ois Le Gall, Jason Miller, Gourab Ray
and Scott Sheffield, for useful remarks and conversations during the
elaboration of this work.

\section{Definition of Brownian disks}\label{sec:defin-main-prop}

Recall that the Brownian map $\bd_0$ is defined (\cite{legall06}, see
also \cite{MM05} and Section~\ref{sec:subs-conv} below) in terms of a certain stochastic process called the
normalized Brownian snake. Likewise, the spaces $\bd_L$, $L>0$
of Theorem~\ref{THMDISK} are defined in terms of stochastic processes,
as we now discuss.

\subsection{First-passage bridges and random continuum forests}\label{sec:first-pass-bridg}

The first building blocks of the Brownian disks are first-passage
bridges of Brownian motion. Informally, given~$A$, $L>0$, the
first-passage bridge at level~$-L$ and time~$A$ is a Brownian
motion conditioned to first hit~$-L$ at time~$A$. 
To be more precise, let us introduce some notation. We let~$X$ be the
canonical continuous process, and $\mathcal{G}_s=\sigma(X_u,u\leq s)$
be the associated canonical filtration. Denote by~$\mathbb{P}$ the law
of standard Brownian motion, and by $\mathbb{P}^A$ the law of standard
Brownian motion killed at time $A>0$.  For $L\geq 0$, let
$T_L=\inf\{s\geq 0:X_s=-L\}$ be the first hitting time of $-L$. We
denote the density function of its law by
\begin{equation}\label{jLA}
j_L(A)=\frac{\mathbb{P}(T_L\in \d A)}{\d A}=\frac{L}{\sqrt{2\pi
    A^3}}\exp\left(-\frac{L^2}{2A}\right)\, .
\end{equation}

With this notation, the law $\mathbb{F}^A_L$ of the first-passage
bridge at level $-L$ and at time $A$ can informally be seen as
$\mathbb{P}^A(\,\cdot\, |\, T_L=A)$. It is best defined by an absolute
continuity relation with respect to $\mathbb{P}^A$. Namely, for every
$s\in (0,A)$ and every non-negative random variable~$G$ that is
measurable with respect to $\mathcal{G}_s$, we let
\begin{equation}
  \label{eq:14}
  \mathbb{F}^A_L(G)=\mathbb{P}^A\left[G\,\ind_{\{T_L>s\}}\frac{j_{L+X_s}(A-s)}{j_L(A)}\right]\,
  .
\end{equation}
It can be seen~\cite{ChPM} that this definition is consistent and uniquely extends
to a law $\mathbb{F}^A_L$ on $\mathcal{G}_A$, supported on continuous
processes, and for which $\mathbb{F}^A_L(T_L=A)=1$. 

An alternative description of first-passage bridges, which will be
useful to us later, is the following. 

\begin{prp}
  \label{sec:first-pass-bridg-1}
  Let $A$, $L>0$. Then for every $y\in (0,L)$ and for every non-negative
  random variable $G$ that is measurable with respect to
  $\mathcal{G}_{T_y}$, we have
  \begin{equation}
    \label{eq:15}
    \mathbb{F}^A_L[G]=\mathbb{P}^A\left[G\,\ind_{\{T_y<A\}}\frac{j_{L-y}(A-T_y)}{j_L(A)}\right]\,
.
\end{equation}
Moreover, this property characterizes $\mathbb{F}^A_L$ among all
measures on $\mathcal{G}_A$ supported on continuous functions. 
\end{prp}

\begin{proof}
  The definition of $\mathbb{F}^A_L$ implies that the process
  $\CM=(\mathbf{1}_{\{T_L>s\}}j_{L+X_s}(A-s)/j_L(A),\,0\leq s<A)$ is a
  $(\mathcal{G}_s,\,0\leq s<A)$-martingale. Therefore, for every
  stopping time $T$ such that $T<A$ a.s.\ under $\mathbb{F}^A_L$, and
  for every $E\in \mathcal{G}_T$, we have
$$\mathbb{F}_L^A(E)=\lim_{s\uparrow A}\mathbb{P}^A[\ind_{E\cap\{T\leq
  s\}}\CM_s]= \lim_{s\uparrow A}\mathbb{P}^A[\ind_{E\cap\{T\leq
  s\}}\mathbb{P}^A[\CM_s\,|\,\mathcal{G}_T]]=\lim_{s\uparrow
  A}\mathbb{P}^A[\ind_{E\cap \{T\leq s\}}\CM_T]\, ,$$ and this is equal
to $\mathbb{P}^A[\ind_{E}\CM_T]$. The formula is obtained by
applying this result to $T=T_y$, and by a standard approximation
procedure of a general measurable function by weighted sums of
indicator functions. 

The fact that $\mathbb{F}^A_L$ is characterized by these formulas
comes from the following observation. Define $\tilde{\mathbb{F}}^A_L$
on $\mathcal{G}_{T_y}$ as being absolutely continuous with respect to
$\mathbb{P}^A|_{\mathcal{G}_{T_y}}$, with density~$\CM_{T_y}$. Then for
every $s<A$,
$\tilde{\mathbb{F}}^A_L(T_y<s)=\mathbb{P}^A[\ind_{\{T_y<s\}}j_{L-y}(A-T_y)/j_L(A)]$,
and this clearly converges to $0$ as $y\uparrow L$. Therefore, $T_y$
converges $\tilde{\mathbb{F}}^A_L$-a.s.\ to $A$ as $y\uparrow L$. Then
for every $s<A$ and $E\in \mathcal{G}_s$, similar manipulations to the
above ones show that
$$\tilde{\mathbb{F}}^A_L(E)=\lim_{y\uparrow
  L}\tilde{\mathbb{F}}^A_L(E\cap\{T_y>s\})=\lim_{y\uparrow
  L}\mathbb{P}^A\left[\ind_{E\cap\{T_y>s\}}\CM_{T_y}\right]=\lim_{y\uparrow
  L}\mathbb{P}^A\left[\ind_{E\cap\{T_y>s\}}\CM_{s}\right]$$
and this is $\lim_{y\uparrow
  L}\mathbb{F}^A_L(E\cap\{T_y>s\})=\mathbb{F}^A_L(E)$. 
\end{proof}

It is convenient to view a first-passage bridge as encoding a random
continuum forest. This is a classical construction that can be
summarized as follows, see for instance \cite{pitmancsp02}. Here we
work under~$\mathbb F^A_L$. For $0\leq s\leq s'\leq A$, define
$\underline{X}_{s,s'}=\inf\{X_u:s\leq u\leq s'\}$ and let
\begin{equation}\label{dX}
d_X(s,s')=X_s+X_{s'}-2\underline{X}_{s\wedge s',s\vee s'}\qquad s,s'\in[0,A].
\end{equation}
The function $d_X$ on $[0,A]^2$ is a pseudo-metric, to which one can
associate a random metric space $\mathcal{F}^A_L=[0,A]/\{d_X=0\}$,
endowed with the quotient metric induced from~$d_X$. This metric space
is a.s.\ a compact $\R$-tree, that is, a compact geodesic metric space
into which $\mathbb{S}^1$ cannot be embedded. It comes with a
distinguished geodesic of length $L$, which is the image of the first
hitting times $\{T_y,0\leq y\leq L\}$ under the canonical projection
$p_X:[0,A]\to \mathcal{F}^A_L$. It is convenient to view this segment
as the \emph{floor} of a forest of $\R$-trees, these trees being
exactly of the form $\mathcal{T}_y=p_X((T_{y-},T_y])$, corresponding
to the excursions of $X$ above its past infimum.  One should imagine
that the $\R$-tree $\mathcal{T}_y$ is grafted at the point $p_X(T_y)$
of the floor lying at distance $y$ from
$p_X(0)$. 

\subsection{Snakes}\label{sec:snakes}

We now enrich the random ``real forest'' described above by assigning
labels to it. Informally speaking, the trees of the forest are labeled
by independent Brownian snakes \cite{legall99,duqleg02}, while the
floor of the forest is labeled by a Brownian bridge with variance
factor $3$.

More precisely, let $X$ be a first-passage bridge with law
$\mathbb{F}^A_L$. 
Conditionally given $X$, we let $(\Zo_s,0\leq s\leq A)$ be a
centered Gaussian process with covariance function
\begin{equation}
  \label{eq:24}
  \cov(\Zo_s,\Zo_{s'})=\inf_{u\in [s\wedge s',s\vee
  s']}(X_u-\underline{X}_u)\qquad s,s'\in[0,A]\,,
\end{equation}
where $\underline{X}_u=\inf_{0\leq
  v\leq u}X_v$ is the past infimum of $X$. Note in particular that
$\Zo_s$ and $\Zo_{s'}$ are independent if $s$, $s'$ belong to two
different excursion intervals of $X$ above $\underline{X}$. It is
classical \cite{legall99} that $\Zo$ admits a continuous modification,
see also \cite{bettinelli10} for a discussion in the current
context. For this modification, we a.s.\ have $\Zo_{T_y}=0$ for every $y\in
[0,L]$ (for a given $y$, this comes directly from the variance
formula). The process $\Zo$ is sometimes called the {\em head of the
  Brownian snake} driven by the process $X-\underline{X}$, the reason
being that it can be obtained as a specialization of a path-valued
Markov process called the Brownian snake \cite{legall99} driven by
$X-\underline{X}$. The process $\Zo$ itself is not Markov. 

Let also $\mathrm{b}$ be a standard Brownian bridge of
duration $L$, so that
$$\cov(\mathrm{b}_y,\mathrm{b}_{y'})=\frac{y(L-y')}{L}\, ,\qquad 0\leq y\leq y'\leq L\, .$$
We define the process $Z$ to be 
\begin{equation}\label{defZ}
Z_s=\Zo_s+\sqrt{3}\,\mathrm{b}_{T^{-1}(s)}\, ,\qquad 0\leq s\leq A\, ,
\end{equation}
where $T^{-1}(s)=\sup\{y\geq 0:T_y\leq s\}$.  We abuse notation and
still denote by $\mathbb{F}^A_L$ the law of the pair $(X,Z)$ thus
defined, so that $\mathbb{F}^A_L$ is seen as a probability
distribution on the space $\mathcal{C}([0,A],\R)^2$. In the same
spirit, we will still denote by $\mathcal{G}_t$ the natural filtration
$\sigma(\{(X_s,Z_s),\,0\leq s\leq t\})$.  Note that the absolute
continuity relations \eqref{eq:14} and \eqref{eq:15} are still valid
{\it verbatim} with these extended notation and, in particular, the
density function involves only~$X$ and not~$Z$.

It is classical that a.s.\ under $\mathbb{F}^A_L$, $Z$ is a class function
on $[0,A]$ for the equivalence relation $\{d_X=0\}$, so that $Z$ can
also be seen as a function on the forest~$\mathcal{F}^A_L$. Note that
$T^{-1}(T_y)=y$ for every $0\leq y\leq L$, which corresponds to the
fact that, in the above depiction of the random forest, the point
$p_X(T_y)$ receives label~$\sqrt{3}\,\mathrm{b}_y$.

It is a simple exercise to check that the above definition of~$Z$ is equivalent to the following quicker (but more
obscure) one. Conditionally given $X$,
we have that $Z$ is Gaussian, centered, with covariance function
$$\cov(Z_s,Z_{s'})=\underline{X}_{s,s'}-\underline{X}_{s'}-3
\underline{X}_{s}(L+\underline{X}_{s'})/L\qquad s,s'\in[0,A]\,.$$

Similarly as~\eqref{dX}, we define a pseudo-metric using the process~$Z$
instead of~$X$, but with an extra twist. As above, let 
$\underline{Z}_{s,s'}=\inf\{Z_u:u\in [s,s']\}$ for $0\leq s\leq s' \leq
A$, and this time we extend the definition to $0\leq s'<s\leq A$ by
setting
$$\underline{Z}_{s,s'}=\inf\{Z_u:u\in [s,A]\cup
[0,s']\}=\underline{Z}_{s,A}\wedge \underline{Z}_{0,s'}\, ,$$
so if we see $[0,A]$ as a circle by identifying $0$ with $A$,
$\underline{Z}_{s,s'}$ is the minimum of $Z$ on the directed arc from
$s$ to $s'$. We let
\begin{equation}\label{dZ}
d_Z(s,s')=Z_s+Z_{s'}-2\max(\underline{Z}_{s,s'},\underline{Z}_{s',s})\qquad s,s'\in[0,A]\,.
\end{equation}

\subsection{Brownian disks}\label{sec:br-disks}

We are now ready to give the definition of Brownian disks.  Consider
the set $\mathcal{D}$ of all pseudo-metrics $d$ on $[0,A]$
satisfying the two properties
$$\left\{\begin{array}{c}
\{d_X=0\}\subseteq \{d=0\}\\
d\leq d_Z\, .
  \end{array}
\right. $$ The set $\mathcal{D}$ is nonempty (it contains the zero
pseudo-metric) and contains a maximal element~$D^*$ defined by
\begin{equation}\label{dstar}
D^*(s,s')=\inf\left\{\sum_{i=1}^kd_Z(s_i,t_i):\begin{array}{l}k\geq
    1\, ,\quad t_1,s_2,\ldots,s_k\in[0,A],\, s_1=s,\,t_k=s' ,\\
d_X(t_i,s_{i+1})=0\, \mbox{ for every }  i\in \{1,\ldots, k-1\}\,
  \end{array}
\right\}\, ,
\end{equation}
see \cite[Chapter~3]{burago01}.  The Brownian disk $\bd_{L,A}$ with
area~$A$ and perimeter~$L$ is the quotient set $[0,A]/\{D^*=0\}$,
endowed with the quotient metric induced from $D^*$ (which we still
denote by $D^*$ for simplicity), and considered under the law
$\mathbb{F}^A_L$. In the case $A=1$, we drop the second subscript and
write $\bd_L=\bd_{L,1}$.

\begin{rem}\label{sec:brownian-disks}
  Observe that, by usual scaling properties of Gaussian random
  variables, under the law $\mathbb{F}^A_L$, the scaled pair
  $((\lambda^{1/2}X_{s/\lambda},\, 0\leq s\leq \lambda A
  ),(\lambda^{1/4}Z_{s/\lambda},\, 0\leq s\leq \lambda A))$ has law
  $\mathbb{F}^{\lambda A}_{\lambda^{1/2} L}$, from which we deduce that
  the random metric space $\lambda^{-1/4}
  \bd_{\lambda^{1/2}\,L,\lambda A}$ has the same distribution as
  $\bd_{L,A}$.
\end{rem}

The reason why we say that $\bd_{L,A}$ has ``area'' $A$ is that it
naturally comes with a non-negative measure of total mass $A$, which
is the image of the Lebesgue measure on $[0,A]$ by the canonical
projection $\mathbf{p}:[0,A]\to \bd_{L,A}$. It will be justified later
that $\bd_{L,A}$ is a.s.\ homeomorphic to the closed unit disk, so
that the term \emph{area} makes more sense in this
context. Furthermore, the boundary $\partial \bd_{L,A}$ will be shown
to be equal to $\mathbf{p}(\{T_y:0\leq y\leq L\})$, so that it can be
endowed with a natural non-negative measure with total mass~$L$, which
is the image of the Lebesgue measure on $[0,L]$ by $y\mapsto
\mathbf{p}(T_y)$. This justifies the term ``perimeter''.

\section{The Schaeffer bijection and two variants}\label{sec:scha-biject-first}

This work strongly relies on powerful encodings of discrete maps by trees and related objects. In this section we present the encodings we will need: the original Cori--Vauquelin--Schaeffer bijection \cite{CoVa, schaeffer98}, a variant for so-called \emph{slices}~\cite{legall11} and a variant for plane quadrangulations with a boundary (particular case of~\cite{BdFGmobiles}). We only give the constructions from the encoding objects to the considered maps and refer the reader to the aforementioned works for converse constructions and proofs.

\subsection{The original Cori--Vauquelin--Schaeffer
  bijection}\label{sec:orig-cori-vauq}

Let $(\bt,\ell)$ be a \emph{well-labeled tree} with $n$ edges. Recall that
this means that $\bt$ is a rooted plane tree with $n$ edges, and
$\ell:\sV(\bt)\to \Z$ is a labeling function such that
$\ell(u)-\ell(v)\in \{-1,0,1\}$ whenever $u$ and $v$ are neighboring
vertices in $\bt$. It is usual to ``normalize'' $\ell$ in such a way
	that the root vertex of $\bt$ gets label $0$, but we will also
consider different conventions: in fact, all our discussion really
deals with the function $\ell$ up to addition of a constant. For
simplicity, in the following, we let
$\ell_*=\min\{\ell(v):v\in \sV(\bt)\}-1$. 
\begin{note}
  Throughout this paper, whenever a function~$f$ is defined at a
  vertex~$v$, we extend its definition to any corner~$c$ incident
  to~$v$ by setting $f(c)=f(v)$. In particular, the label $\ell(c)$ of
  a corner is understood as the label of the incident vertex. 
\end{note}

Let $c_0$, $c_1$, \ldots, $c_{2n-1}$ be the sequence of corners of $\bt$ in
contour order, starting from the root corner. We extend the list of
corners by periodicity, setting $c_{2n+i}=c_i$ for every $i\geq 0$,
and adding one corner $c_\infty$ incident to a vertex $v_*$ not
belonging to $\bt$, with label
$\ell(c_\infty)=\ell(v_*)=\ell_*$. Once this is done, we define the \emph{successor} functions by setting
$$s(i)=\inf\{j>i:\ell(c_j)=\ell(c_i)-1\}\in \Z_+\cup\{\infty\}\, ,\qquad i\in
\{0,1,\ldots,2n-1\}\, ,$$
and $s(c_i)=c_{s(i)}$.
The Cori--Vauquelin--Schaeffer construction consists in linking $c_i$ with
$s(c_i)$ by an arc, in a non-crossing fashion, for every $i\in
\{0,1,\ldots,2n-1\}$. The embedded graph~$\bq$ with vertex set
$\sV(\bt)\cup\{v_*\}$ and edge set the set of arcs (excluding the edges
of $\bt$) is then a quadrangulation, which is rooted according to some
convention (we omit details here as this point is not important for
our purposes), and is naturally pointed at~$v_*$. Moreover, the
labels on $\sV(\bq)$ inherited from those on~$\bt$ (and still denoted by $\ell$) are exactly
the relative distances to~$v_*$ in~$\bq$:
$$d_\bq(v,v_*)=\ell(v)-\ell_*\, ,\qquad v\in \sV(\bq).$$
(This entirely determines~$\ell$ as soon as the value $\ell(v_0)$
is known for some specific~$v_0$, but recall that in general we do not
want to fix the normalization of $\ell$.) See Figure~\ref{fig:schaeffer1} for an example of the construction. 

For every corner $c$ of $\bt$, there is an associated path in $\bq$
that follows the arcs between the consecutive successors $c$, $s(c)$,
$s(s(c))$, \ldots, $c_\infty$. This path is a geodesic path between
the vertex incident to $c$ and $v_*$, it is called the {\em maximal
  geodesic} from $c$ to $v_*$, it can be seen as the geodesic path to
$v_*$, with first step the arc from $c$ to $s(c)$, and that turns as
much as possible to the left.  

Following these paths provides a very
useful upper-bound for distances in $\bq$. Let us denote by $v_i$ the
vertex incident to the corner $c_i$, and let $\ell(i)=\ell(v_i)$ to
simplify notation.  Let $\check{\ell}(i,j)$ is the minimal value of
$\ell(r)$ for $r$ between $i$ and $j$ in cyclic order modulo
$2n$, that is
$$\check{\ell}(i,j)=\left\{\begin{array}{ll}
\min\{\ell(r),r\in[i,j]\} & \mbox{ if }i\leq j\\
\min\{\ell(r),r\in [i,2n]\cup [0,j]\} & \mbox{ otherwise. }
  \end{array}\right.
  $$
Then it holds that 
\begin{equation}
  \label{eq:20}
  d_\bq(v_i,v_j)\leq \ell(i)+\ell(j)-2\max\{\check{\ell}(i, j),\check{\ell}(
j, i)\}+2\, .
\end{equation}
The interpretation of this is as follows.  Consider
the maximal geodesics from the corners $c_{i}$ and $c_{j}$ to
$v_*$. These two geodesics coalesce at a first corner $c_k$, and the
upper bound is given by the length of the concatenation of the
geodesic from $c_i$ to $c_k$ with the segment of the geodesic from
$c_k$ to $c_j$.  This path will be called the {\em maximal wedge path}
from $c_i$ to $c_j$.

\begin{figure}[htb!]
  \centering
  \includegraphics{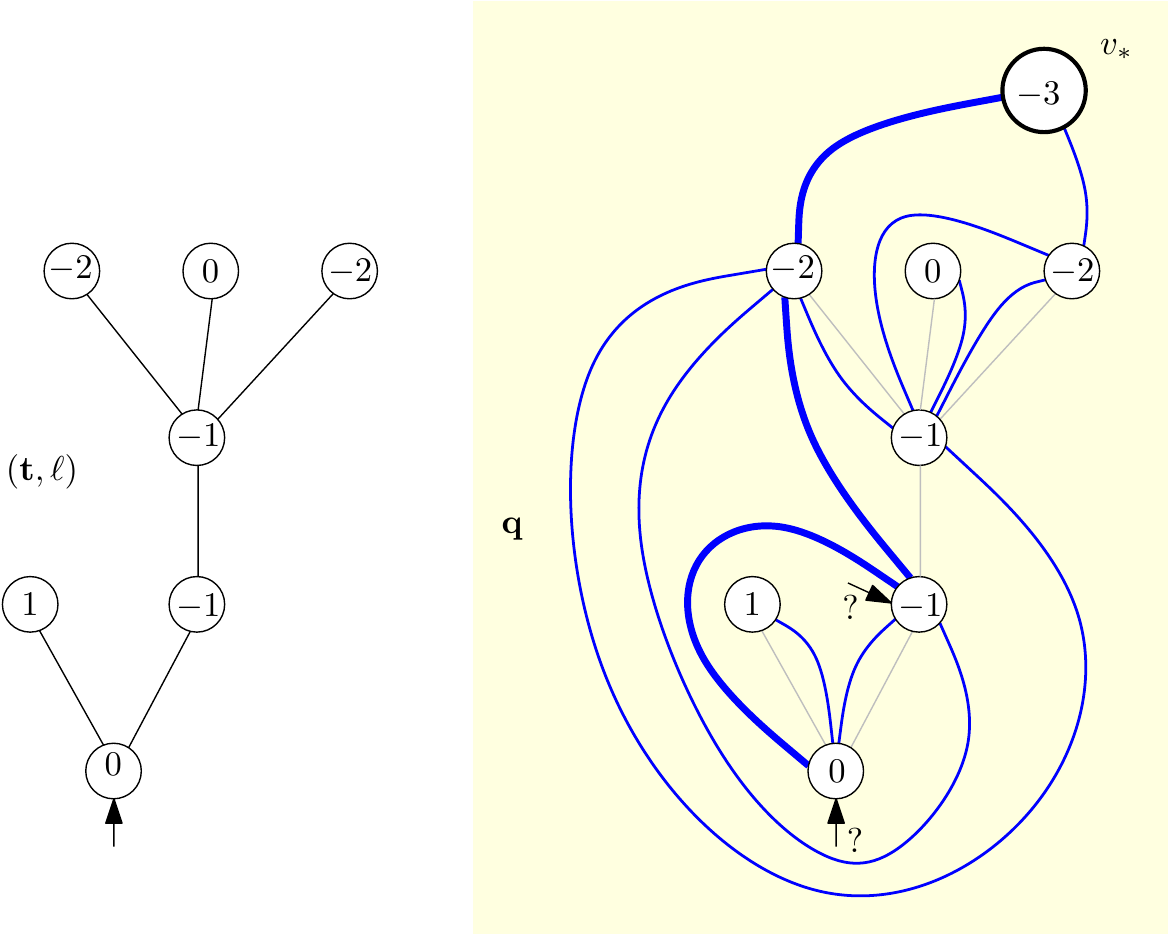}
  \caption{The Cori--Vauquelin--Schaeffer bijection. There are two possible
    rootings of~$\bq$: they are indicated with question marks. The
    maximal geodesic from the corner $c_0$ to $v_*$ has been
    magnified. }
  \label{fig:schaeffer1}
\end{figure}

\subsection{Slices}\label{sec:slices}

We now follow~\cite{legall11} and describe a modification of the
previous construction that, roughly speaking, cuts open the maximal
geodesic of~$\bq$ from~$c_0$ to~$v_*$.  See
Figure~\ref{fig:schaeffer2} for an example, and compare with
Figure~\ref{fig:schaeffer1}.

\begin{figure}[htb!]
  \centering
 \includegraphics{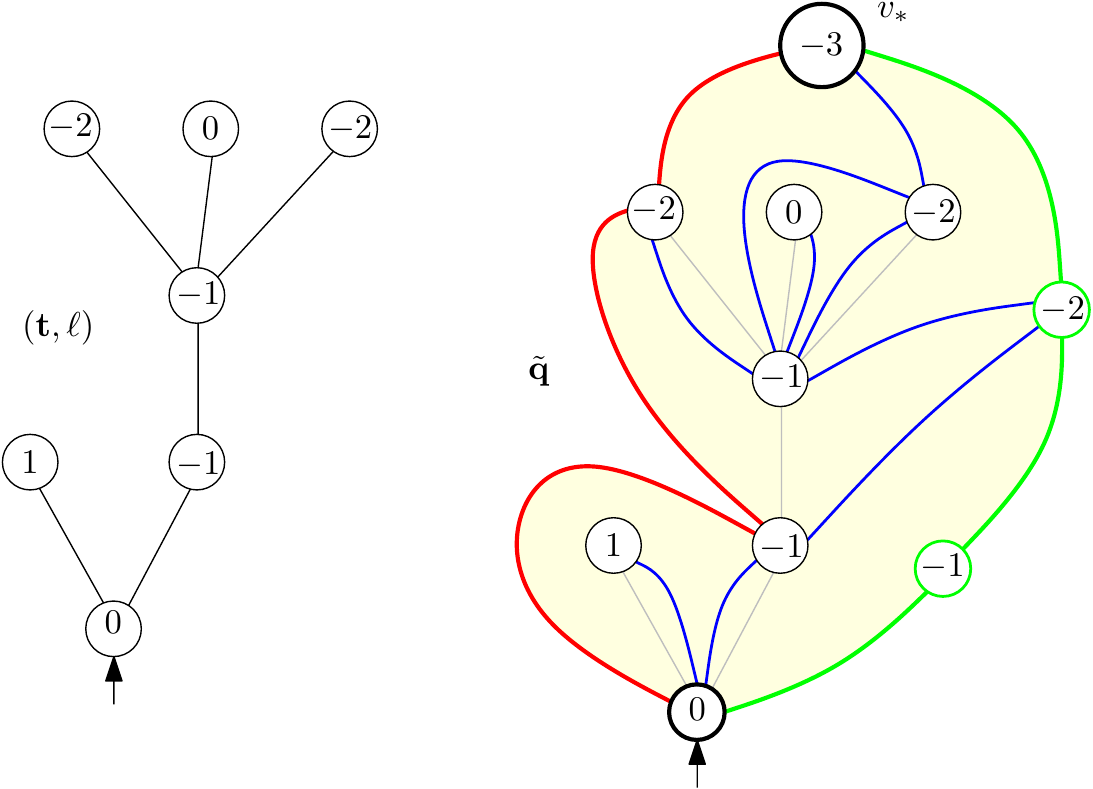} 
  \caption{ A map with geodesic boundary is associated with a
    well-labeled tree via the modified Schaeffer bijection. The
    maximal geodesic is represented in red on the left and the shuttle is
    the green chain on the right.}
  \label{fig:schaeffer2}
\end{figure}

Rather than appending to $\bt$ a single
corner $c_\infty$ incident to a vertex $v_*$, we add a sequence
of corners $c'_1$, $c'_2$, \ldots, $c'_{\ell(c_0)-\ell_*-1}$, $c'_{\ell(c_0)-\ell_*}=c_\infty$, and set labels $\ell(c'_i)=\ell(c_0)-i$, so
in particular this is consistent with the label we already set for $c_\infty$. Also,
instead of extending the sequence $c_0$, $c_1$, \ldots, $c_{2n-1}$ by
periodicity, we add an extra corner $c_{2n}$ to the right of $c_0$ and we let $c_{2n+i}=c'_i$ for $i\in \{1,2,\ldots,
\ell(c_0)-\ell_*\}$. The definition of the successor
\begin{equation}\label{eq:succ}
s(i)=\inf\{j>i:\ell(c_j)=\ell(c_i)-1\}\, ,\qquad s(c_i)=c_{s(i)}
\end{equation}
then makes sense for $i\in \{0,1,\ldots,2n+\ell(c_0)-\ell_*-1\}$, and we
can draw the arcs from~$c_i$ to~$s(c_i)$ for every $i\in
\{0,1,\ldots,2n+\ell(c_0)-\ell_*-1\}$. In particular, note that the arcs
link $c_{2n}$ with $c'_1$, $c'_2$, \ldots, $c'_{\ell(c_0)-\ell_*-1}$, $c_\infty$
into a chain, which we call \emph{shuttle}, and to which are connected the
arcs $c_i\to s(c_i)$ with $i\leq 2n-1$ and $s(i)>2n-1$. Let~$\tilde{\bq}$ be
the map obtained by this construction. It is called the {\em slice}
coded by $(\bt,\ell)$. 

This map contains two distinguished geodesic chains, which are, on the
one hand, the maximal geodesic from $c_0$ to $v_*$ made of arcs
between consecutive successors $c_0$, $s(c_0)$, $s(s(c_0))$, \ldots,
$c_\infty$ and, on the other hand, the shuttle linking~$c_{2n}$,
$c'_1$, $c'_2$, \ldots, $c'_{\ell(c_0)-\ell_*-1}$, $c_\infty$. Note
that both chains indeed have the same length (number of edges), equal
to $\ell(c_0)-\ell_*$. In
particular, we have
$d_{\tilde{\bq}}(c_0,c_\infty)=d_\bq(c_0,c_\infty)=\ell(c_0)-\ell_*$,
where~$\bq$ is the quadrangulation from the previous section,
constructed from the same well-labeled tree $(\bt,\ell)$. These two
chains are incident to a face of~$\tilde{\bq}$ of degree
$2d_\bq(c_0,c_\infty)$, and all other faces have degree~$4$. Observe
that the maximal geodesic and the shuttle only intersect at the root
vertex of the tree and~$v_*$; as a result, the boundary of the degree
$2d_\bq(c_0,c_\infty)$-face is a simple curve.

Finally, the quadrangulation~$\bq$ can then be obtained
from~$\tilde{\bq}$ by identifying one by one the edges of the maximal
geodesic with the edges of the shuttle, in the same order. More
precisely, we note that there is a natural projection $p$ from
$E(\tilde{\bq})$ to $E(\bq)$ defined by $p(e)=e$ for every edge $e$
that is not an edge of the shuttle, and $p(e'_i)=e_i$ if $e_i$ is the
$i$-th edge on the maximal geodesic, and $e'_i$ is the $i$-th edge of
the shuttle, starting from $c_0$. In particular, $p^{-1}(e)$ contains
two edges of $\tilde{\bq}$ if and only if $e$ is a vertex of the
maximal geodesic of $\bq$. The projection $p$ induces also a
projection, still denoted by $p$, from $V(\tilde{\bq})$ onto $V(\bq)$
such that, if $u$, $v$ are the extremities of $e$, then $p(u)$, $p(v)$
are the extremities of $p(e)$. For this reason, any path in
$\tilde{\bq}$ projects into a path in $\bq$ via $p$, and the graph
distances satisfy the inequality
$$d_{\bq}(p(u),p(v))\leq
d_{\tilde{\bq}}(u,v)\, ,\qquad u,v\in V(\tilde{\bq})\, .$$

Using the same idea as in the preceding section, we obtain another
useful bound for distances in $\tilde{\bq}$, as follows. Again, let
$v_i$ be the vertex incident to the corner $c_i$, and
$\ell(i)=\ell(v_i)$. Then 
\begin{equation}
  \label{eq:23}
  d_{\tilde{\bq}}(v_i,v_j)\leq \ell(i)+\ell(j)-2\check{\ell}(i\wedge
j,i\vee j)+2\, ,
\end{equation}
where $\check{\ell}(i,j)$ is again defined as the
minimal value of $\ell$ between $i$ and $j$. Again, this upper bound
corresponds to the length of a concatenation of maximal geodesics from
$c_i$, $c_j$ to $v_*$ up to the point where they coalesce. In words, the
difference is that by taking systematically $\check{\ell}(i\wedge
j,i\vee j)$ in the definition rather than the maximum of
$\{\check{\ell}(i,j),\check{\ell}(j,i)\}$, we do not allow to ``jump'' from
the shuttle to the maximal geodesic boundary (or vice-versa), which
would result in a path present in~$\bq$ but not in~$\tilde{\bq}$.

\subsection{Plane quadrangulations with a boundary}\label{sec:bij_forests}

We now present the variant for plane quadrangulation with a boundary, which is a particular case of the Bouttier--Di~Francesco--Guitter bijection~\cite{BdFGmobiles}. We rather use the presentation of~\cite{bettinelli14gbs}, better fitted to our situation.

The encoding object of a plane quadrangulation with a boundary having~$n$ internal faces and perimeter~$2l$ is a forest $\bff=(\bt_1, \dots, \bt_l)$ of~$l$ trees with~$n$ edges in total, together with a labeling function $\ell:\sV(\bff)=\bigsqcup_i \sV(\bt_i)\to\Z$ satisfying the following:
\begin{itemize}
	\item for $1\le i \le l$, the tree $\bt_i$ equipped with the restriction of~$\ell$ to $V(\bt_i)$ is a well-labeled tree;
	\item for $1\le i \le l$, we have $\ell(\rho_{i+1})\ge
          \ell(\rho_{i})-1$, where~$\rho_i$ denotes the root vertex
          of~$\bt_i$ and setting $\ell(\rho_{l+1})=\ell(\rho_1)$ by
          convention.
\end{itemize}
Note that the condition on the labels of the root vertices is different from the condition on the labels of neighboring vertices of a given tree. The reader familiar with the Bouttier--Di~Francesco--Guitter bijection may recognize the label condition for faces of even degree more than~$4$. We will come back to this during Section~\ref{sec:univ}.

Here and later, it will be convenient to normalize~$\ell$ by asking
that $\ell(\rho_1)=0$. As before, we define $\ell_*=\min\{\ell(v):v\in
\sV(\bff)\}-1$. We identify~$\bff$ with the map obtained by adding~$l$
edges linking the roots $\rho_1$, $\rho_2$, \ldots, $\rho_l$ of the
successive trees in a cycle. This map has two faces, one of degree $2n
+l$ (the bounded one on Figure~\ref{fig:schaeffer3}) and one of
degree~$l$ (the unbounded one on Figure~\ref{fig:schaeffer3}). We then
follow a procedure similar to that of
Section~\ref{sec:orig-cori-vauq}. We let $c_0$, $c_1$, \ldots,
$c_{2n+l-1}$ be the sequence of corners of the face of degree $2n+l$
in contour order, starting from the root corner of~$\bt_1$. We extend
this list by periodicity and add one corner $c_\infty$ incident to a
vertex~$v_*$ lying inside the face of degree $2n+l$, with label
$\ell(c_\infty)=\ell(v_*)=\ell_*$. We define the successor functions
by~\eqref{eq:succ} and draw an arc from~$c_i$ to~$s(c_i)$ for every
$i\in\{0,1,\ldots,2n+l-1\}$, in such a way that this arc does not
cross the edges of~$\bff$, or other arcs. 

The embedded graph~$\bq$ with vertex set $\sV(\bff)\cup\{v_*\}$ and
edge set given by the added arcs is a plane quadrangulation with a
boundary, whose external face is the degree-$2l$ face corresponding to
the face of degree~$l$. It is rooted at the corner of the unbounded
face that is incident to the root vertex of~$\bt_1$, and it is
naturally pointed at~$v_*$. See Figure~\ref{fig:schaeffer3}.

\begin{figure}[htb!]
  \centering
  \includegraphics[width=.98\linewidth]{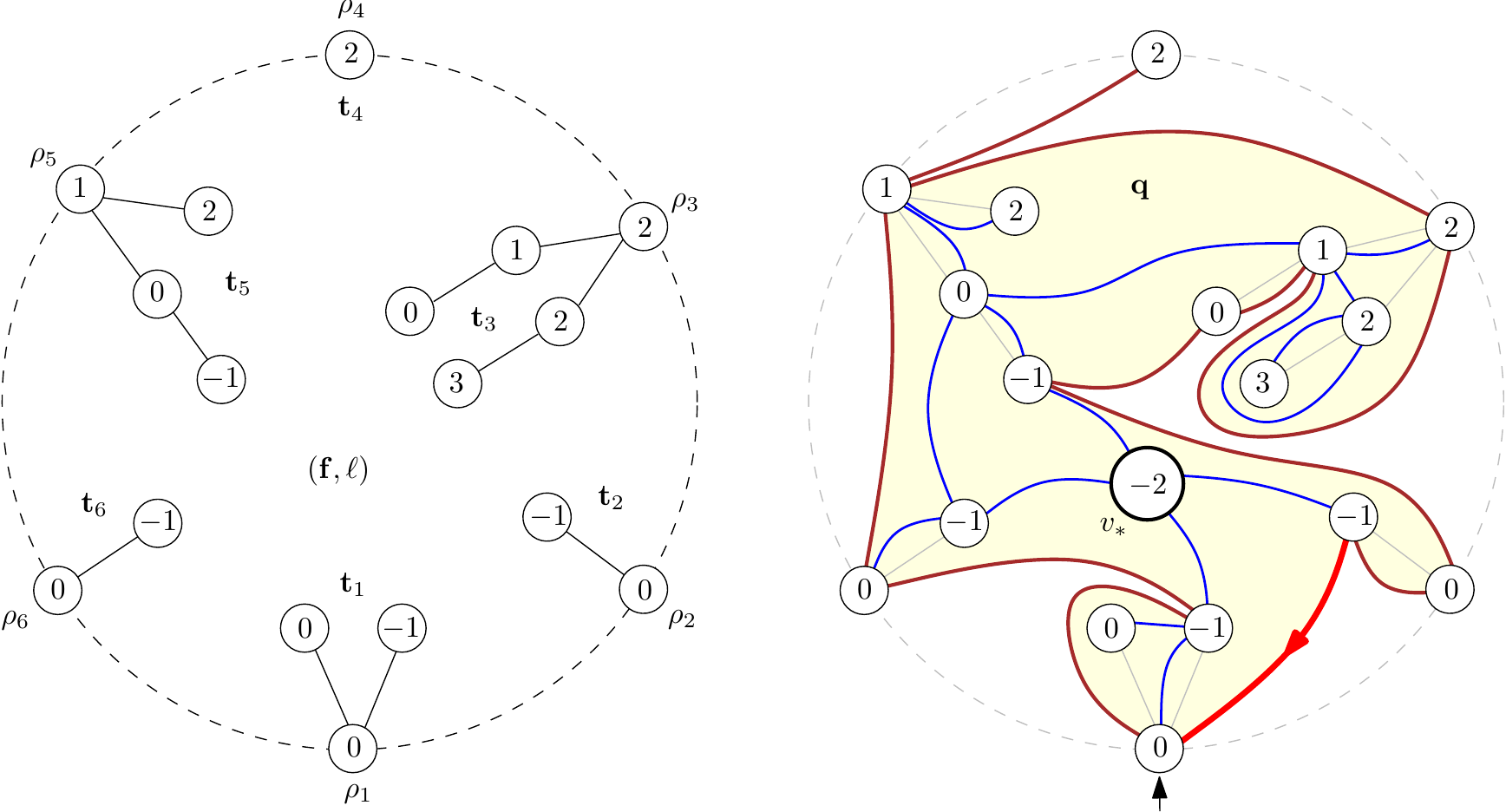}
  \caption{The bijection for a plane quadrangulation with a
    boundary. The~$l$ edges we added to~$\bff$ are represented by dotted lines and the root edge of~$\bq$ is represented in red. Note that we used for~$\ell$ the normalization given by $\ell(\rho_1)=0$.}
  \label{fig:schaeffer3}
\end{figure}

The above mapping is a bijection between previously described labeled
forests and the set of pointed plane quadrangulations $(\bq,v_*)$ with
a boundary having~$n$ internal faces and perimeter~$2l$ that further
satisfy the property that $d_\bq(e_*^+,v_*)=d_\bq(e_*^-,v_*)+1$,
where~$e_*$ denotes the root edge of~$\bq$, that is, the oriented edge
incident to the root face that directly precedes the root corner in
the contour order (see Figure~\ref{fig:schaeffer3}). In words, the
pointed quadrangulations that are in the image of the above mapping
are those whose root edge \emph{points away from} the distinguished
vertex~$v_*$.

The requirement that the root edge is directed away from the
distinguished vertex is not a serious issue, as we can dispose of this
constraint simply by re-rooting along the boundary:

\begin{lmm}
  \label{awayfrom}
  Let $(Q,v_*)$ be uniformly distributed in the set $\bQns^{\bullet,+}$ of rooted and pointed
  quadrangulations such that $Q\in \bQns$ and such that the root edge $e_*$
  points away from $v_*$. Let $c'$ be a uniformly chosen random corner
  incident to the root face of $Q$, and let $Q'$ be the map $Q$
  re-rooted at~$c'$. Then $Q'$ is a uniform random element of $\bQns$.   
%
\end{lmm}

\begin{proof}
The probability that $Q'$ is a given rooted map $\bq'$ is 
equal 
$$\P(Q'=\bq')=\frac{1}{2l}\sum_{v\in
  \sV(\bq')}\sum^+_{c}\P((Q,v_*)=(\bq,v))=\frac{1}{2l}\sum_{v\in
  \sV(\bq')}\sum^+_{c}\frac{1}{|\bQns^{\bullet,+}|}\, ,$$
where the factor $1/2l$ is the probability that $c'$ is chosen
to be the
root corner of $\bq'$, the symbol $\sum^+_c$ stands for the sum over
all corners incident to the root face of $\bq'$ that point away from
$v$, and $\bq$ is the map $\bq'$ re-rooted at the corner $c$.  

Now fix the vertex $v\in \sV(\bq')$. Due to the bipartite nature
of~$\bq'$, among the~$2l$ oriented edges incident to the root face,
$l$ are pointing away from~$v$, and~$l$ are pointing
toward~$v$. Indeed, let $\tilde c_0$ be the root corner of~$Q$, and
$\tilde c_1$, $\tilde c_2$, \ldots, $\tilde c_{2l-1}$, $\tilde
c_{2l}=\tilde c_0$ be the corners incident to the root face in cyclic
order. The sequence $(d_{\bq'}(\tilde c_i,v_*),\, 0\leq i\leq 2l)$
takes integer values, varies by~$\pm 1$ at every step as~$\bq'$ is
bipartite, and takes the same value at times~$0$ and~$2l$. This means
that~$l$ of its increments are equal to~$+1$ and~$l$ are equal
to~$-1$, respectively corresponding to edges that point away
from~$v$ and toward~$v$.

Therefore, the sum $\sum^+_c$ contains exactly $l$ elements. Noting
that every map in~$\bQns$ has $n+l+1$ vertices, by the Euler
characteristic formula, this gives
$$\P(Q'=\bq')=\frac{n+l+1}{2\,|\bQns^{\bullet,+}|}\, ,$$
which depends only on $n$, $l$ and not on the particular choice of~$\bq'$. 
\end{proof}

\section{Scaling limit of slices}\label{sec:known-scaling-limit}

In this section, we elaborate on Proposition~3.3 and Proposition~9.2 in
\cite{legall11}, by showing that uniform random slices converge after
rescaling to a limiting metric space, which can be called the \emph{Brownian
map with a geodesic boundary}. Such a property was indeed shown in
\cite{legall11}, but with a description of the limit that is different from
the one we will need. 

\subsection{Subsequential convergence}\label{sec:subs-conv}

Let $(T_n,\ell_n)$ be a random variable that is uniformly distributed
over the set of well-labeled trees with $n$ edges. With this random
variable, we can associate two pointed and rooted random maps
$(Q_n,v_*)$ and $(\tilde{Q}_n,v_*)$ by the constructions of Sections
\ref{sec:orig-cori-vauq} and \ref{sec:slices} respectively. We use the
same notation for the distinguished vertex $v_*$ since $Q_n$ and
$\tilde{Q}_n$ share naturally the same vertex set, except for the
extra vertices on the shuttle of $\tilde{Q}_n$.  

Let $c_0$, $c_1$, \ldots, $c_{2n-1}$, $c_{2n}=c_0$ be the sequence of corners of
$T_n$ starting from the root corner, and let $v_i$ be the vertex
incident to~$c_i$ in~$T_n$. We let $C_n(i)$ be the distance in $T_n$
between the vertices $v_0$ and $v_i$, so that $C_n(i)$ can be seen as
the height of $v_i$ in the tree $T_n$ rooted at $c_0$. The process
$(C_n(i),0\leq i\leq 2n)$, extended to a continuous random function on
$[0,2n]$ by linear interpolation between integer values, is called the
\emph{contour process} of~$T_n$. Similarly, we let $\ell_n(i)=\ell_n(v_i)$ and
call the process $(\ell_n(i),0\leq i\leq 2n)$, which we also extend to
$[0,2n]$ in a similar fashion, the \emph{label process} of
$(T_n,\ell_n)$. 

For $0\leq i,j\leq 2n$, let $D_n(i,j)=d_{Q_n}(v_i,v_j)$ and
$\tilde{D}_n(i,j)=d_{\tilde{Q}_n}(i,j)$. 
 We extend $D_n$, $\tilde{D}_n$ to continuous functions on
$[0,2n]^2$ by ``bilinear interpolation,'' writing
$\{s\}=s-\lfloor s\rfloor$ for the fractional part of $s$ and then setting
\begin{equation}\label{eq:6}\begin{split}
    D_n(s,t)&=(1-\{s\}) (1-\{t\}) D_n(\lfloor s\rfloor,\lfloor
    t\rfloor)+\{s\}(1-\{t\}) D_n(\lfloor s\rfloor+1,\lfloor t\rfloor) \\
    &\qquad\qquad + (1-\{s\})\{t\}D_n(\lfloor s\rfloor,\lfloor
    t\rfloor+1)+\{s\}\{t\} D_n(\lfloor s\rfloor+1,\lfloor t\rfloor+1),
\end{split}\end{equation}
and similarly for $\tilde{D}_n$. 
We define the renormalized versions of $C_n$, $\ell_n$, $D_n$ and
$\tilde{D_n}$ by
$$
C_{(n)}(s)=\frac{C_n(2ns)}{\sqrt{2n}}\, ,\quad
\ell_{(n)}(s)=\left(\frac{9}{8n}\right)^{1/4}\ell_n(2ns)\, ,
$$
and 
$$
  D_{(n)}(s,t)=\left(\frac{9}{8n}\right)^{1/4}D_n\big(
  2ns, 2nt\big)\, ,\qquad \tilde{D}_{(n)}(s,t)=\left(\frac{9}{8n}\right)^{1/4}\tilde{D}_n\big(
  2ns, 2nt\big)
$$
for every $s$, $t\in [0,1]$. 

From \cite[Proposition~3.1]{legall11}, it holds that up to extraction,
one has the joint convergence
\begin{equation}
  \label{eq:21}
  (C_{(n)},\ell_{(n)},D_{(n)},\tilde{D}_{(n)})\build\longrightarrow_{n\to\infty}^{(d)}(\ee,Z,D,\tilde{D})\,
,
\end{equation}
where $\ee$ is the normalized Brownian excursion, $Z$ is the head of
the snake driven by $\ee$ (which is defined as the process $\Zo$
around \eqref{eq:24}, with $\ee$ in place of $X$) and $D$, $\tilde{D}$
are two random pseudo-metrics on $[0,1]$ such that $D\leq
\tilde{D}$. In the rest of this section, we are going to fix one
extraction along which this convergence holds, and always assume that
the values of $n$ that we consider belong to this particular
extraction. Moreover, by a use of the Skorokhod representation
theorem, we may and will assume that the convergence holds in fact in
the a.s.\ sense.

For $s$, $s'\in[0,1]$, define
$d_\ee(s,s')=\ee_s+\ee_{s'}-2\inf_{s\wedge s'\leq u\leq s\vee
  s'}\ee_u$, $d_Z(s,s')$ as in formula~\eqref{dZ} (with $A=1$), and let
$$\tilde{d}_Z(s,s')=Z_s+Z_{s'}-2\underline{Z}_{s\wedge s',s\vee s'}\,
,$$
so that clearly one has $d_Z\leq \tilde{d}_Z$.
The quotient space $S=[0,1]/\{D=0\}$ endowed with the
distance induced by $D$ (and still denoted by $D$), is the so-called {\em
  Brownian map}. 
Likewise, we set $\tilde{S}=[0,1]/\{\tilde{D}=0\}$ and endow it with
the induced distance still denoted by $\tilde{D}$.  We let
$\bp:[0,1]\to S$, $\tilde{\bp}:[0,1]\to \tilde{S}$ denote the
canonical projections, which are continuous since $D$, $\tilde{D}$ are
continuous functions on $[0,1]^2$. Note that, since $D\leq \tilde{D}$,
there exists a unique continuous (even $1$-Lipschitz) projection
$\pi:\tilde{S}\to S$ such that $\bp=\pi\circ\tilde{\bp}$.
 

 The main result of
\cite{legall11,miermont11} states that a.s., for every
$s$, $t\in [0,1]$, $D(s,t)$ is given by the
explicit formula
\begin{equation}\label{DDstar}
D(s,t)= \inf\left\{\sum_{j=1}^kd_Z(s_j,t_j):\begin{array}{l}k\geq
    1\, ,\quad t_1,s_2,t_2,\ldots,s_k\in[0,1],\, s_1=s,\,t_k=t ,\\
d_\ee(t_j,s_{j+1})=0\, \mbox{ for every }  j\in \{1,\ldots, k-1\}\,
  \end{array}
\right\}\, .
\end{equation}
The main goal of this section is to show that the following analog
formula holds for $\tilde{D}$. First, we recall from \cite{legall11}
that $\tilde{D}\leq \tilde{d}_Z$ and that $\{d_\ee=0\}\subseteq
\{\tilde{D}=0\}$. Note that the first of these two properties results
from a simple passage to the limit in the bound \eqref{eq:23}.  We let
$\tilde{D}^*$ be the largest pseudo-metric on $[0,1]$ such that these
two facts are verified, that is,
$$\tilde{D}^*(s,t)=\inf\left\{\sum_{j=1}^k\tilde{d}_Z(s_j,t_j):\begin{array}{l}k\geq
    1\, ,\quad t_1,s_2,t_2,\ldots,s_k\in[0,1],\, s_1=s,\,t_k=t ,\\
d_\ee(t_j,s_{j+1})=0\, \mbox{ for every }  j\in \{1,\ldots, k-1\}\,
  \end{array}
\right\}\, .$$ In particular, $\tilde{D}\leq \tilde{D}^*$.  We will
show that $\tilde{D}=\tilde{D}^*$ a.s., and in particular, the
convergence in \eqref{eq:21} holds without having to extract a
subsequence. Results by Le Gall \cite[Propositions~3.3
and~9.2]{legall11} provide yet another formula for $\tilde{D}$, which
is expressed in terms of cutting the space $(S,D)$ along a certain
distinguished geodesic. However, it is not clear that this formula is
equivalent to $\tilde{D}=\tilde{D}^*$.

\begin{thm}
  \label{sec:scaling-limit-slices-3}
Almost surely, it holds that for every $s$, $t\in [0,1]$,
$\tilde{D}(s,t)=\tilde{D}^*(s,t)$. 

Moreover, we have for every $x$, $y\in \tilde{S}$,
$$\tilde{D}(x,y)=\inf\left\{\mathrm{length}_D 
    (\pi\circ\gamma):\begin{array}{c}
\gamma:[0,1]\to \tilde{S} \mbox{ continuous}\\
\gamma(0)=x,\gamma(1)=y 
  \end{array}
  \right\}
$$
where $\tilde{S}$ is endowed with the quotient topology of
$[0,1]/\{\tilde{D}=0\}$.
\end{thm}

Here, the length function
is defined as follows. If $(M,d)$ is a metric space (or a pseudo-metric space), and
$\gamma:[0,1]\to M$ is a continuous path, we let
$$\mathrm{length}_d(\gamma)=\sup\sum_{i=1}^kd\big(\gamma(r_{i-1}),\gamma(r_i)\big)\, ,
$$
where the supremum is taken over all partitions
$0=r_0<r_1<\ldots<r_{k-1}<r_k=1$ of $[0,1]$. 

\subsection{Basic properties of the limit spaces}\label{sec:basic-prop-limit}

We will need some more properties of
the distances $D$, $\tilde{D}$, $\tilde{D}^*$. An important fact that we
will need is the following identification of the sets $\{D=0\}$, $\{\tilde{D}=0\}$, $\{\tilde{D}^*=0\}$, which is a reformulation of
\cite[Theorem 4.2]{legall06}, \cite[Lemma 3.2]{lgp} and
\cite[Proposition 3.1]{legall11} in our setting. Point~\eqref{lemeniii} comes from \cite[Proposition~2.5]{legweill} and \cite[Proposition 3.2]{legall11}.

\newcounter{lemen}
\setcounter{lemen}{0}
\renewcommand{\thelemen}{\roman{lemen}}
\begin{lmm}
  \label{sec:scaling-limit-slices-1}
\noindent{\em (i)}\refstepcounter{lemen}\label{lemeni} Almost surely, for every $s$, $t\in [0,1]$ such that $s\neq t$, it holds
that $D(s,t)=0$ if and only if either $d_\ee(s,t)=0$ or $d_Z(s,t)=0$,
these two cases being mutually exclusive, with the only exception of
$\{s,t\}=\{0,1\}$. 

\medskip
\noindent{\em (ii)}\refstepcounter{lemen}\label{lemenii} Likewise, almost surely, for every $s$, $t\in [0,1]$ such that $s\neq t$,
it holds that $\tilde{D}(s,t)=0$ if and only if either $d_\ee(s,t)=0$
or $\tilde{d}_Z(s,t)=0$, and these two cases are mutually
exclusive. 

\medskip
\noindent{\em (iii)}\refstepcounter{lemen}\label{lemeniii} There is only one time $s_*\in [0,1]$ such that $Z_{s_*}=\inf_{[0,1]}Z$. Moreover, $D(0,s_*)=\tilde D(0,s_*)=-Z_{s_*}$.
\end{lmm}

This implies that the equivalence relations $\{\tilde{D}=0\}$ and
$\{\tilde{D}^*=0\}$ coincide, since $\tilde{D}\leq \tilde{D}^*\leq
\tilde{d}_Z$ and $\{d_\ee=0\}\subseteq \{\tilde{D}^*=0\}$ by
definition. In particular, we see that
$\tilde{S}=[0,1]/\{\tilde{D}^*=0\}$ endowed with the
induced metric $\tilde{D}^*$ is homeomorphic to
$(\tilde{S},\tilde{D})$. 

Note that \eqref{lemeni} and \eqref{lemenii} in the last statement are very closely related. One sees that
the points $s<t$ such that $D(s,t)=0$ but $\tilde{D}(s,t)\neq
0$ are exactly the points such that
$$Z_s=\inf_{[0,s]}Z=\inf_{[t,1]}Z=Z_t\qquad\text{ and }(s,t)\neq (0,1)\, .$$
Indeed, the previous equalities imply that $d_Z(s,t)=0$ and
that $\inf_{[s,t]}Z<Z_s$, by~\eqref{lemeniii}, so that $\tilde d_Z(s,t)>0$; furthermore, $d_Z(s,t)=0$ and $(s,t)\neq (0,1)$ imply by~\eqref{lemeni} that $d_\ee(s,t)>0$. 
This entails that, for $s<t$ of this form,
one has that $x=\bp(s)=\bp(t)\in S$ has two preimages
$\pi^{-1}(x)=\{\tilde{\bp}(s),\tilde{\bp}(t)\}\in \tilde{S}$, while
for any other point $x\in S$, $\pi^{-1}(x)$ is a singleton. 

More precisely, let $\Delta=D(0,s_*)=\tilde D(0,s_*)$ and, for $r\in [0,\Delta]$, let
$$\Gamma_0(r)=\inf\{s\geq 0:Z_s=-r\}\, ,\qquad \mbox{ and }\qquad 
\Gamma_1(r)=\sup\{s\geq 0:Z_s=-r\}\, .$$ We also let
$\gamma_i(r)=\tilde{\bp}(\Gamma_i(r))$ for $i\in \{0,1\}$ and $r\in
[0,\Delta]$, and $\gamma(r)=\bp(\Gamma_0(r))=\bp(\Gamma_1(r))$. We let
$\overset{\circ}{\gamma}_0=\gamma_0((0,\Delta))$ and we define 
$\overset{\circ}{\gamma}_1$ and
$\overset{\circ}{\gamma}$ in a similar fashion.

\begin{crl}
  \label{sec:basic-prop-limit-2}
  It holds that
  $\overset{\circ}{\gamma}_0\cap\overset{\circ}{\gamma}_1=\varnothing$.
  Moreover, the projection $\pi$ is one-to-one from
  $\tilde{S}\setminus(\overset{\circ}{\gamma}_0\cup
  \overset{\circ}{\gamma}_1)$ onto $S\setminus
  \overset{\circ}{\gamma}$, while
  $\pi^{-1}(\gamma(r))=\{\gamma_0(r),\gamma_1(r)\}$ for every $r\in
  [0,\Delta]$, and the latter is a singleton if and only if $r\in\{0,\Delta\}$.
\end{crl}

Next, we say that a metric space $(M,d)$ is a \emph{length space} if for
every $x$, $y\in M$, $d(x,y)=\inf\mathrm{length}_d(\gamma)$ where the
infimum is taken over all continuous paths $\gamma:[0,1]\to M$ with
$\gamma(0)=x$ and $\gamma(1)=y$. A path $\gamma$ for which the infimum
is attained is called a \emph{geodesic}, and a \emph{geodesic metric space} is a
length space $(M,d)$ such that every pair of points is joined by a
geodesic. A compact length space is a geodesic space by \cite[Theorem
2.5.23]{burago01}.

\begin{lmm}
  \label{sec:scaling-limit-slices-4}
The spaces $(S,D)$, $(\tilde{S},\tilde{D})$ and $(\tilde{S},\tilde{D}^*)$
are compact geodesic metric spaces. 
\end{lmm}

\begin{proof}
We only sketch the proof of this lemma. Recall that the property 
of being a compact geodesic metric space is preserved by taking 
Gromov--Hausdorff limits, by \cite[Theorem 7.5.1]{burago01}. Now, we 
use the fact that $(S,D)$, $(\tilde{S},\tilde{D})$ are Gromov--Hausdorff 
limits of the metric spaces $(V(Q_n),(9/8n)^{1/4}d_{Q_n})$ and 
$(V(\tilde{Q}_n),(9/8n)^{1/4}d_{\tilde{Q}_n})$, which in turn are at 
distance less than $(9/8n)^{1/4}$ from metric graphs obtained by 
linking any two adjacent vertices by an edge of length $(9/8n)^{1/4}$,
the latter being geodesic metric spaces. For 
$(\tilde{S},\tilde{D}^*)$, this comes from the fact that $\tilde{D}^*$
is a quotient pseudo-metric of the space 
$([0,1]/\{\tilde{d}_Z=0\},\tilde{d}_Z)$ with respect to the 
equivalence relation induced on $[0,1]/\{\tilde{d}_Z=0\}$ by 
$\{d_\ee=0\}$. Since $([0,1]/\{\tilde{d}_Z=0\},\tilde{d}_Z)$ is a 
length space (it is indeed an $\R$-tree), the quotient pseudo-metric 
$(\tilde{S},\tilde{D}^*)$ is also a length space, hence a geodesic space 
since it is compact. See the discussion after Exercise 3.1.13 in 
\cite{burago01}. 
\end{proof}

We conclude by mentioning that the mappings $r\mapsto
\gamma_0(r)$ and $r\mapsto\gamma_1(r)$ are geodesics
in $(\tilde{S},\tilde{D})$ and that $r\mapsto \gamma(r)$ is a geodesic
in $(S,D)$. This follows
easily from approximations ($\gamma_0$ is the continuum counterpart to
the maximal geodesic in Section \ref{sec:slices}, and $\gamma_1$ to
the shuttle) and is discussed in \cite{legall11}.

\subsection{Local isometries between \texorpdfstring{$\tilde{S}$}{S tilde} and
  \texorpdfstring{$S$}{S}}\label{sec:local-isom-betw}

In the following, if $(M,d)$ is a metric space or a pseudo-metric
space, and if $x\in M$, $A\subseteq M$, we let
$d(x,A)=\inf\{d(x,y):y\in A\}$. For $i\in \{0,1\}$, we also use the
shorthand $\Gamma_i$, $\gamma_i$ to designate the image sets $\{\Gamma_i(r),0\leq r\leq
\Delta\}$ and $\{\gamma_i(r),0\leq r\leq
\Delta\}$.

\begin{lmm}
  \label{sec:scaling-limit-slices-2}
  The following holds almost surely. Fix $\eps>0$, $i\in \{0,1\}$, and
  let $s$, $t\in [0,1]$ be such that $\tilde{D}(s,\Gamma_i)\wedge
  \tilde{D}(t,\Gamma_i)> \eps$ and $\tilde{D}(s,t)<\eps/2$. Then,
  it holds that $\tilde{D}(s,t)=D(s,t)=\tilde{D}^*(s,t)$. 
\end{lmm}

\begin{proof}
  Assume that $i=0$.  Let $i_n$, $j_n\in \{0,1,\ldots,2n\}$ be such that $i_n/2n\to s$ and $j_n/2n\to t$ as
  $n\to\infty$. Recall that, throughout this section, we have fixed an extraction along which~\eqref{eq:21} holds and that $n\to\infty$ is understood along this extraction. Then,
$$\left(\frac{9}{8n}\right)^{1/4}d_{Q_n}(v_{i_n},v_{j_n})\build\longrightarrow_{n\to\infty}^{}D(s,t)\,
,\qquad
\left(\frac{9}{8n}\right)^{1/4}d_{\tilde{Q}_n}(v_{i_n},v_{j_n})\build\longrightarrow_{n\to\infty}^{}\tilde{D}(s,t)\,
.$$ From the fact that $\tilde{D}(s,\Gamma_0)>\eps$ we deduce that for
every $n$ large enough, the vertex $v_{i_n}$ is at
$d_{\tilde{Q}_n}$-distance at least $(8n/9)^{1/4}\eps$ from the
maximal geodesic in $\tilde{Q}_n$. Indeed, if this were not the case,
then for infinitely many values of $n$, we could find a vertex
$v_{k_n}$ of the maximal geodesic with
$d_{\tilde{Q}_n}(v_{k_n},v_{i_n})\leq (8n/9)^{1/4}\eps$. By definition
of the maximal geodesic, it must hold that $\ell_n(k_n)=\inf
\{\ell_n(i):0\leq i\leq k_n\}$, and by passing to the limit up to
further extraction, we may assume that $k_n/2n$ converges to some $u$
such that $Z_u=\inf\{Z_s,0\leq s\leq u\}$, so that $\tilde{\bp}(u)\in
\gamma_0$ and $\tilde{D}(u,s)\leq \eps$, a contradiction with the fact
that $\tilde{D}(s,\Gamma_0)>\eps$.

Fix $\eta>0$. By~\eqref{DDstar}, there exist $s=s_1$, $t_1$, \ldots, $s_k$, $t_k=t$ such that
$d_\ee(t_m,s_{m+1})=0$ for $1\leq m\leq k-1$, and 
$$D(s,t)\geq \sum_{m=1}^k d_Z(s_m,t_m)-\eta\, .$$
Then, we can choose integers $i_n(m)$, $j_n(m)$, $1\leq m\leq k$ such that
$i_n(m)/2n\to s_m$ and $j_n(m)/2n\to t_m$ as $n\to\infty$, and we can
also require that $v_{j_n(m)}=v_{i_n(m+1)}$ for all $m\in
\{1,\ldots,k-1\}$. Indeed, this last property amounts to the fact that
$C_n(j_n(m))=C_n(i_n(m+1))$ and that $C_n$ is greater than or equal to
this common value on $[j_n(m)\wedge i_n(m+1),j_n(m)\vee i_n(m+1)]$; 
we can require this as a simple consequence of the
fact that $d_\ee(t_m,s_{m+1})=0$ and of the convergence of $C_{(n)}$
to~$\ee$.  For every $m\in \{1,\ldots,k\}$, let $g_{n,m}$ be the
maximal wedge path in $Q_n$ from $c_{i_n(m)}$ to $c_{j_n(m)}$, as
defined at the end of Section~\ref{sec:orig-cori-vauq}. The length of this path
is given by the upper-bound of~\eqref{eq:20} for $\ell=\ell_n$, $i=i_n(m)$ and
$j=j_n(m)$ and, after renormalization by $(8n/9)^{1/4}$, this length
converges to $d_Z(s_m,t_m)$. Therefore, if we let $g_n$ be the
concatenation of the paths $g_{n,1}$, $g_{n,2}$, \ldots, $g_{n,k}$, then the
length of $g_n$ is asymptotically $(8n/9)^{1/4}\sum_{1\leq m\leq
  k}d_Z(s_m,t_m)\leq (8n/9)^{1/4}(D(s,t)+\eta)$.

If $g_n$ does not intersect the maximal geodesic from the root $c_0$
to $v_*$ in $Q_n$, then $g_n$ is also a path in $\tilde{Q}_n$ (meaning that it can
be lifted via the projection $p$ from $\tilde{Q}_n$ to $Q_n$, as
defined in Section \ref{sec:slices}). In this case, this also means
that the maximal wedge paths $g_{n,m}$ are also paths in
$\tilde{Q}_n$, entailing that their lengths are given by the
upper-bounds in \eqref{eq:23}. If, for infinitely many $n$'s, $g_n$ does not intersect the maximal geodesic from the root~$c_0$ to~$v_*$ in~$Q_n$ then, by passing to the limit, we obtain
$d_Z(s_m,t_m)=\tilde{d}_Z(s_m,t_m)$. 
We immediately get
$$\tilde{D}(s,t)\leq \tilde{D}^*(s,t)\leq \sum_{m=1}^k \tilde{d}_Z(s_m,t_m)=\sum_{m=1}^k
d_Z(s_m,t_m)\leq D(s,t)+\eta\, .$$
Since $\eta$ was arbitrary we obtain $\tilde{D}(s,t)\leq
\tilde{D}^*(s,t)\leq D(s,t)$, but since $D\leq \tilde{D}$, we conclude that this must be
an equality all along. 

Suppose now that, for infinitely many $n$'s, the path $g_n$ does intersect the maximal geodesic
from $c_0$ to $v_*$ in $Q_n$. For such an~$n$ fixed, let $a$, $b$ be the minimal and maximal
integers such that $g_n(a)$, $g_n(b)$ belong to this path. Clearly, we
can modify the path $g_n$ by replacing it if necessary by the arc of
the maximal geodesic between $g_n(a)$ and $g_n(b)$ without increasing
its length. Now, the vertices $g_n(0),g_n(1),\ldots,g_n(a-1)$ are
vertices of $Q_n$ that are not in the maximal geodesic, so they lift
via the projection $p$ to a path in $\tilde{Q}_n$, with same
length. The edge between $g_n(a-1)$ and $g_n(a)$ also lifts into an
edge of $\tilde{Q}_n$, and it arrives at a point $g'_n(a)\in
p^{-1}(g_n(a))$ which is either on the maximal geodesic or on the
shuttle of $\tilde{Q}_n$. However, the first case is impossible for large~$n$'s, since
the maximal geodesic is at $d_{\tilde{Q}_n}$-distance at least
$(8n/9)^{1/4}\eps$ from $v_{i_n}=g_n(0)$ and $d_{\tilde{Q}_n}(g_n(0),g'_n(a))\le d_{\tilde{Q}_n}(v_{i_n},v_{j_n})\le (8n/9)^{1/4}\eps/2$. The same argument applies to
the path $g_n(b)$, $g_n(b+1)$, \ldots, $v_{j_n}$, which can be viewed as a
path in $\tilde{Q}_n$ leaving the vertex $g'_n(b)$ of the shuttle and
going to $v_{j_n}$. Moreover, since the shuttle projects to a geodesic
path in $Q_n$, the length $b-a$ of
$g_n(a)$, $g_n(a+1)$, \ldots, $g_n(b-1)$, $g_n(b)$ is not smaller than the
length of the segment of the shuttle between the vertices $g_n'(a)$
and $g_n'(b)$. 

Therefore, we see that, if $g_n$ intersects the maximal geodesic
from~$c_0$ to~$v_*$ in~$Q_n$, we can construct from it a path in
$\tilde{Q}_n$ with same length, going from $v_{i_n}$ to $g'_n(a)$,
then taking the segment of the shuttle from $g'_n(a)$ to $g'_n(b)$,
then going from $g'_n(b)$ to $v_{j_n}$. This path is still a
concatenation of maximal wedge paths that are now in $\tilde{Q}_n$, so
by a new passage to the limit (possibly up to a new extraction), we
can find $k'\leq k$ and $s=s'_1$, $t'_1$, \ldots, $s'_{k'}$, $t'_{k'}=t$ such
that, for every $m$,
$$\tilde{d}_Z(s'_m,t'_m)=d_Z(s'_m,t'_m)\, ,\qquad
d_\ee(t'_m,s'_{m+1})=0\, ,$$
and such that 
$$ D(s,t)+\eta\geq
\sum_{m=1}^kd_Z(s_m,t_m)\geq  
\sum_{m=1}^{k'}d_Z(s'_m,t'_m) =\sum_{m=1}^{k'}\tilde{d}_Z(s'_m,t'_m)\geq \tilde{D}^*(s,t)\geq \tilde{D}(s,t)\, .$$ 
Again, since $\eta$ was arbitrary, this yields
$D(s,t)=\tilde{D}^*(s,t)=\tilde{D}(s,t)$. 

We obtain the same result with $\Gamma_1$ replaced by $\Gamma_0$ by a
similar reasoning. 
\end{proof}

\subsection{Proof of Theorem \ref{sec:scaling-limit-slices-3}}

We now turn the ``local'' lemma that we just proved into a
``global'' result, which is the content of Theorem
\ref{sec:scaling-limit-slices-3}.

\begin{proof}[Proof of Theorem \ref{sec:scaling-limit-slices-3}]
  Fix two points $x$, $y\in \tilde{S}$, and a continuous, injective path
  $f:[0,1]\to \tilde{S}$ going from $x$ to $y$.

  For every $r\in [0,1]$, let $F(r)\in
  [0,1]$ be an arbitrary point such that
  $\tilde{\bp}(F(r))=f(r)$.  Suppose first that $f$
  does not visit the points $\tilde{\bp}(0)$ and $\tilde{\bp}(s_*)$.
  Then for every $r\in [0,1]$, $\tilde{\bp}(F(r))$ is either not in
  $\gamma_0$, or not in $\gamma_1$. Assume for the
  moment that we are in the first case. It means that we can find a
  neighborhood $V_r$ of $r$ in $[0,1]$ and $\eps_r>0$ such that
  $\tilde{D}(F(r'),F(r))<\eps_r/2$ and $\tilde{D}(F(r'),\Gamma_0)>\eps_r$ for
  every $r'\in V_r$. In the second case, a similar property holds
  with $\Gamma_1$ instead of $\Gamma_0$. By taking a finite subcover,
  and applying Lemma \ref{sec:scaling-limit-slices-2}, we obtain the
  existence of $\eps>0$ depending on $f$ such that for every $r$, $r'\in
  [0,1]$, $|r-r'|\leq \eps$ implies 
$$\tilde{D}\big(F(r),F(r')\big)=\tilde{D}^*\big(F(r),F(r')\big)=D\big(F(r),F(r')\big)\, .$$
Hence, for every partition $0=r_0<r_1<\ldots<r_k=1$ such that
$|r_{i+1}-r_i|<\eps$ for every $i\in \{0,\ldots,k-1\}$,
$$ \sum_{i=0}^{k-1}\tilde{D}\big(F(r_i),F(r_{i+1})\big)
=\sum_{i=0}^{k-1}\tilde{D}^*\big(F(r_i),F(r_{i+1})\big)
=\sum_{i=0}^{k-1}D\big(F(r_i),F(r_{i+1})\big)\, , $$
which implies that 
\begin{equation}
  \label{eq:22}
  \mathrm{length}_{\tilde{D}}(f)=\mathrm{length}_{\tilde{D}^*}(f)=\mathrm{length}_{D}(\pi\circ
  f)\,
  .
\end{equation}
We now use the easy fact that for any metric space $(M,d)$ and every
continuous path $\gamma:[0,1]\to M$, the function $r\mapsto
\mathrm{length}_d(\gamma|_{[0,r]})$ is a non-decreasing,
left-continuous function from $[0,1]$ to $[0,\infty]$. Moreover, the
length function is additive in the sense that
$\mathrm{length}_d(\gamma)=\mathrm{length}_d(\gamma|_{[0,r]})+\mathrm{length}_d(\gamma|_{[r,1]})$
for every $r\in [0,1]$. These two properties together clearly imply
that \eqref{eq:22} is still valid if the injective, continuous path
$f$ is allowed to visit $\tilde{\bp}(0)$, $\tilde{\bp}(s_*)$, or
both. Taking the infimum over all such functions from a point $x$ to
$y$, and using Lemma \ref{sec:scaling-limit-slices-4}, we finally get
$\tilde{D}(x,y)=\tilde{D}^*(x,y)$, and that this quantity is the
infimum of $\mathrm{length}_D(\pi\circ f)$ over all injective continuous paths
from $x$ to $y$ in $\tilde{S}$, hence over all continuous paths from
$x$ to $y$ in $\tilde{S}$, not necessarily injective. 
\end{proof}

\section{Proof of Theorem~\ref{THMDISK}}\label{sec:conv-brown-disk}

\subsection{Subsequential convergence}\label{sec:subs}

We now move to quadrangulations with boundaries, which are our main
object of interest. Recall the construction of
Section~\ref{sec:bij_forests} and consider an encoding labeled forest
$(\bff,\ell)$ for a quadrangulation with a boundary. As in the
preceding section, we will further encode it by a pair of real-valued
functions. Before we proceed, it will be convenient to add an extra
vertex-tree $\rho_{l+1}$ with label $\ell(\rho_{l+1})=\ell(\rho_1)$ to
the forest. This extra vertex does not really play a part but its
introduction will make the presentation simpler. We also add~$l$ edges
between~$\rho_i$ and~$\rho_{i+1}$, for $1\le i \le l$. See
Figure~\ref{coding}.

We let $c_0$, $c_1$, \ldots, $c_{2n+l-1}$ be as in Section~\ref{sec:bij_forests} and we add to this list the corner~$c_{2n+l}$ incident to the extra vertex-tree~$\rho_{l+1}$. We define the contour and label processes on $[0,2n+l]$ by
$$C(j)=d_{\bff}(c_j,\rho_{l+1})-l\quad\text{ and }\quad \ell(j)=\ell(c_j),\qquad 0\leq j\leq 2n+l$$
and by linear interpolation between integer values.

\begin{figure}[htb!]
	\centering\includegraphics[width=.95\linewidth]{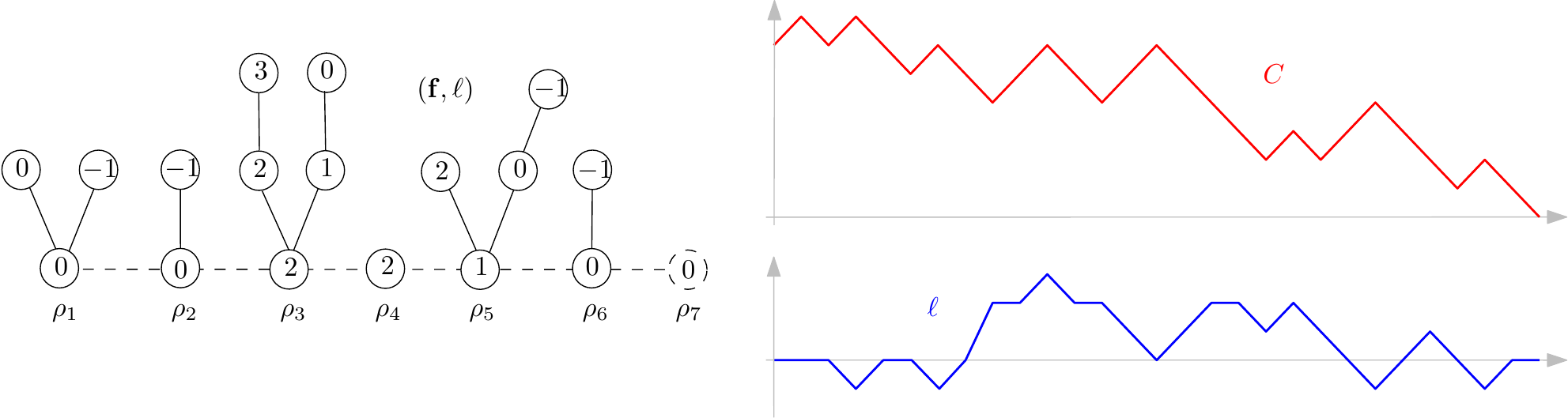}
	\caption{The contour and label processes associated with the labeled forest of Figure~\ref{fig:schaeffer3}. The extra vertex-tree~$\rho_7$ and the edges linking the roots are represented with a dashed line. Note that the normalization we chose for the labels is equivalent to imposing $\ell(0)=0$.}
	\label{coding}
\end{figure}

Let us fix $L\in (0,\infty)$ and a sequence $(l_n,n\geq 1)$ such that $l_n\sim L\sqrt{2n}$ as $n\to\infty$. 
We let $(F_n,\ell_n)$ be uniformly distributed over the set of labeled
forests of $l_n$ trees with $n$ edges in total, and let $(Q_n,v_*)$ be
the random pointed quadrangulation\footnote{We will use notation like
  $Q_n$, $C_n$, $\ell_n$, $D_n$, $D$ with a different meaning from the preceding section
  in order to keep exposition lighter.} associated with $F_n$ via the 
bijection of Section~\ref{sec:bij_forests}. Note that up to re-rooting
$Q_n$ at a uniform corner incident to the root face, we may assume
that $Q_n$ is uniform in $\bQnsn$ by Lemma \ref{awayfrom}. 

We let~$C_n$, $\ell_n$ be the associated contour and label processes, and we define their
renormalized versions
$$C_{(n)}(s)=\frac{C_n\big((2n+{l_n})s\big)}{\sqrt{2n}}\, ,\qquad
\ell_{(n)}(s)=\left(\frac{9}{8n}\right)^{1/4}\ell_n\big((2n+{l_n})s\big)\,
,\qquad 0\leq s\leq 1\, .$$ We let $D_n(i,j)$ be the distance in~$Q_n$
between the vertices incident to the $i$-th and $j$-th corner
of~$F_n$, for $i,j\in \{0,1,\ldots,2n+{l_n}\}$. We extend $D_n$ to a
continuous function on $[0,2n+{l_n}]^2$ by the exact same formula as
\eqref{eq:6}, and we finally define its renormalized version
\begin{equation}\label{eq:Dn}
D_{(n)}(s,t)=\left(\frac{9}{8n}\right)^{1/4}D_n\big(
  (2n+{l_n})\,s, (2n+{l_n})\,t\big)\, ,\qquad 0\leq s,t\leq 1\, .
\end{equation}
It is shown in \cite{bettinelli11b} that, from every increasing family
of positive numbers, one can extract a further subsequence along which
$$(C_{(n)},\ell_{(n)},D_{(n)})\build\longrightarrow_{n\to\infty}^{(d)}(X,Z,D)$$
in distribution in the space
$\mathcal{C}([0,1])\times\mathcal{C}([0,1])\times
\mathcal{C}([0,1]^2)$. (At this moment, the need of extracting a
subsequence is caused by the last coordinate~$D_{(n)}$ and the
convergence without extraction holds if one drops this coordinate.)
Here, $D$ is a random pseudo-metric on $[0,1]$ and $(X,Z)$
law\footnote{This is of course an abuse of notation since $(X,Z)$
  previously denoted the canonical process, however we did not want to
  introduce a further specific notation at this point.}
$\mathbb{F}^1_L$ defined in Section~\ref{sec:defin-main-prop}, so that
$X$ is a first-passage bridge, attaining level $-L$ for the first time
at time~$1$, and~$Z$ is the associated snake process.

Moreover, the pointed random metric space
$(\sV(Q_n),(9/8n)^{1/4}d_{Q_n},v_*)$ converges in distribution, still
along the same subsequence, to the random metric space
$([0,1]/\{D=0\},D,x_*)$, in the sense of the pointed Gromov--Hausdorff topology. Here, we let $x_*=\bp(s_*)$, where
$\bp:[0,1]\to [0,1]/\{D=0\}$ is the canonical projection, and $s_*$ is
the (a.s.\ unique \cite[Lemma~11]{bettinelli11b}) point in $[0,1]$ at which $Z$ reaches its global
minimum.

\begin{prp}[\cite{bettinelli11b}]
  Almost surely, the space $\mathbf{D}=[0,1]/\{D=0\}$ is a topological
  disk whose boundary $\partial \mathbf{D}$ satisfies
\begin{equation}\label{eq:bdry}
\bp^{-1}(\partial \mathbf{D})=\{s\in
  [0,1]:X_s=\underline{X}_s\}\, .
\end{equation}
Almost surely, the Hausdorff dimension of $\mathbf{D}$ is 4, and that
of $\partial\mathbf{D}$ is $2$. 
\end{prp}

Recall that $v_*$ is a uniform random vertex in $Q_n$, conditionally
given the latter. From this observation, we obtain an invariance under
re-rooting property of $(\mathbf{D},D,x_*)$, along the same lines as
\cite{legall08}. 

\begin{lmm}
\label{sec:conv-brown-disk-3}
Let $U$ be a uniform random variable in $[0,1]$, independent of
$(X,Z,D)$. Then the two pointed spaces $(\mathbf{D},D,\bp(U))$ and
$(\mathbf{D},D,x_*)$ have the same distribution.
\end{lmm}

The following lemma is an easy consequence of the study of geodesics done in~\cite{bettinelli14gbs}. 
\begin{lmm}
\label{sec:conv-brown-disk-2}
  Almost surely, for every $x\in
 \mathbf{D}\setminus \partial\mathbf{D}$, there exists a geodesic
 from $x$ to $x_*$ that does not intersect $\partial
 \mathbf{D}$. Moreover, this is the only geodesic from $x$ to $x_*$ for $\mu$-almost every
 $x\in \mathbf{D}$, where $\mu=\bp_*(\mathrm{Leb}_{[0,1]})$. 
\end{lmm}

\begin{proof}
For $s\in[0,1]$, we define the path $\Phi_s: [0, D(s_*,s)]\to\mathbf{D}$ by
$$\Phi_s(w)= \bp\left(\sup\left\{ r\, :\, \underline{Z}_{r,s} = Z_{s_*}+ w\right\} \right),\qquad 0\le w\le D(s_*,s)=Z_s-Z_{s_*}.$$
It is shown in \cite[Proposition~23]{bettinelli14gbs} that the
path~$\Phi_s$ is a geodesic from~$x_*$ to~$\bp(s)$ in~$\mathbf{D}$ and
that a.s.\ all the geodesics from~$x_*$ are of this form. We call \emph{increase point} of a function a point~$t$ such that the function is greater than its value at~$t$ on a small interval of the form $[t-\eps,t]$ or $[t, t+\eps]$ for some $\eps>0$. Clearly, for $0\le w < D(s_*,s)$, the point $\sup\left\{ r\, :\, \underline{Z}_{r,s} = Z_{s_*}+ w\right\}$ is an increase point of the process~$Z$, which is furthermore different from~$0$. On the other hand, the expression~\eqref{eq:bdry} shows that $\bp^{-1}(\partial \mathbf{D})$ is made only of increase points of~$X$, together with the point~$0$. Moreover, \cite[Lemma~18]{bettinelli14gbs} states that, a.s., the processes~$X$ and~$Z$ do not share any increase points. As a consequence, $\Phi_s$ may only intersect $\partial \mathbf{D}$ at its endpoint~$\bp(s)$ and the first statement follows.

In addition, \cite[Proposition~17]{bettinelli14gbs} entails that, for $0\le s\le t\le 1$, $D(s,t)=0$ if and only if one of the following occurs:
\refstepcounter{equation}\label{identification}
\begin{enumerate}[(\theequation a)]
	\item $X_s=X_t=\underline{X}_{s,t}$;\label{idX}
	\item $Z_s=Z_t=\underline{Z}_{s,t}$ or $Z_s=Z_t=\underline{Z}_{t,s}$.\label{idZ}
\end{enumerate}
Moreover, for $s\neq t$, only one of the previous situations can happen. In some sense, this can be thought of as a continuous version of the bijection from Section~\ref{sec:bij_forests}: point~(\ref{identification}\ref{idX}) constructs the continuous random forest and drawing an arc between a corner and its successor becomes, in the limit, identifying points with the same label and such that the labels visited in between in the contour order are all larger (point~(\ref{identification}\ref{idZ})). Standard properties of the process~$Z$ then allow us to conclude that $\mathrm{Leb}_{[0,1]}(\{s:\,\exists t\neq s:\, D(s,t)=0\})=0$, so that, for $\mu$-almost every $x\in \mathbf{D}$, the set $\bp^{-1}(x)$ is a singleton and the only geodesic from~$x_*$ to~$x$ is thus~$\Phi_{\bp^{-1}(x)}$.
\end{proof}

Combining Lemmas \ref{sec:conv-brown-disk-3} and
\ref{sec:conv-brown-disk-2}, we see that the conclusion of the latter
is still valid if~$x_*$ is replaced by a uniformly chosen point in
$\mathbf{D}$, that is, a random point of the form $\bp(U)$ as in the
first lemma. Finally, we will use the following result.

\begin{lmm}[\cite{bettinelli11b}]\label{lemD}
The following properties hold almost surely.
\begin{itemize}
	\item $D\in\mathcal D$.
	\item $D(s,s_*)=Z_s-Z_{s_*}$ for every $s\in[0,1]$.
\end{itemize}
\end{lmm}

\subsection{Identification of the limit}\label{sec:identification-limit}

Recall the notation $D^*$ from Section \ref{sec:br-disks}. In this
section, we show the following analog to the first part of Theorem~\ref{sec:scaling-limit-slices-3}.

\begin{thm}
  \label{sec:conv-brown-disk-1}
Almost surely, it holds that $D=D^*$. 
\end{thm}

Theorem~\ref{THMDISK} is an
immediate consequence of this. Indeed, since $D^*$ is a
measurable function of $(X,Z)$, this shows that $D^*$ is the only
possible subsequential limit of $D_{(n)}$. This, combined with the
tightness of the sequence $(D_{(n)},n\geq 1)$ that we alluded to
above, implies that $D_{(n)}$ converges in distribution to $D^*$. 

In turn, this convergence implies that of $(9/8n)^{1/4}Q_n$ to
$\bd_{L}=(\mathbf{D},D^*)$ in the Gromov--Hausdorff sense and
even that of the pointed space $((9/8n)^{1/4}Q_n,v_*)$ to
$(\bd_L,x_*)$, where we recall that $x_*=\bp(s_*)$. Let us
recall how to prove this fact. First, one can assume that the convergence
of $(C_{(n)},\ell_{(n)},D_{(n)})$ to $(X,Z,D^*)$ is almost-sure, by using
Skorokhod's representation theorem. Then we define a correspondence
$\mathcal{R}_n$ between $Q_n$ and $\bd_L$ by
$$\mathcal{R}_n=\{(v_{\lfloor(2n+l_n)s\rfloor},\bp(s)):s\in
[0,1]\}\cup\{(v_*,\bp(s_*))\}\, ,$$ where $v_i$ is the vertex of $Q_n$
incident to the $i$-th corner $c_i$. It is elementary to see from the
uniform convergence of $D_{(n)}$ to $D^*$ that the distortion of
$\mathcal{R}_n$ with respect to the metrics $(9/8n)^{1/4}d_{Q_n}$ and
$D^*$ converges to $0$ as $n\to\infty$.

Recall that $[a,b]$ is an excursion interval of $X$ above
$\underline{X}$ if $a<b$ and $X_a=X_b=\underline{X}_b$. Let us arrange
the excursion intervals of $X$ above $\underline{X}$ as $[a_i,b_i]$,
$i\geq 1$ in decreasing order of length.  For a given $i$, the
excursion interval $[a_i,b_i]$ encodes a slice in the sense of Section~\ref{sec:known-scaling-limit}. Namely,  for $s$, $t\in
[a_i,b_i]$, let
$d_Z^i(s,t)=Z_s+Z_t-2\underline{Z}_{s\wedge t,s\vee t}$, and
$$\tilde{D}^i(s,t)=\inf\left\{\sum_{j=1}^kd^i_Z(s_j,t_j):\begin{array}{l}k\geq
    1\, ,\quad t_1,s_2,t_2,\ldots,s_k\in[a_i,b_i],\, s_1=s,\,t_k=t ,\\
d_X(t_j,s_{j+1})=0\, \mbox{ for every }  j\in \{1,\ldots, k-1\}\,
  \end{array}
\right\}\, .$$ By simple scaling properties and excursion theory,
conditionally given the excursion lengths $(b_i-a_i)$, $i\geq 1$, the
spaces $\tilde{S}^i=[a_i,b_i]/\{\tilde{D}^i=0\}$, equipped with the
induced distance, still called $\tilde{D}^i$, are independent versions
of the Brownian slices of Section~\ref{sec:known-scaling-limit}, with
distances rescaled by $(b_i-a_i)^{1/4}$ respectively.  The next key
lemma states that the distance $D$ can be identified as a metric
gluing of these slices along their boundaries. This guides the
intuition of its proof, which will partly consist in going back to the
discrete slices that compose the quadrangulations with a boundary of which we
took the limit.

\begin{lmm}\label{lem:Dfinite}
  Let $\tilde{D}$ be the pseudo-metric on $\bigcup_{i\geq 1} [a_i,b_i]$
  defined by $\tilde{D}(s,t)=\tilde{D}^i(s,t)$ if $s$, $t\in [a_i,b_i]$ for some
  $i\geq 1$, and $\tilde{D}(s,t)=\infty$ otherwise. 
Then for almost every $(s,t)\in [0,1]^2$ with respect to the Lebesgue measure, it holds that 
$$D(s,t)=\inf\left\{\sum_{j=1}^k\tilde{D}(s_j,t_j):\begin{array}{l}k\geq
    1\, ,\quad t_1,s_2,t_2,\ldots,s_k \in[0,1],\, s_1=s,\,t_k=t ,\\
d_Z(t_j,s_{j+1})=0\, \mbox{ for every }  j\in \{1,\ldots, k-1\}\,
  \end{array}
\right\}\, .$$
Moreover, the above infimum is attained. 
\end{lmm}

\begin{proof}
Clearly, $D(s,t)\le d_Z(s,t)\le d_Z^i(s,t)$ whenever $s$, $t\in [a_i,b_i]$, so that $D(s,t)\le \tilde D(s,t)$ for $s$, $t\in [0,1]$ and, as a consequence, the left-hand side is smaller than the right-hand side. We then only need to prove the converse inequality.

Let us first define the discrete analogs to the functions
$\tilde{D}^i$. We consider the $i$-th largest tree~$\bt$ of~$F_n$ and
we suppose that it is visited between times~$a_i^n$ and~$b_i^n$ in the
contour order of~$F_n$. For $j$, $k\in\{a_i^n,\ldots,b_i^n\}$, we let
$\tilde D_n(j,k)$ be the distance in the slice corresponding to~$\bt$
between the vertices~$v_j$ and~$v_k$ incident to the $j$-th and $k$-th
corner of~$F_n$. In other words, $\tilde D_n(j,k)$ is the length of a
shortest path linking~$v_j$ to~$v_k$ and that do not ``traverse'' the
images in~$Q_n$ of the maximal geodesic and shuttle of the
aforementioned slice. We then extend~$\tilde D_n$ to a continuous
function on $[a_i^n,b_i^n]^2$ by bilinear interpolation, and define
its renormalized version $\tilde{D}^i_{(n)}$ on a subsquare of
$[0,1]^2$ by the analog of~\eqref{eq:Dn}. We define
$\tilde{D}^i_{(n)}$ arbitrarily for $i>l_n$. 

As a simple consequence of the convergence \eqref{eq:21}, reformulated
in the context of the excursion intervals $[a_i^n,b_i^n]$, and of
Theorem~\ref{sec:scaling-limit-slices-3}, we have that
\begin{equation}\label{eq:slices}
\left(C_{(n)},\ell_{(n)},\big(\tilde D^i_{(n)}\big)_{i\geq
    1}\right)\ton \left(X,Z,\big(\tilde D^i\big)_{i\geq 1}\right)
\end{equation}
in distribution in the space $\mathcal{C}([0,1])\times\mathcal{C}([0,1])\times\mathcal{C}([0,1]^2)^N$.
Applying Skorokhod's representation theorem, we also assume from now on that this convergence holds a.s.

It suffices to prove the claimed formula for $D(s,t)$ when $s$, $t$
are replaced by two independent uniform random variables $U$, $V$,
independent of the other random variables considered so far. Let
$\gamma:[0, D(U,V)]\to\mathbf D$ be the geodesic in $(\mathbf{D},D)$
from $\bp(U)$ to $\bp(V)$, which by Lemmas~\ref{sec:conv-brown-disk-3}
and~\ref{sec:conv-brown-disk-2} is unique and does not intersect
$\partial \mathbf{D}$, a.s. Let also $\Img(\gamma)=\gamma([0,D(U,V)])$
be the image of $\gamma$ and define
$$I(U,V)=\big\{i\geq 1:\bp^{-1}(\Img(\gamma))\cap [a_i,b_i]\neq \varnothing\big\} .$$ 

\begin{clm}\label{claim1}
The set $I(U,V)$ is finite almost surely.
\end{clm}

\begin{proofc}
  Let us argue by contradiction, assuming that $I(U,V)$ is infinite
  with positive probability. Then it holds that, still with
  positive probability, there is an increasing integer sequence
  $(i_n)_{n\geq 1}$ and a sequence $(r_n)_{n\geq 1}$ with values in
  $[0,D(U,V)]$ such that $\gamma(r_n)\in
  \bp([a_{i_n},b_{i_n}])$. Then, up to extraction, the sequence
  $(r_n)$ converges to some limit $r$, and if $s_n\in
  [a_{i_n},b_{i_n}]$ is a choice of a given element in
  $\bp^{-1}(\gamma(r_n))$, then, again up to possibly further
  extraction, $(s_n)$ converges to a limit $s$ with
  $\bp(s)=\gamma(r)$. By construction, $s$ is not in $\bigcup_{i\geq
    1}(a_i,b_i)$, since the intervals in this union are pairwise
  disjoint. This implies that $X_s=\underline{X}_s$, meaning that
  $\gamma(r)=\bp(s)\in \partial \mathbf{D}$, which is the
  contradiction we were looking for.
\end{proofc}

Let $\gamma^i_l$, $\gamma^i_r$ be the left and right ``geodesic boundaries'' of the
space $([a_i,b_i]/d_Z^i,d_Z^i)$, defined by
$$\gamma^i_l(t)=\bp\big(\inf\{s\in [ a_i,b_i]:Z_s=Z_{a_i}-t\}\big)\, ,\quad
\gamma^i_r(t)=\bp\big(\sup\{s\in [a_i,b_i]:Z_s=Z_{a_i}-t\}\big)\, ,$$
where~$t$ ranges over $[0,Z_{a_i}-\underline{Z}_{a_i,b_i}]$ (recall
that $Z_{a_i}=Z_{b_i}$). Those are geodesic paths in $(\mathbf{D},D)$
from $\bp(a_i)$ to $\bp(s_i^*)$, where $s_i^*$ is the (a.s.\ unique
\cite{legweill}) time in $[a_i,b_i]$ at which $Z$ attains its infimum
on that same interval. Alternatively, these paths are parts of the
geodesics~$\Phi_{a_i}$ and~$\Phi_{b_i}$ introduced earlier. Note also
that $\Img(\gamma^i_l)\cap\Img(\gamma^i_r)$ is not necessarily reduced
to $\{\bp(s_i^*)\}$.

\begin{clm}\label{claim2}
It is not possible to find $r_1<r_2<r_3$ such that
$\gamma(r_1),\gamma(r_3)\in \Img(\gamma^i_l)$ and
$\gamma(r_2)\notin \Img(\gamma^i_l)$. The same
statement is valid for $\gamma^i_r$ instead of $\gamma^i_l$.
\end{clm}

\begin{proofc}
Indeed, such a situation
would clearly violate the uniqueness of the geodesic $\gamma$, since
we could replace it between times $r_1$ and $r_3$ by the arc of
$\gamma^i_l$ from $\gamma(r_1)$ to $\gamma(r_3)$, and still obtain a
geodesic from $\bp(U)$ to $\bp(V)$, distinct from $\gamma$.
\end{proofc}

\begin{clm}\label{claim3}
Almost surely, for every $i\geq 1$, the topological
  boundary of $\bp([a_i,b_i])$ in $(\mathbf{D},D)$ is included in
  $\Img(\gamma_l^i)\cup \Img(\gamma^i_r)$.
\end{clm}

\begin{proofc}
  This claim is relatively obvious with the interpretation that
  $([a_i,b_i]/\{\tilde{D}^i=0\},\tilde{D}^i)$ is a space with geodesic
  boundaries given by $\gamma_l^i$, $\gamma_r^i$, but since we are not
  referring explicitly to these spaces, let us give a complete proof
  for this. In fact, the topological boundary of $\bp([s,t])$ for any
  $a_i\le s\le t\le b_i$ is given by~\cite[Lemma~21]{bettinelli14gbs}
  but, as the proof is quite short, we restate the arguments
  here. Note that $\bp([a_i,b_i])$ is closed so that every point in
  $\partial \bp([a_i,b_i])$ is of the form~$\bp(s')$ for some
  $s'\in[a_i,b_i]$ and is a limit of a sequence of points of the form
  $\bp(s_n)$, $n\geq 1$, where $s_n\notin [a_i,b_i]$ for every $n\geq
  1$. Up to extraction, $(s_n)$ converges to a limit
  $s\notin(a_i,b_i)$ such that $D(s,s')=0$. If $s\in \{a_i,b_i\}$ then
  the claim follows immediately. Otherwise, $s\neq s'$ and, as
  mentioned during the proof of Lemma~\ref{sec:conv-brown-disk-2},
  this implies $d_X(s,s')=0$ (\ref{identification}\ref{idX}) or
  $d_Z(s,s')=0$ (\ref{identification}\ref{idZ}). It cannot
  hold that $d_X(s,s')=0$ because $s'\in [a_i,b_i]$ while $s\notin[a_i,b_i]$, so necessarily
  $d_Z(s,s')=0$. Assuming for instance that
  $Z_s=Z_{s'}=\underline{Z}_{s,s'}$, so that $Z_u\geq Z_{s'}$ for
  every $u\in [a_i,s']$, this implies that $d_Z(s,\inf\{u\in
  [a_i,b_i]:Z_u=Z_s)=0$.  Finally, we get that
  $\bp(s)=\gamma_l^i(Z_{a_i}-Z_s)\in \Img(\gamma^i_l)$. Similarly, if
  $Z_s=Z_{s'}=\underline{Z}_{s',s}$, we obtain that $\bp(s)\in
  \Img(\gamma^i_r)$.
\end{proofc}

From the three claims above, we obtain that there exists a finite number of points $x_1$, $x_2$, \ldots, $x_{k+1}\in\mathbf D$ and integers $i_1$, \ldots, $i_k$ with
$x_1=\bp(U)$, $x_{k+1}=\bp(V)$, such that~$\gamma$ visits the points~$x_1$, $x_2$, \ldots, $x_{k+1}$ in this order, and such that the segment of~$\gamma$ between~$x_j$ and~$x_{j+1}$ is 
\refstepcounter{equation}\label{slice}
\begin{enumerate}[(\theequation a)]
	\item either included in $\Img(\gamma^{i_j}_l)$ or included in $\Img(\gamma^{i_j}_r)$\label{bdry}
	\item or included in $\bp([a_{i_j},b_{i_j}])$ and such that its intersection with $\Img(\gamma^{i_j}_l)\cup\Img(\gamma^{i_j}_r)$ is a subset of $\{x_j,x_{j+1}\}$.\label{inside}
\end{enumerate}
Indeed, Claims~\ref{claim1} and~\ref{claim2} entail that $\Img(\gamma)\cap(\bigcup_i \Img(\gamma^{i}_l)\cup\Img(\gamma^{i}_r))$ is a finite union of segments satisfying~(\ref{slice}\ref{bdry}) and the parts of~$\gamma$ linking two successive such segments satisfy~(\ref{slice}\ref{inside}), by Claim~\ref{claim3}. Since the segment of~$\gamma$ between~$x_j$ and~$x_{j+1}$ is included in $\bp([a_{i_j},b_{i_j}])$ in both cases, we may choose~$s_j$, $t_j\in [a_{i_j},b_{i_j}]$ such that $x_j=\bp(s_j)$ and $x_{j+1}=\bp(t_j)$. For any such choice,
$$D(U,V)=\sum_{j=1}^{k} D(s_j,t_j)\, .$$

We will soon justify that we can choose $s_j$, $t_j$ satisfying the
extra property that $D(s_j,t_j)=\tilde{D}(s_j,t_j)$ on the event $\{\max(I(U,V))\le N\}$. Since, by definition, $\bp(t_j)=\bp(s_{j+1})$, one has $d_X(t_j,s_{j+1})=0$ or $d_Z(t_j,s_{j+1})=0$ so that $\tilde D(t_j,s_{j+1})=0$ or $d_Z(t_j,s_{j+1})\allowbreak=0$. Similarly, $\tilde D(U,s_{1})=0$ or $d_Z(U,s_{1})=0$ and $\tilde D(t_k,V)=0$ or $d_Z(t_k,V)=0$. As a result, up to potentially doubling some $s_j$'s and $t_j$'s, we wrote $D(U,V)$ in the desired form and we conclude the proof by letting $N\to\infty$, as $\max(I(U,V))<\infty$ almost surely by Claim~\ref{claim1}.

\bigskip Let us work from now on on the event $\{\max(I(U,V))\le N\}$
and justify the possibility of choosing $s_j$, $t_j$ as previously
claimed. If the segment of~$\gamma$ between~$x_j$ and~$x_{j+1}$
satisfies~(\ref{slice}\ref{bdry}), then the claim readily follows from
Lemma~\ref{lemD}, as $$D(s_j,t_j)\ge
|D(s_j,s_*)-D(t_j,s_*)|=|Z_{s_j}-Z_{t_j}|\ge d^i_Z(s_j,t_j) \ge\tilde
D(s_j,t_j)$$ (Recall that the converse inequality always holds.)  We
now suppose that the segment of~$\gamma$ between~$x_j$ and~$x_{j+1}$
satisfies~(\ref{slice}\ref{inside}) and we go back to the discrete
setting. For $u\in\{0,1/(2n+l_n),\ldots,1\}$, we denote by $c_n(u)$
the $(2n+l_n)u$-th corner of~$F_n$. We let~$a_{i_j}^n$ and~$b_{i_j}^n$
be such that $c_n(a_{i_j}^n)$ and $c_n(b_{i_j}^n)$ are the first and
last corners of the $i_j$-th largest tree of~$F_n$. Standard
properties of Brownian motion and the convergence $C_{(n)}\to X$
entail that $a_{i_j}^n\to a_{i_j}$ and~$b_{i_j}^n\to b_{i_j}$. Choose
two sequences $s^n_j$, $t^n_j\in[a_{i_j}^n,b_{i_j}^n]$ indexed by $n$
such that $s^n_j\to s_j$ and $t^n_j\to t_j$. We denote by~$u_j^n$
and~$v_j^n$ the vertices incident respectively to $c_n(s^n_j)$ and
$c_n(t^n_j)$ and we let~$\gamma_j^n$ be a geodesic in~$Q_n$
from~$u_j^n$ to~$v_j^n$.

We also let $\sV_j$ be the set of vertices of $i_j$-th largest tree
of~$F_n$ that do not belong to the maximal geodesic of the slice
corresponding to this tree, seen as a subset of $\sV(Q_n)$. We will see
that $\Img(\gamma_j^n)\setminus \sV_j$ is only constituted of vertices
``close'' to the extremities of~$\gamma_j^n$ in the
scale~$n^{-1/4}$. Notice first that the middle point
$\gamma_j^n\big(\lfloor D_{Q_n}(u_j^n,v_j^n)/2\rfloor\big)$
of~$\gamma_j^n$ necessarily belongs to~$\sV_j$ for large~$n$. Indeed,
let us assume otherwise. Then, for infinitely many values of $n$, we can
find real numbers $u_n\notin[a_{i_j}^n,b_{i_j}^n]$ such that
$c_n(u_n)$ is incident to the middle point of~$\gamma_j^n$. Up to
further extraction, we may suppose that $u_n\to u\notin
(a_{i_j},b_{i_j})$, so that $\bp(u)$ does not belong to the interior
of $\bp([a_{i_j},b_{i_j}])$. As $\bp(u)$ is at mid-distance
between~$x_j$ and~$x_{j+1}$, we obtain a contradiction
with~(\ref{slice}\ref{inside}).

We then let $\tilde s_j^n\in[a_{i_j}^n,b_{i_j}^n]$ be such that
$c_n(\tilde s_j^n)$ is incident to
$$\gamma_j^n\big(\min\big\{\alpha\le
D_{Q_n}(u_j^n,v_j^n)/2\,:\,\gamma_j^n(\beta)\in \sV_j \text{ for all }
\beta\in [\alpha, D_{Q_n}(u_j^n,v_j^n)/2]\big\}\big)$$ and,
symmetrically, $\tilde t_j^n\in[a_{i_j}^n,b_{i_j}^n]$ be such that
$c_n(\tilde t_j^n)$ is incident to
$$\gamma_j^n\big(\max\big\{\alpha\ge D_{Q_n}(u_j^n,v_j^n)/2\,:\,\gamma_j^n(\beta)\in \sV_j \text{ for all } \beta\in[D_{Q_n}(u_j^n,v_j^n)/2,\alpha]\big\}\big).$$
Up to further extraction, we may suppose that $\tilde s_j^n\to \tilde
s_j$ and $\tilde t_j^n\to \tilde t_j$. We necessarily have
$\bp(s_j)=\bp(\tilde s_j)$. Indeed, let us argue by contradiction and
suppose that $\bp(s_j)\neq\bp(\tilde s_j)$. The definition immediately
entails that $\bp(\tilde
s_j)\in\Img(\gamma^{i_j}_l)\cup\Img(\gamma^{i_j}_r)$. But, as
$\bp(\tilde s_j)\in \bp([a_{i_j},b_{i_j}])$, the
condition~(\ref{slice}\ref{inside}) yields $\bp(s_j)=\bp(\tilde s_j)$,
a contradiction. This implies $D(s_j,\tilde s_j)=0$, which also
implies $d_X(s_j,\tilde s_j)=0$ or $d_Z(s_j,\tilde s_j)=0$, so that
$\tilde D(s_j,\tilde s_j)=0$ as~$s_j$ and~$\tilde s_j$ both belong to
the same excursion interval $[a_{i_j},b_{i_j}]$. The same argument
shows that $D(t_j,\tilde t_j)=\tilde D(t_j,\tilde t_j)=0$. Finally,
$D_{(n)}(\tilde s^n_j,\tilde t^n_j)=\tilde D_{(n)}^{i_j}(\tilde
s^n_j,\tilde t^n_j)$ by construction and we obtain $D(\tilde
s_j,\tilde t_j)=\tilde D(\tilde s_j,\tilde t_j)$ by~\eqref{eq:slices},
and then $D(s_j,t_j)=\tilde{D}(s_j,t_j)$ by the previous discussion.
\end{proof}

Note that from the formula for $D(s,t)$ given in the statement of
Lemma \ref{lem:Dfinite} and the definition of $D^*$, it holds that
$D(s,t)\geq D^*(s,t)$ for Lebesgue-almost every $s$, $t\in [0,1]$, so
that equality holds since $D\leq D^*$ by Lemma~\ref{lemD}. Since
$D^*\leq d_Z$, which is continuous on $[0,1]^2$ and null on the
diagonal, we get immediately that the pseudo-metrics $D$, $D^*$ are
continuous when seen as functions on $[0,1]^2$, and by density we get
that $D=D^*$. This proves Theorem~\ref{sec:conv-brown-disk-1}.

\section{Boltzmann random maps and well-labeled mobiles}\label{sec:univ}

\subsection{The Bouttier--Di Francesco--Guitter
  bijection}\label{sec:boutt-di-franc}

There is a well-known extension of the Cori--Vauquelin--Schaeffer
bijection to general maps. This extension, due to Bouttier, Di
Francesco and Guitter \cite{BdFGmobiles}, can roughly be described in
the following way. Any {\em bipartite} map can be coded by an object
called a {\em well-labeled mobile}. Namely, a mobile is a rooted plane
tree $\bt$ (we usually call $e_0$ its root edge) together with a
bicoloration of its vertices into ``white vertices'' and ``black
vertices.'' We denote by $\sV_\circ(\bt)$, $\sV_\bullet(\bt)$ the
corresponding sets of vertices, and ask that any two neighboring
vertices carry different colors, and that $e_0^-\in\sV_\circ(\bt)$,
meaning that mobiles are rooted at a white vertex.

Moreover, the set $\sV_\circ(\bt)$
carries a label function $\ell:\sV_\circ(\bt)\to \Z$, that satisfies the
following property: if $v'\in \sV_\bullet(\bt)$ is a black vertex, and if
$v'_0$, $v'_1$, \ldots, $v'_{k-1}\in \sV_\circ(\bt)$ denote the neighbors of~$v'$
arranged in clockwise order around~$v'$ induced by the planar
structure of $\bt$ (so that $k=\deg_\bt(v')$), it holds that
$$\ell(v'_{i+1})-\ell(v'_{i})\geq -1\, ,\qquad \forall\, i\in
\{0,1,\ldots,k-1\}\, ,$$ with the convention that $v'_k=v'_0$. A
simple counting argument shows that, as soon as one of the labels, say
that of $\ell(v'_0)$, is fixed, there are exactly $\binom{2k-1}{k}$
possible choices for the other labels $\ell(v'_1)$, \ldots,
$\ell(v'_{k-1})$. At this point of the discussion, we do not insist
that the label of any given vertex is fixed, so we really view $\ell$
as a function defined up an additive constant, as we did in
Section~\ref{sec:scha-biject-first}. We will fix a normalization in
the next section.

In our context of maps with a boundary, we use the following
conventions. The objects encoding the bipartite maps with
perimeter~$2l$ (maps of~$\bB_l$) are forests $\bff=(\bt_1, \dots,
\bt_l)$ of~$l$ mobiles, together with a labeling function
$\ell:\sV_\circ(\bff)=\bigsqcup_i \sV_\circ(\bt_i)\to\Z$ satisfying
the following:
\begin{itemize}
\item for $1\le i \le l$, the mobile $\bt_i$ equipped with the
  restriction of~$\ell$ to $V_\circ(\bt_i)$ is a well-labeled mobile;
\item for $1\le i \le l$, we have $\ell(\rho_{i+1})\ge
  \ell(\rho_{i})-1$, where~$\rho_i$ denotes the root vertex of~$\bt_i$
  and $\ell(\rho_{l+1})=\ell(\rho_1)$.
\end{itemize}
\begin{rem}
  These forests are in simple bijection with the set of mobiles rooted (unusually)
  at a black vertex of degree~$l$. But since the external face really
  plays a different role from the other faces, we prefer indeed to
  view those as forests of individual mobiles, rather than one single
  mobile. 
\end{rem}

The BDG bijection is very similar to the construction presented in
Section~\ref{sec:bij_forests}. We consider a forest $\bff=(\bt_1,
\dots, \bt_l)$ of~$l$ mobiles, labeled by~$\ell$ as above and we set
$\ell_*=\min\{\ell(v):v\in \sV_\circ(\bff)\}-1$. We let $N^\sV$ be its
number of white vertices, $N^\sF$ be its number of black vertices, and
$N^\sE=N^\sV+N^\sF$ be its total number of vertices. (The reason for
this notation will become clear in a short moment.)

We identify~$\bff$ with the map obtained by adding~$l$ edges linking the roots $\rho_1$, $\rho_2$, \ldots, $\rho_l$ of the successive trees in a cycle. This map has one face of degree~$l$ incident to the~$l$ added edges and another face of degree $2N^\sE-l$, incident to the~$l$ added edges as well as all the mobiles. In the latter face, we let $c_0$, $c_1$, \ldots, $c_{2N^\sV-l-1}$ be the sequence of corners incident to white vertices, listed in contour order, starting from the root corner of~$\bt_1$. We extend this list by periodicity and add one corner~$c_\infty$ incident to a vertex~$v_*$ lying inside the face of degree $2N^\sE-l$, with label $\ell(c_\infty)=\ell(v_*)=\ell_*$. We define the successor functions by~\eqref{eq:succ} and draw arcs in a non-crossing fashion from~$c_i$ to~$s(c_i)$ for every $i\in\{0,1,\ldots,2N^\sV-l-1\}$.
We root the resulting map at the corner of the degree~$2l$-face that is incident to the root vertex of~$\bt_1$. We obtain a rooted bipartite map~$\bm$ with perimeter~$2l$, with vertex set $\sV_\circ(\bt)\cup\{v_*\}$, which is naturally pointed at~$v_*$, and such that the root edge points away from~$v_*$. 

As in Section~\ref{sec:bij_forests}, the fact that the root edge necessarily points away from~$v_*$ is a bit unfortunate and we use the same trick in order to overcome this technicality. More precisely, we consider the map obtained from~$\bm$ by forgetting its root and re-rooting it at a corner chosen uniformly at random among the~$2l$ corners of the root face.

\bigskip
A noticeable fact about the BDG bijection is that the black vertices of the forest are in bijection with the internal faces of the map. More precisely, if $v\in \sV_\bullet(\bff)$ corresponds to the face~$f$ of~$\bm$, then $\deg_\bm(f)=2\deg_\bff(v)$. Furthermore, the white vertices are bijectively associated with $\sV(\bm)\setminus\{v_*\}$ (so that we can naturally identify these two sets), in such a way that the label
function $\ell$ gives distances to~$v_*$ via the formula
\begin{equation}
   \label{eq:10}
d_\bm(v,v_*)=\ell(v)-\min_{\sV_\circ(\bt)}\ell +1\, .
\end{equation}
As a result (and with the help of the Euler characteristic formula),
note that $N^\sV+1$, $N^\sF$ and $N^\sE$ respectively correspond to
the number of vertices, internal faces, and edges of~$\bm$ --- this
explains the notation.

\subsection{Random mobiles}\label{sec:random-mobiles}

We now show how to represent the pointed Boltzmann measures
$\W^\bullet_l$ of Section~\ref{secZq} in terms of random trees, via the BDG bijection.  Let
$\mu_\circ$ be the geometric distribution with parameter $1/\CZ_q$,
given by
$$\mu_\circ(k)=\frac{1}{\CZ_q}\left(1-\frac{1}{\CZ_q}\right)^k\,
,\qquad k\geq 0\, .$$
Let also 
$$\mu_\bullet(k)=\frac{\CZ_q^k\binom{2k+1}{k}\,q_{k+1}}{f_q(\CZ_q)}\,
\qquad k\geq 0\, .$$
Let $\M_l$ be the law of a two-type
Bienaym\'e--Galton--Watson forest, with $l$ independent tree
components, and in which even generations (white vertices) use the offspring
distribution~$\mu_\circ$, while odd generations (black vertices) use the offspring
distribution~$\mu_\bullet$. Formally, we let $\M_l=(\M_1)^{\otimes l}$
where $\M_1$ is defined by 
$$\M_1(\{\bt\})=\prod_{u\in
  \sV_\circ(\bt)}\mu_\circ(k_u(\bt))\prod_{u\in
  \sV_\bullet(\bt)}\mu_\bullet(k_u(\bt))\, $$ 
for every tree $\bt$, where
$k_u(\bt)$ is the number of children of $u$ in 
$\bt$. 
Finally, given a forest with law $\M_l$, the white vertices carry
random integer labels with the following law. Let $\xi_1$, $\xi_2$,
\ldots be a sequence of i.i.d.\ random variables with shifted
geometric($1/2$) distributions
$$\P(\xi_i=l)=2^{-l-2}\, , \qquad l\geq -1\, ,\ i\ge 1\, ,$$
and let $(Y_1,\ldots,Y_k)$ be distributed as the partial sums
$(\xi_1,\xi_1+\xi_2,\ldots,\xi_1+\ldots+\xi_k)$ conditionally given
$\xi_1+\ldots+\xi_{k+1}=0$. We say that $(Y_1,\ldots,Y_k)$ is a
\emph{discrete bridge with shifted geometric steps}, and we let~$\nu_k$ be
the law of this random vector. It is simple to see that, if~$\nu_k^0$
is the uniform distribution on $$\left\{(x_1,\ldots,x_k,x_{k+1})\in
\{-1,0,1,2,\ldots\}^{k+1}\, :\, \sum_{i=1}^{k+1}x_i=0\right\}\, ,$$
then~$\nu_k$ is the image measure of~$\nu_k^0$ under 
$(x_1,\ldots,x_{k+1})\mapsto (\sum_{i=1}^jx_j,\,1\leq j\leq k)$. 

Conditionally given the tree, if~$u$ is a black vertex with parent
$u_0$ and children $u_1$, $u_2$, \ldots, $u_k$, then the law~$\nu_k$ of the
label differences $(\ell(u_i)-\ell(u_0),\,1\leq i\leq k)$ is given by
$\nu_k$, while those label differences are independent as $u$ ranges
over all black vertices. Finally, the labels of the roots
$\rho_1$, \ldots, $\rho_l$ of the forest 
have same law as $(0,Y_1,\ldots,Y_l)$, where
$(Y_1,\ldots,Y_l)$ has law $\nu_{l}$. These specify entirely
the law of the labels, and in fact, one sees that labels are uniform
among all admissible labelings of the forest, in which the root~$\rho_1$ of the
first tree carries label~$0$. For simplicity, we still denote by $\M_l$ the law of
forest of well-labeled mobiles thus obtained. 

For $\sS\in \{\sV,\sE,\sF\}$, we let $\bB^{\bullet,\sS}_{l,n}$ be the set of pointed maps $(\bm,v_*)\in
\bB^\bullet$ such that $\bm\in \bB^{\sS}_{l,n}$ and we define
\begin{equation}\label{WpointS}
\W^{\bullet,\sS}_{l,n}=W^\bullet\big(\cdot\, \big|\, \bB^{\bullet,\sS}_{l,n}\big)\,,
\end{equation}
where $W^\bullet$ was defined by~\eqref{Wpoint}.

\begin{prp}
  \label{sec:boutt-di-franc-1}
  Let $q$ be an admissible sequence, and $l\geq 1$. Then the image of~$\M_l$ under the Bouttier--Di Francesco--Guitter
  bijection is, after uniform re-rooting on the boundary, the probability measure~$\W^{\bullet}_l$.

For $\sS\in \{\sV,\sE,\sF\}$, the same statement holds if we replace both~$\M_l$ with $\M_l(\,\cdot\,|\, N^\sS=n)$) and~$\W^{\bullet}_l$ with~$\W^{\bullet,\sS}_{l,n}$.
\end{prp}

This is proved by following the same steps as in
\cite[Proposition~7]{MaMi07} and by applying a straightforward analog of Lemma~\ref{awayfrom}; we omit the details. 
At this point, we can prove Lemma \ref{sec:boltzm-rand-maps}, which
describes the set $\A^\sS(q)$ of pairs $(l,n)$ such that
$\W(\bB^\sS_{l,n})>0$ or, equivalently, such that $\M_l(N^\sS=n) >0$. 

\begin{proof}[Proof of Lemma \ref{sec:boltzm-rand-maps}]
  Let us fix the symbol $\sS$. By Proposition~2.2 in
  \cite{stephenson14}, under the law $\M_1$, there exist two constants
  $\alpha$, $h$ such that the support of $N^\sS$ is included in
  $\alpha+h\Z_+$, and moreover, for every $m$ large enough,
  $\M_1(N^\sS=\alpha+hm)>0$. In particular, there exists $\beta$ such
  that $\M_1(N^\sS=\beta+hm)>0$ for every $m\geq 0$. This means that
  the support of the law of~$N^\sS$ under~$\M_1$ is equal to
  $R\cup(\beta+h\Z_+)$, for some
  $R\subseteq\{0,1,\ldots,\beta-1\}$. From this, we immediately deduce
  the similar result for forests under the distribution~$\M_l$. 
  Namely, the support of~$N^\sS$ under~$\M_l$ is equal to
  $R_l\cup(\beta l+h\Z_+)$, for some $R_l\subseteq\{0,1,\ldots,\beta
  l-1\}$. From this observation and Proposition~\ref{sec:boutt-di-franc-1}, 
  using the remark at the end of the
  preceding section that the image of~$N^\sS$
  under the BDG bijection is $|\sS| - \ind_{\{\sS=\sV\}}$, we obtain that the support of
  the law of $|\sS(\bm)|-\ind_{\{\sS=\sV\}}$ under $\W_l^\bullet$ (or under $\W_l$ by
  the absolute continuity relation~\eqref{eq:16}) is equal to 
$$R_l\cup (\beta l+h\Z_+)\, .$$
The result follows immediately from this, since the explicit form of
$h$ was computed in Section~6.3.1 of~\cite{stephenson14}.
\end{proof}

Again, in all the following, when considering pairs $(l,n)$ where $l$
corresponds to the boundary length of a map, and~$n$ to its size (measured
with respect to the symbol~$\sS$), it will always be implicitly
assumed that $(l,n)\in \A^\sS(q)$, which by Lemma
\ref{sec:boltzm-rand-maps} means that, up to finitely many exceptions, 
$$n=\beta^\sS l\quad[\mathrm{mod }\,  h^\sS]\, .$$

\section{Convergence of the encoding processes}\label{sec:conv-encod-proc}

Let us now consider an infinite forest~$F$ with distribution
$\M_\infty=(\M_1)^{\otimes \N}$. With it, we associate several
exploration processes. Let $v_0$, $v_1$, $v_2$\ldots\ denote the
vertices of~$F$ (black or white), listed in depth-first order, tree by
tree. Let~$H$ be the so-called \emph{height process} associated
with~$F$, that is, $H(i)$ denotes the distance between the
vertex~$v_i$ and the root of the tree to which it belongs. For $i\ge
0$, we denote by $\hat\ell(i)$ the label of~$v_i$, as well as
$\hat\ell^{\mathbf{0}}(i)=\hat\ell(i)-\ell(\rho_{(i)})$, where 
$\rho_{(i)}$ is the root of the tree to which~$v_i$ belongs. Note that
this notion of label process differs from the one introduced during
Section~\ref{sec:conv-brown-disk}; we use the notation with a hat in
order to avoid confusion. Recall also that, under~$\M_\infty$, the
labels are normalized in such a way that the root of the first tree
gets label~$0$, so that the process~$\hat\ell$ is defined without
ambiguity. Finally, let $\Upsilon(i)$ be the number of fully explored
trees at time~$i$, that is, $\Upsilon(i)+1=p$ whenever $v_i$ belongs
to the $p$-th tree of~$F$, that is $\rho_{(i)}=\rho_p$. We also let
$$\tau_l=\inf\{i\geq 0:\Upsilon(i)=l\}$$
be the number of (black or white) vertices in the first~$l$ trees of
the forest. Note for instance that, with the notation of
Section~\ref{sec:boutt-di-franc}, one has $N^\sE=\tau_l$ under the
law~$\M_l$.

\subsection{Convergence for an infinite forest}\label{sec:conv-an-infin}

A key result is the following. Recall that $\CZ_q$ is given by~\eqref{eq:1} and $\rho_q=2+\CZ_q^3f_q''(\CZ_q)$. Define
$$\sigma_q^2=\frac{\CZ_q \rho_q}{4}\, ,\qquad \sie^2=\frac{\rho_q}{\CZ_q}\, .$$
\begin{prp}
  \label{sec:boutt-di-franc-2}
The following joint convergence holds in distribution in $\mathcal{C}(\R_+,\R)$ under
$\M_\infty$:
$$\left(\frac{H({m\,\cdot})}{\sqrt{m}},\frac{\Upsilon({m\,\cdot})}{\sqrt{m}},\frac{\hat\ell^{\mathbf{0}}({m\,\cdot})}{m^{1/4}}\right)\build\longrightarrow_{m\to\infty}^{(d)}\left(\frac{2}{\sigma_q}(X-\underline{X}),-\sie\,\underline{X},\sqrt{\frac{2\sie}{3}}\Zo\right)$$
where $(X_t,t\geq 0)$ is a standard Brownian motion,
$\underline{X}_t=\inf_{0\leq s\leq t}X_s$ and $\Zo$ is the Brownian snake with
driving process $X-\underline{X}$, introduced in Section~\ref{sec:snakes}.
\end{prp}

\begin{proof}
  We note that the two-type branching process with offspring
  distributions $\mu_\circ$, $\mu_\bullet$ and alternating types is a
  critical branching process, in which the offspring distributions
  have small exponential moments (this is the place where we use the
  fact that~$q$ is regular critical), as discussed in Proposition 7
  of \cite{MaMi07}. Furthermore, the spatial displacements with distribution~$\nu_k$ are
  centered and carried by $[-k,k]^k$ respectively. In particular,
  they have moments of all orders, which grow at most polynomially, in
  the sense that for every $a>0$,
$$\langle\nu_k,|\cdot|^a\rangle=O(k^{a})\, ,$$
where $|\cdot|$ is the Euclidean norm in $\R^k$. 
This is exactly what is needed to apply Theorems~1 and~3 in
  \cite{miergwmulti}, which in our particular context stipulate that 
$$\left(\frac{H({m\,\cdot})}{\sqrt{m}},\frac{\Upsilon({m\,\cdot})}{\sqrt{m}},\frac{\hat\ell^{\mathbf{0}}({m\,\cdot})}{m^{1/4}}\right)\build\longrightarrow_{m\to\infty}^{(d)}\left(\frac{2}{\sigma}(X-\underline{X}),-\frac{\sigma}{b_\circ}\,\underline{X},\Sigma\sqrt{\frac{2}{\sigma}}\,\Zo\right)\,
,$$
where the constants $\sigma$, $b_\circ$ and $\Sigma$ are defined in
the following way. The mean matrix of the two-type Galton--Watson
process under consideration is given by $$\left(
    \begin{array}{cc}
0 & m_\circ\\
m_\bullet & 0      
    \end{array}
  \right)\, ,$$ where $m_\circ$ is the mean of $\mu_\circ$, and
  $m_\bullet$ is the mean of $\mu_\bullet$. Note that $m_\bullet=m_\circ^{-1}$ as an immediate consequence of the fact that~$q$ is regular critical. This matrix
  admits a left invariant vector $\mathbf{a}=(a_\circ,a_\bullet)$
  normalized to be a probability, namely $a_\circ=(1+m_\circ)^{-1}$
  and $a_\bullet=(1+m_\bullet)^{-1}$, and a right invariant
  $\mathbf{b}=(b_\circ,b_\bullet)$ normalized in such a way that the scalar product
  $\mathbf{a}\cdot\mathbf{b}=1$, namely $b_\circ=(1+m_\circ)/2$ and
  $b_\bullet=(1+m_\bullet)/2$. Finally, with $(\mu_\circ,\mu_\bullet)$, one can
  associate a quadratic function $\bQ:\R^2\to\R^2$ given by
$$\bQ(x_\circ,x_\bullet)=((\sigma_\circ^2+m_\circ(m_\circ-1))\,x_\bullet^2\, 
,
(\sigma_\bullet^2+m_\bullet(m_\bullet-1))\,x_\circ^2)\,
,$$
where $\sigma_\circ^2$ and $\sigma_\bullet^2$ are the variances of
$\mu_\circ$ and $\mu_\bullet$. Then $\sigma^2$ is given by the scalar product
$$\sigma^2=\mathbf{a}\cdot\mathbf{Q}(\mathbf{b})\, .$$
Finally, $\Sigma^2$ is given by the formula
$$\Sigma^2=\frac{1}{2}\sum_{k\geq 1}\frac{\mu_\bullet(k)}{m_\bullet}(\Sigma_\bullet^k)^2$$
where $(\Sigma_\bullet^k)^2=\langle\nu_k,|\cdot|^2\rangle=k(k+1)/3$,
as can be checked in \cite{MaMi07}. After
computations, which have been performed in Section 3.2 of \cite{MaMi07},
one obtains in particular
$$m_\circ=\CZ_q-1\, ,\quad b_\circ=\frac{\CZ_q}{2}\, ,\quad\sigma^2=\frac{\CZ_q\rho_q}{4}\, ,\quad
\Sigma^2=\frac{\rho_q}{6}\, .$$
The conclusion follows. 
\end{proof}

We are also going to need the following fact. For every $m\geq 1$, let
\begin{equation}\label{lambd}
\begin{split}
\Lambda^\sV(m)=\big|\big\{i\in \{0,1,\ldots,m-1\}:v_i\in \sV_\circ(\bff)\big\}\big|\,,\\
\Lambda^\sF(m)=\big|\big\{i\in \{0,1,\ldots,m-1\}:v_i\in \sV_\bullet(\bff)\big\}\big|\,,
\end{split}
\end{equation}
be respectively the number of white vertices and the number of black vertices among the first~$m$
vertices of~$F$ in depth-first order. For convenience, we also let $\Lambda^\sV(m)=m$ (the number of vertices of either type), so that $\Lambda^\sS$ makes sense for every $\sS\in \{\sV,\sE,\sF\}$. Define
\begin{equation}\label{astar}
a_\sV=\CZ_q^{-1}\,,\qquad a_\sF=1-\CZ_q^{-1}\,,\qquad a_{\sE}=1\,.
\end{equation}
The first two quantities are the ones that appeared in the proof of Proposition~\ref{sec:boutt-di-franc-2}, under the notation $a_\sV=a_\circ$ and $a_\sF=a_\bullet$. (Recall that, through the BDG bijection, $\sV$ correspond essentially to white vertices, $\sF$ to black vertices and~$\sE$ to edges of the mobile, which are in direct bijection with the set of vertices of both colors.)

In the following
statement and later, the notation $\mathrm{oe}(n)$ stands for a
quantity that is bounded from above by $c\exp(-c'n^{c''})$ for three
positive constants $c$, $c'$, $c''$, uniformly in~$n$.

\begin{prp}
  \label{sec:conv-encod-proc-2}
Fix $\sS\in \{\sV,\sE,\sF\}$. Then it holds that
$$\left(\frac{\Lambda^\sS(m\,\cdot)}{m}\right)\build\longrightarrow_{m\to\infty}^{}
(a_\sS\, t,\, t\geq 0)$$ in probability under $\M_\infty$ for the
uniform topology over compact subintervals of $\R_+$. More precisely,
for every $K>0$, one has the concentration result
$$\M_\infty\left(\max_{0\leq k\leq Km}|\Lambda^\sS(k) - a_\sS\, k|>m^{3/4}\right)=\mathrm{oe}(m)\,.$$
\end{prp}

\begin{proof}
  The result is obvious for $\sS=\sE$, so that we suppose $\sS\in\{\sV,\sF\}$. We first note that, since $\Lambda^\sS(k)\leq k$, it suffices to
  prove the same bound with the maximum restricted over indices $k\in
  [m^{1/2},Km]$. Now \cite[Proposition 6 (ii)]{miergwmulti} shows
  that if $G_k^\sS$ is the number of vertices in depth-first
  order (of either type) that have been visited
  before the $k$-th vertex of type~$\sS$ (white if $\sS=\sV$, black if $\sS=\sF$), then $\M_\infty(|G_k^\sS - a_\sS^{-1}\,
  k|>k^{3/4})=\mathrm{oe}(k)$. Now $|\Lambda^\sS(k)-a_\sS\,
  k|>m^{3/4}$ implies that $G^\sS_{a_\sS k+m^{3/4}}\leq k$ or
  $G^\sS_{(a_\sS k-m^{3/4})_+}\geq k$, the probability of which is
  bounded from above by
  \begin{multline*}
   \M_\infty\big(|G^\sS_{a_\sS k+m^{3/4}}-a_\sS^{-1}(a_\sS\, k+m^{3/4})|\geq a_\sS^{-1}m^{3/4}\big)\\
	+\M_\infty\big(|G^\sS_{(a_\sS k-m^{3/4})_+}-a_\sS^{-1}(a_\sS\, k-m^{3/4})_+|\geq a_\sS^{-1}m^{3/4}\big)\, .
\end{multline*}
Taking the maximum over all $k\in
[m^{1/2},Km]\cap \Z$, we see that this quantity is $\mathrm{oe}(m)$,
as claimed. 
\end{proof}

\subsection{Convergence for a conditioned forest}\label{sec:conv-cond-forest}

We now want a conditioned version of Proposition~\ref{sec:conv-encod-proc-2}. We are
going to need the following estimates. Recall the definition~\eqref{hS} of~$h^\sS$, the definition~\eqref{jLA} of~$j_L(A)$, and define $Q^\sS(l,n)=\M_l(N^\sS=n)$. We will also need the notation
$$\tau^\sS_l=\Lambda^\sS(\tau_l)\,.$$
In words, $\tau^\sE_l=\tau_l$ is the number of vertices in the~$l$ first trees of the forest, while $\tau^\sV_l$ (resp.~$\tau^\sF_l$) is the number of white
(resp.\ black) vertices in these trees. 

\begin{lmm}
  \label{sec:conv-encod-proc-1}
Let $\sS\in \{\sV,\sE,\sF\}$. Then
$$\sup_{n \in \A_l^\sS}\left|\, l^2Q^\sS(l,n)-
h^\sS\, j_{1/\sis}\left(\frac{n}{l^2}\right)\right|\build\longrightarrow_{n\to\infty}^{}0\,
.$$
\end{lmm}

\begin{proof}
  Suppose first that $\sS=\sE$. In this case, a consequence of the
  convergence of the second component in Proposition
  \ref{sec:boutt-di-franc-2} is that 
$$\frac{\tau_l}{l^2}\build\longrightarrow_{l\to\infty}^{(d)}T_{1/\sie}\,
,$$
where we recall that $T_{1/\sie}=\inf\{t\geq 0:\underline{X}_t=-1/\sie\}$ is a.s.\
a continuous function of $X$ under the Wiener measure, due to the fact
that $\underline{X}_{(T_{1/\sie}-\eps)_+}>-1/\sie>\underline{X}_{T_{1/\sie}+\eps}$ a.s.\ for
every $\eps>0$. Moreover, $\tau_l$ under $\M_\infty$ is the sum of~$l$
i.i.d.\ random variables with same law as~$\tau_1$: these are given by the
number of vertices of the first $l$ trees in the infinite forest of
independent random mobiles. Since it is well known that $T_{1/\sie}$ follows a stable distribution
with index $1/2$, with a density given by $j_{1/\sie}$, we conclude that
$\tau_1$ under $\M_\infty$ is in the domain of attraction of this
law. The statement is then a consequence of the local limit theorem
for stable random variables \cite[Theorem~8.4.1]{BGT}. 

The remaining two
cases $\sS\in \{\sV,\sF\}$ are now direct consequences of the case $\sS=\sE$ and of
Proposition \ref{sec:conv-encod-proc-2}, which together imply that we have
\begin{equation}
  \label{eq:4}
  \frac{\tau_l^\sS}{l^2}=\frac{\Lambda^\sS(\tau_l)}{l^2}=\frac{a_\sS\,\tau_l\,(1+o_\P(1))}{l^2}\,,
\end{equation}
where $o_\P(1)$ denotes a quantity that converges to $0$ in
probability. This yields that 
$$\frac{\tau_l^\sS}{l^2}\build\longrightarrow_{l\to\infty}^{(d)}a_\sS
T_{1/\sie}\build=_{}^{(d)}T_{\sqrt{a_\sS}/\sie}=T_{1/\sis}\,,$$
as can be checked using~\eqref{eq:3} and~\eqref{astar}. The conclusion follows by the same
arguments as in the case $\sS=\sE$. 
 \end{proof}

In this section, it is convenient to consider processes whose total
duration is not fixed. We let $\WW$ be the set of real-valued
continuous functions $f$ defined on an interval of the form
$[0,\zeta]$ for some $\zeta=\zeta(f)\in [0,\infty)$. This set is
endowed with the distance
$$\mathrm{dist}(f,g)=\|f(\cdot\wedge \zeta(f))-g(\cdot\wedge
\zeta(g))\|_\infty+|\zeta(f)-\zeta(g)|\,$$ which makes it a complete
separable metric space. For instance, the height process $H$ under the
law $\M_l$ is a function with duration $\zeta(H)=N^\sE$.

Recall the definition of~$\SSL$ given right after~\eqref{eq:3}.

\begin{prp}
  \label{sec:conv-brown-disk-5}
Let $\sS\in \{\sV,\sE,\sF\}$ and $(l_k,n_k)_{k\geq 0}\in \SSL$ for some $L>0$. Then, under
  $\M_{l_k}(\,\cdot\, |\, N^\sS=n_k)$, it holds that
$$\left(\frac{H(a_\sS^{-1}{n_k}\,\cdot)}{\sqrt{{n_k}}},\frac{\Upsilon(a_\sS^{-1}{n_k}\,\cdot)}{\sqrt{{n_k}}},\frac{\hat\ell^{\mathbf{0}}(a_\sS^{-1} {n_k}\,\cdot)}{{n_k}^{1/4}}\right)\build\longrightarrow_{k\to\infty}^{(d)}\left(\frac{2}{\sqrt{a_\sS}\,\sigma_q}(X-\underline{X}),-\sis\,\underline{X},\sqrt{\frac{2\sis}{3}}\Zo\right)\,
,$$ in distribution in the space $\WW^3$ where, in the limit, $X$, $\Zo$ are
understood under the law~$\mathbb{F}^1_L$ defined in Section~\ref{sec:snakes}.
\end{prp} 

\begin{proof}
  For simplicity, let $\Xi_k$ denote the triple appearing in the
  left-hand side of the convergence. Denote by $\mathcal{F}_{p}$ the
  $\sigma$-field generated by the~$p$ first trees of the generic
  (canonical process) infinite forest
  $\bff=(\bt_1,\bt_2,\ldots)$, together with their labels. Let~$G$ 
  be measurable with respect to $\mathcal{F}_{l_k'}$, with $l_k'<
  l_k$. Then we have
  \begin{equation}
    \label{eq:5}
    \M_{l_k}\left[G\, \big|\, N^\sS=n_k\right]=\M_\infty\left[G
  \, \frac{Q^\sS(l_k-l'_k,n_k-\tau^\sS_{l_k'})}{Q^\sS(l_k,n_k)}\right],
\end{equation}
where it should be understood that the quantity in the expectation is~$0$ whenever $\tau^\sS_{l_k'}> n_k$. 
Now, we impose that $G=\Phi(\Xi_k')$ is a continuous, bounded function
of the triple of processes
$$\Xi_k'=\left(\frac{H(a_\sS^{-1}n_k\cdot\wedge
    \tau_{l'_k})}{\sqrt{n_k}},\frac{\Upsilon({a_\sS^{-1}n_k\cdot\wedge
      \tau_{l'_k}})}{\sqrt{n_k}},\frac{{\hat\ell}^{\mathbf{0}}({a_\sS^{-1}n_k\cdot
    \wedge \tau_{l'_k}})}{n_k^{1/4}}\right)\, ,$$
where we assume that $l'_k\sim L'\sis\sqrt{n_k}$ for some $0<L'<L$.
Proposition \ref{sec:boutt-di-franc-2} shows the convergence in
distribution
$$\left(\frac{H({a_\sS^{-1}{n_k}\,\cdot})}{\sqrt{a_\sS^{-1}{n_k}}},\frac{\Upsilon({a_\sS^{-1}{n_k}\,\cdot})}{\sqrt{a_\sS^{-1}{n_k}}},\frac{\hat\ell^{\mathbf{0}}({a_\sS^{-1}{n_k}\,\cdot})}{\big(a_\sS^{-1}{n_k}\big)^{1/4}}\right)\build\longrightarrow_{{k}\to\infty}^{(d)}
\left(\frac{2}{\sigma_q}(X-\underline{X}),-\sie\,\underline{X},\sqrt{\frac{2\sie}{3}}\,\Zo\right)\,
,$$ where the limit is understood under the law $\P$. Using the
convergence of the second component, the asymptotic behavior of $l'_k$
and the fact that $a_\sS\sis^2=\sie^2$, it follows that
$$\frac{\tau_{l'_k}}{a_\sS^{-1}n_k}\build\longrightarrow_{k\to\infty}^{(d)}T_{L'}\,
,$$ and that this convergence holds jointly with the previous
one. From this, it follows that~$\Xi_k'$ converges in distribution
under $\M_\infty$ to the triple
$$\Xi_\infty'=\left(\frac{2}{\sqrt{a_\sS}\,\sigma_q}(X-\underline{X})_{\,\cdot\,\wedge
    T_{L'}},-\sis\,\underline{X}_{\,\cdot\,\wedge
    T_{L'}},\sqrt{\frac{2\sis}{3}}\Zo_{\,\cdot\,\wedge T_{L'}}\right)$$ and an
application of~\eqref{eq:4} implies that $\tau^\sS_{l_k'}/{n_k}\to T_{L'}$
jointly with the above convergence. By the Skorokhod representation
theorem, we may assume that the probability space is chosen in such a
way that these convergences hold in the almost-sure sense, and then~\eqref{eq:5} together with Lemma~\ref{sec:conv-encod-proc-1} implies
that
$$\M_{l_k}\left[\Phi(\Xi'_k)\, |\, N^\sS={n_k}\right]\build\longrightarrow_{k\to\infty}^{(d)}
\E\left[\Phi(\Xi'_\infty)\,\frac{L^2}{(L-{L'})^2}\frac{\, j_{1/\sis}\bigg(\dfrac{(1-T_{L'})}{\sis^2\,(L-{L'})^2}\bigg)}{\,j_{1/\sis}\bigg(\dfrac{1}{\sis^2 L^2}\bigg)}\right]$$
and the limit can be re-expressed as
$$\E\left[\Phi(\Xi'_\infty)\,\frac{j_{L-{L'}}(1-T_{L'})}{j_L(1)}\right]=\mathbb{F}^1_L[\Phi(\Xi'_\infty)]\,
.$$
By definition, $\Phi(\Xi'_\infty)$ is
$\mathcal{G}_{T_{L'}}$-measurable and, by Galmarino's test, we have
$\mathcal{G}_{T_{L'}}=\sigma(X_{\cdot\,\wedge T_{L'}},\Zo_{\,\cdot\,\wedge
  T_{L'}})$, so that, if it exits, the limit of the triple considered in
the statement of the proposition necessarily has the claimed law, by
virtue of Proposition~\ref{sec:first-pass-bridg-1}. 

To conclude the proof, it remains to prove that the laws of the
processes under consideration are tight in~$\WW^3$. We can argue as
follows. Let~$f$ be a continuous function defined on an interval $I$,
and $J\subseteq I$ be a subinterval of $I$. Denote by
$$\omega(f,\delta,J)=\sup_{s,t\in J,\,
  	|t-s|\leq \delta}|f(t)-f(s)|$$ the modulus of continuity of $f$
restricted to $J$, and let $\omega(f,\delta)=\omega(f,\delta,I)$.

Here, let $Y_k$ denote either of the components of $\Xi_k$. 
Then, under $\M_{l_k}(\,\cdot\, |\, N^\sS={n_k})$,
\begin{align*}
  \omega(Y_k,\delta)
\leq 
 \omega\big(Y_k,\delta,[0,a_\sS\tau_{l'_k}/{n_k}]\big)
 +  \omega\big(Y_k,\delta,[a_\sS\tau_{l'_k}/{n_k},a_\sS\tau_{l_k}/{n_k}]\big)
\end{align*}
while the second component has same distribution as
$$
\omega\big(Y_k,\delta,[0,a_\sS\tau_{l_k-l_k'}/{n_k}]\big)$$
by a symmetry argument (the $l_k$ trees of the labeled forest are
exchangeable). Choosing $l'_k\sim l_k/2$, we obtain from the
convergence of $\Xi_k'$ (for ${L'}=L/2$) that 
\begin{align*}
  \limsup_{n\to\infty}\M_{l_k}\big(\omega(Y_k,\delta)\geq \eps\big)\hspace{-40mm}\\
&\leq \limsup_{n\to\infty}
\M_{l_k}\big(\omega\big(Y_k,\delta,[0,a_\sS\tau_{l_k'}/{n_k}]\big)\geq
\eps/2\big)+\M_{l_k}\big(\omega\big(Y_k,\delta,[0,a_\sS\tau_{l_k-l_k'}/{n_k}]\big)\geq \eps/2\big)\\
&\leq 2\,\mathbb{F}^1_L\big(\omega\big(Y_\infty,\delta,[0,T_{L/2}]\big)\geq \eps/2\big)\, ,
\end{align*}
where $Y_\infty$ is the limit of $Y_k$ (for instance
$Y_\infty=2(X-\underline{X})/\sqrt{a_\sS}\,\sigma_q$ if $Y_k$ is the
first component of $\Xi_k$). This quantity
converges to $0$ as $\delta\to0$, for any fixed $\eps>0$. From this,
it is an immediate consequence of the Ascoli--Arzela theorem that the
laws of $\Xi_k$, $k\geq 1$ are relatively compact in~$\WW^3$.
\end{proof}

\begin{crl}
Let $\sS\in \{\sV,\sE,\sF\}$ and $(l_k,n_k)_{k\geq 0}\in \SSL$ for some $L>0$. Then, under
  $\M_{l_k}(\,\cdot\, |\, N^\sS=n_k)$, it holds that
$$\left(\frac{H(a_\sS^{-1}{n_k}\,\cdot)}{\sqrt{{n_k}}},\frac{\Upsilon(a_\sS^{-1}{n_k}\,\cdot)}{\sqrt{{n_k}}},\frac{\hat\ell(a_\sS^{-1} {n_k}\,\cdot)}{{n_k}^{1/4}}\right)\build\longrightarrow_{k\to\infty}^{(d)}\left(\frac{2}{\sqrt{a_\sS}\,\sigma_q}(X-\underline{X}),-\sis\,\underline{X},\sqrt{\frac{2\sis}{3}}Z\right)\,
,$$ in distribution in the space $\WW^3$ where, in the limit, $X$, $Z$ are understood under the law~$\mathbb{F}^1_L$, defined in Section~\ref{sec:snakes}.
\end{crl}

\begin{proof}
  It suffices to apply the preceding proposition, noting that one can
  get the following representation for the label process $\hat\ell$ in terms
 of $\hat\ell^{\mathbf{0}}$: 
 $$\hat\ell(i)=\hat\ell^{\mathbf{0}}(i)+\mathrm{B}({\Upsilon(i)})\, ,$$ 
 where $\mathrm{B}$ is a discrete bridge with shifted geometric step,
 with law $\nu_{l_k}$ defined in Section~\ref{sec:random-mobiles}. It holds that, under our hypotheses,
 $\mathrm{B}({\sis\sqrt{n_k}\,\cdot})/\sqrt{2\sis}n_k^{1/4}$ converges in
 distribution to a Brownian bridge~$\mathrm{b}$
 with duration~$L$ (see \cite[Proposition~7]{bettinelli11b}). Putting things together, we obtain that
$$\frac{\hat\ell({a_\sS^{-1}{n_k}\,\cdot})}{{n_k}^{1/4}}=\frac{\hat\ell^{\mathbf{0}}({a_\sS^{-1}{n_k}\,\cdot})}{{n_k}^{1/4}}+\frac{\mathrm{B}\big({\sis\sqrt{{n_k}}\,
    (\Upsilon({a_\sS^{-1}{n_k}\,\cdot})/\sis\sqrt{{n_k}})}\big)}{{n_k}^{1/4}}\, ,$$
which converges in distribution to
$s\mapsto\sqrt{2\sis/3}\,\Zo_s+\sqrt{2\sis}\,\mathrm{b}_{-\underline{X}_s}$, jointly with the
rescaled processes $H$ and $\Upsilon$. By definition~\eqref{defZ}, this yields the result. 
\end{proof}

Finally, we note that the convergence of $\Lambda^\sS$ stated in
Proposition \ref{sec:conv-encod-proc-2} still holds under conditioned
forests. Indeed, since the conditioning event $\{N^\sS=n_k\}$ has a
probability $Q^\sS(l_k,n_k)=\Theta(l_k^{-2})=\Theta(n_k^{-1})$
by Lemma~\ref{sec:conv-encod-proc-1}, we obtain that for any $c'>0$, 
\begin{multline*}
\M_{l_k}\left(\max_{0\leq i\leq N^\sE}|\Lambda^\sS(i)-a_\sS\,i|>n_k^{3/4}\, |\,
N^\sS=n_k\right)\\
\leq cn_k\,\M_\infty\left(\max_{0\leq i\leq
  c'n_k}|\Lambda^\sS(i)-a_\sS\,i|>n_k^{3/4}\right)+\M_{l_k}(N^\sE>c'n_k\, |\, N^\sS=n_k)\, .
\end{multline*}
for some constant $c>0$. The first term is $\mathrm{oe}(n_k)$ by
Proposition~\ref{sec:conv-encod-proc-2}. The second term is
equal to~$0$ if $\sS=\sE$ and $c'>1$. If $\sS=\sV$, it can
be bounded by
$$cn_k\,\M_\infty(\Lambda^\sV(c'n_k)\leq n_k)=\mathrm{oe}(n_k)\, ,$$
as soon as $c'$ is chosen strictly larger than $a_\sV^{-1}$, again by Proposition~\ref{sec:conv-encod-proc-2}. The
argument is the same if $\sS=\sF$. In particular, as $\Lambda^\sS(N^\sE)=N^\sS$, this implies that,
under $\M_{l_k}(\,\cdot\, |\, N^\sS=n_k)$, one has 
$$\frac{N^\sS}{N^\sE}\longrightarrow a_\sS$$
in probability as $k\to\infty$. This implies the following
reformulation and refinement of the preceding corollary. 

\begin{crl}
  \label{sec:conv-encod-proc-3}
Let $\sS\in \{\sV,\sE,\sF\}$ and $(l_k,n_k)_{k\geq 0}\in \SSL$ for some $L>0$. Then, under
  $\M_{l_k}(\,\cdot\, |\, N^\sS=n_k)$, it holds that
$$\left(\frac{H(N^\sE\,\cdot)}{\sqrt{{n_k}}},\frac{\Upsilon(N^\sE\,\cdot)}{\sqrt{{n_k}}},\frac{\hat\ell(N^\sE\,\cdot)}{{n_k}^{1/4}}
\right)\build\longrightarrow_{k\to\infty}^{(d)}
\left(\frac{2}{\sqrt{a_\sS}\,\sigma_q}(X-\underline{X}),-\sis\,\underline{X},\sqrt{\frac{2\sis}{3}}Z\right)\,
,$$ in distribution in the space $\WW^3$, where, in the limit, $X$, $Z$ are understood under the law~$\mathbb{F}^1_L$. Moreover, one has, still under under
  $\M_{l_k}(\,\cdot\, |\, N^\sS=n_k)$, 
$$\left(\frac{\Lambda^\sS(N^\sE t)}{N^\sS},\,0\leq t\leq
  1\right)\build\longrightarrow_{k\to\infty}^{}\mathrm{Id}_{[0,1]}$$ 
in probability. 
\end{crl}

\subsection{Convergence of the white contour and label processes}\label{sec:conv-white-cont}

Finally, we consider a variant of the latest corollary where the height
process is replaced by the slightly more convenient contour
processes. First, given the forest~$F$ with possibly infinitely many trees,
we let~$\bar C$ be its contour process, defined as follows. We first add edges linking the roots of consecutive trees as we did before (see for instance the left part of Figure~\ref{coding}). We then let $c_0$, $c_1$\dots\ be the list of corners of the trees, arranged in contour order. The purpose of the added edges linking the roots is to ``split in two'' the root corners of the trees: a tree with~$p$ edges will thus have $2p$+1 corners. Finally, we let $\bar C(i)$ be the distance between~$c_i$ and the root of the tree to which it belongs. Note that there is slight difference with the contour
function~$C$ defined in Section~\ref{sec:subs}, where a downstep was separating the contours of successive trees, instead of the horizontal step we have here. It is a
standard fact, proved in \cite[Chapter~2.4]{duqleg02}, that the contour
and height process of a forest are asymptotically similar in the following
sense. First, let $f(i)+1$ be the number of distinct vertices incident
to the corners $c_0$, $c_1$, \ldots, $c_i$ (these vertices being $v_0$, \ldots, $v_{f(i)}$). Then it holds that 
$f(i)\leq i$ for every $i\geq 0$, and for every $m\geq 0$, one has 
$$\sup_{0\leq i\leq m}\big|\bar C(i)-H(f(i))\big|\leq 1+\sup_{0\leq i\leq
  m}\big|H(i+1)-H(i)\big|$$
and 
\begin{equation}
  \label{eq:17}
  \max_{0\leq i\leq m}\left|f(i)-\frac{i}{2}\right|\leq
  1+\frac{1}{2}\max_{0\leq i\leq m}H_i\, .
\end{equation}
From this, and the convergence of the rescaled height process stated
in Corollary~\ref{sec:conv-encod-proc-3}, it
follows easily that under the same hypotheses, 
$$\frac{\bar C(2N^\sE\,\cdot)}{\sqrt{n_k}}\build\longrightarrow_{k\to\infty}^{(d)}\frac{2}{\sqrt{a_\sS}\,
    \sigma_q}(X-\underline{X})\, ,$$ where $X$ is understood under
$\mathbb{F}^1_L$.  Now, we want to consider the \emph{white contour
process}~$\bar C^\circ$, defined as follows. We let $c^\circ_0$, $c^\circ_1$, \dots, $c^\circ_{N^\sE-1}$ be the list of corners that are incident to white vertices, arranged in contour order as above. Then 2$\bar C^\circ(i)$ is the distance between~$c^\circ_i$ and the root of the tree to which it belongs (note that this number is even). In the contour process, white vertices are visited once in every pair of steps, except at times when one of the trees has been fully
explored. The number of such exceptions is $l_k=O(\sqrt{n_k})$, so
clearly the preceding convergence implies
$$\frac{\bar C^\circ(N^\sE\cdot)}{\sqrt{n_k}}\build\longrightarrow_{k\to\infty}^{(d)}\frac{1}{\sqrt{a_\sS}\,
    \sigma_q}(X-\underline{X})\, ,$$
where the limit is understood
under $\mathbb{F}^1_L$. 
Define $\Upsilon^\circ(i)$ to be the number of completely explored
trees when visiting~$c^\circ_i$, as well as $\ell^\circ(i)$ to be the
label of~$c^\circ_i$. Beware that the definitions of~$\bar C^\circ$,
$\Upsilon^\circ$ and~$\ell^\circ$ involve corners listed in contour
order, instead of vertices listed in depth-first order, as in the
definitions of~$H$, $\Upsilon$ and~$\hat\ell$. Similar arguments
entail the following joint convergence.

\begin{crl}
  \label{corlo}
Let $\sS\in \{\sV,\sE,\sF\}$ and $(l_k,n_k)_{k\geq 0}\in \SSL$ for some $L>0$. Then, under
  $\M_{l_k}(\,\cdot\, |\, N^\sS=n_k)$, it holds that
$$\left(\frac{\bar C^\circ(N^\sE\,\cdot )}{\sqrt{{n_k}}},\frac{\Upsilon^\circ(N^\sE\,\cdot )}{\sqrt{{n_k}}}, \frac{\ell^\circ(N^\sE\,\cdot)}{{n_k}^{1/4}}\right)\build\longrightarrow_{k\to\infty}^{(d)}
\left(\frac{1}{\sqrt{a_\sS}\,\sigma_q}(X-\underline{X}),-\sis\,\underline{X},\sqrt{\frac{2\sis}{3}}Z\right)\,
,$$
in distribution in the space $\WW^3$, where, in the limit, $X$, $Z$ are understood under the law~$\mathbb{F}^1_L$.
\end{crl}

\section{Proof of the invariance principle}\label{sec:conv-brown-disk-4}

In this section, we prove Theorem \ref{sec:admiss-regul-crit} and
Theorem \ref{thmboltz_nfree}. The arguments, originating in
\cite[Section~8]{legall11}, are now very standard, and have been
applied successfully in
\cite{BeLG,addarioalbenque2013simple,BeJaMi14,abr14} in
particular. Our approach is an easy adaptation of the arguments that
can be found in either of these papers, so here we will be a bit
sketchy.
Let~$q$ be a regular critical sequence as in the previous section.

\subsection{Convergence of conditioned pointed
  maps}\label{sec:conv-cond-point}

The goal of this subsection is to prove the following analog of Theorem~\ref{sec:admiss-regul-crit}
under the pointed laws $\W^{\bullet,\sS}_{l,n}$, defined by~\eqref{WpointS}.

\begin{thm}\label{sec:admiss-regul-crit-2}
Let~$\sS$ denote one of the symbols $\sV$, $\sE$, $\sF$, and $(l_k,n_k)_{k\geq 0}\in \SSL$ for some $L>0$. For $k\ge 0$, denote by~$M_k^\bullet$ a random map with distribution $\W^{\bullet,\sS}_{l_k,n_k}$. Then
$$\left(\frac{4\sis^2}{9}\, n_k\right)^{-1/4}M^\bullet_k\build\longrightarrow_{k\to\infty}^{(d)}\bd_{L}$$
in distribution for the Gromov--Hausdorff topology. 
\end{thm}

Fix $\sS\in \{\sV,\sE,\sF\}$ and $(l_k,n_k)_{k\geq 0}\in \SSL$ for some $L>0$. For every $k\geq 0$, consider a forest $(F,\ell)$ of labeled mobiles with law
$\M_{l_k}(\,\cdot\, |\, N^\sS=n_k)$, and let~$M_k^\bullet$ be the random map with
distribution~$\W^{\bullet,\sS}_{l_k,n_k}$ obtained by the process of Proposition~\ref{sec:boutt-di-franc-1}.

The proof of the convergence of~$M_k^\bullet$ to the Brownian disk
follows closely in spirit that of \cite[Section~8]{legall11}. Let
$$D'_k(i,j)=d_{M_k^\bullet}(c^\circ_i,c^\circ_j)\, ,\qquad 0\leq i,j\leq
N^\sE$$
(with the convention that $c^\circ_{N^\sE}=c^\circ_0$) 
and extend $D'_k$ to a continuous function on $[0,N^\sE]^2$ by a
formula similar to~\eqref{eq:6}. In this way, $D'_k$ satisfies the
triangle inequality, and one has 
\begin{equation}\label{eq:7}
  D'_k(i,j)\leq\ell^\circ(i)+\ell^\circ(j)-2\max\left(\min_{[i\wedge
    j,i\vee j]} \ell^\circ,\min_{[i\vee j,N^\sE]\cup[0,i\wedge
    j]}\ell^\circ\right)+2\, ,
\end{equation}
see for instance \cite[Lemma~3.1]{legall06} for the special case of
$p$-mobiles (in which black vertices all have degree~$p$,
which corresponds to $2p$-angulations via the BDG bijection), but
the proof in this general context is the same. Clearly, it also holds
that, if~$c^\circ_i$ and~$c^\circ_j$ are incident to the same vertex, which means that $\bar C^\circ(i)=\bar C^\circ(j)=\min_{r\in
  [i\wedge j,i\vee j]} \bar C^\circ(r)$ and
$\Upsilon^\circ(i)=\Upsilon^\circ(j)$ then $D'_k(i,j)=0$. This
generalizes to all $s$, $t\in [0,N^\sE]$ rather than just integer values.

Now for $s$, $t\in [0,1]$ let $D'_{(k)}(s,t)=(4\sis^2
n_k/9)^{-1/4}D'_k(N^\sE s,N^\sE t)$. The same proof as \cite[Proposition~3.2]{legall06} 
(the key ingredients being~\eqref{eq:7} and the convergence of the
rescaled labeled process~$\ell^\circ$, established in Corollary~\ref{corlo}) shows
that, under $\M_{l_k}$, the laws of $D'_{(k)}$ are tight in the space
$\mathcal{C}([0,1]^2,\R)$. Therefore, from any extraction, one can further extract a
subsequence along which one has the following joint convergence in
distribution under $\M_{l_k}(\cdot\, |\, N^\sS=n_k)$,
\begin{equation}
  \label{eq:8}
  \left(\frac{\bar C^\circ(N^\sE\,
    \cdot)}{\sqrt{{n_k}}/(\sqrt{a_\sS}\sigma_q)},\frac{\Upsilon^\circ(N^\sE\,
    \cdot)}{\sis\sqrt{{n_k}}},\frac{\ell^\circ(N^\sE\, \cdot)}{(4\sis^2
    {n_k}/9)^{1/4}},D'_{(k)}\right)\build\longrightarrow_{k\to\infty}^{(d)}\left(X-\un{X},-
  \un{X},Z,D'\right)
\end{equation}
where $D'$ is some random continuous function on $[0,1]^2$, and
$(X,Z)$ is the snake process under the law $\mathbb{F}^1_L$ introduced in Section~\ref{sec:snakes}. 
Recall the definition of the set $\mathcal{D}$ in Section~\ref{sec:br-disks}, as well as the definition of~$s_*$, the a.s.\
unique time in $[0,1]$ such that $Z_{s_*}=\inf Z$. We have the
following result, which follows from a simple limiting argument, based
on~\eqref{eq:7} and the discussion below, as well as \eqref{eq:10} for
the last point. 

\begin{lmm}
Almost-surely, it holds that 
\begin{itemize}
\item 
the random function $D'$ is a pseudo-metric on $[0,1]$, such that
$D'\in \mathcal{D}$, and 
\item
for every
$s\in [0,1]$, $D'(s,s_*)=Z_s-Z_{s_*}$. 
\end{itemize} 
\end{lmm}
 
Let $D^*$ be the Brownian disk distance defined from the process
$(X,Z)$ by~\eqref{dstar}. By definition, it holds that $D'\leq D^*$. 
The conclusion will follow from the following re-rooting property. 

\begin{lmm}
  \label{sec:proof-invar-princ}
Let $U$, $V$ be two independent uniform random variables in $[0,1]$,
independent of the other random variables under consideration. Then 
$D'(U,V)$ and $D'(s_*,U)$ have the same distribution. 
\end{lmm}

\begin{proof}
  This is again obtained by a limiting argument. The idea is to couple
  the random variables~$U$ and~$V$ with two uniformly chosen vertices
  of~$M_k^\bullet$. 

  For $i\in \{1,2,\ldots,N^\sV\}$, let $g(i)$ be the first time~$j$
  such that $c^\circ_j$ is the $i$-th white vertex in depth-first
  order. We also let $g(0)=0$ and extend by linear interpolation~$g$
  into a continuous increasing function on $[0,N^\sV]$. 
Recall the definition~\eqref{lambd} of
  $\Lambda^\sV(i)$ (for $0\leq i\leq N^\sE$), the number of white
  vertices among $v_0$, \ldots, $v_i$, where $v_0$, $v_1$\ldots\ are listed in 
  depth-first order. For $i\in \{1,2,\ldots,N^\sV\}$, let
$$K(i)=\inf\{j\in\{0,1,\ldots,N^\sE\}: \Lambda^\sV(j)=i\}\, ,$$
so by definition, $v_{K(i)}$ is the $i$-th visited white vertex in
depth-first order. 
Recall also from Section~\ref{sec:conv-white-cont} that $f(j)+1$ is
the number of distinct vertices incident to the corners
$c_0$, \ldots, $c_j$, so if we let
$$K'(i)=\inf\{j\ge 0:f(j)=i\}\, ,$$
then $c_{K'(i)}$ is the first time of visit of the vertex $v_i$ in
contour order. Consequently, $K'(K(i))$ is the first time of visit, in
the contour sequence $c_0$, $c_1$\ldots, of $v_{K(i)}$. Finally, since
$g(i)$ is the first index~$j$ such that $c^\circ_j$ is incident to
$v_{K(i)}$, we have that
\begin{equation}
  \label{eq:18}
  g(i)=K'(K(i))/2+R_k(i)
\end{equation}
where the error term satisfies $\sup_{1\leq i\leq N^\sV}|R_k(i)|\leq
l_k=O(\sqrt{n_k})$, recalling the discussion of Section~\ref{sec:conv-white-cont}. 

It follows from the last part
of Corollary~\ref{sec:conv-encod-proc-3} and from~\eqref{eq:17}
respectively  that
$$\frac{K(N^\sV \cdot)}{N^\sE}\build\longrightarrow_{k\to\infty}^{}\mathrm{Id}_{[0,1]}\,
,\qquad
\frac{K'(N^\sE\,\cdot)}{2N^\sE}\build\longrightarrow_{k\to\infty}^{}\mathrm{Id}_{[0,1]}\, ,$$
in probability under
$\M_{l_k}(\,\cdot\, |\, N^\sS=n_k)$. 
From
this and~\eqref{eq:18}, we conclude that 
\begin{equation}
  \label{eq:19}
  \frac{g(N^\sV
  \cdot)}{N^\sE}\build\longrightarrow_{k\to\infty}^{}\mathrm{Id}_{[0,1]}\,
,
\end{equation}
still in the same sense. 

Now, if $I$ is a uniform random variable on
$\{1,2,\ldots,N^\sV\}$ independent of the rest, then $v_{K(I)}$ is
uniformly distributed among the white vertices of the forest, that is,
among the vertices of the map $M_k^\bullet$ distinct from the distinguished
vertex~$v_*$. Therefore, if we let $v_{(1)}=v_{K(\lceil N^\sV
  U\rceil)}$ and $v_{(2)}=v_{K(\lceil N^\sV V\rceil)}$, then
$v_{(1)}$, $v_{(2)}$ can be coupled with two independent uniform
vertices $v'_{(1)}$, $v'_{(2)}$ of~$M_k^\bullet$ in such a way that the
conditional probability given $M_k^\bullet$ that $v_{(i)}\neq v'_{(i)}$,
$i\in\{1,2\}$, is at most $1/|\sV(M_k^\bullet)|$. The latter quantity, also
equal to $1/(N^\sV+1)$, converges to~$0$ in probability under
$\M_{l_k}(\,\cdot\, |\, N^\sS=n_k)$ as $k\to\infty$.

Since $v_*$ is a uniform random vertex of $M_k^\bullet$, we obtain that 
$$d_{M_k^\bullet}(v_*,v'_{(1)})\build=_{}^{(d)}d_{M_k^\bullet}(v'_{(1)},v'_{(2)})\,
.$$
Due to the above discussion, outside a set of vanishing probability,
we may assume that $v'_{(1)}=v_{(1)}$ and $v'_{(2)}=v_{(2)}$. 

Now note that, by~\eqref{eq:10},
$$d_{M_k^\bullet}(v_*,v_{(1)})=\ell^\circ\big(g(\lceil N^\sV U\rceil)\big)-\inf
\ell^\circ+1\, ,$$
and, by definition of~$D'_k$,
$$d_{M_k^\bullet}(v_{(1)},v_{(2)})=D'_k\big(g(\lceil N^\sV
U\rceil),g(\lceil N^\sV V\rceil)\big)\, .$$ 
Using \eqref{eq:19}, we conclude that 
$$\left(\frac{4\sis^2{n_k}}{9}\right)^{-1/4}d_{M_k^\bullet}(v_*,v_{(1)})\build\longrightarrow_{k\to\infty}^{(d)}Z_U-\inf
Z=D'(s_*,U)\, ,$$ 
while 
$$\left(\frac{4\sis^2{n_k}}{9}\right)^{-1/4}d_{M_k^\bullet}(v_{(1)},v_{(2)})\build\longrightarrow_{k\to\infty}^{(d)}D'(U,V)\,
.$$
It follows that $D'(U,V)$ and $D'(s_*,U)$ have same distribution, as
claimed.  
\end{proof}

To conclude the proof of Theorem~\ref{sec:admiss-regul-crit-2}, we note that, by Lemma~\ref{sec:proof-invar-princ},
$$E[D'(U,V)]=E[D'(s_*,U)]=E[Z_U-\inf
Z]=E[D^*(s_*,U)]=E[D^*(U,V)]\, ,$$ whence it follows that
$D'(U,V)=D^*(U,V)$ a.s.\ since $D'\leq D^*$. Note that we have used
the fact that $D^*(s_*,U)$ and $D^*(U,V)$ have same distribution, a
fact that follows from Theorem~\ref{sec:conv-brown-disk-1} and
Lemma~\ref{sec:conv-brown-disk-3}. This is in fact the only place where we
use the specific study of Sections \ref{sec:scha-biject-first} and
\ref{sec:conv-brown-disk}. 

This implies, by Fubini's theorem, that a.s.\ $D'(s,t)=D^*(s,t)$ for
a.e.\ $s$, $t\in [0,1]$, so that $D'=D^*$ a.s.\ by a density
argument. This identifies $D'$ uniquely, and shows that the
convergence of $D'_{(k)}$ to $D'=D^*$ holds without having to pass to
a subsequence. From there, showing the Gromov--Hausdorff convergence of
$(4\sis^2n_k/9)^{-1/4}M_k^\bullet$ under $\W^{\bullet, \sS}_{l_k,n_k}$ to
$\bd_{L}$ is routine, see e.g.\ \cite[Section~3.2]{bettinelli11b}.

\subsection{De-pointing}\label{sec:de-pointing}

Here we show how to dispose of the pointing that intervenes in Theorem~\ref{sec:admiss-regul-crit-2}. The argument closely follows
the last section of~\cite{abr14}, see also \cite{BeJaMi14} for a similar
situation.

Similarly to the absolute continuity relation~\eqref{eq:16}, for
$\sS\in \{\sV,\sE,\sF\}$, we have
\begin{equation}
  \label{eq:9}
  \d\W^\sS_{l,n}(\bm)=\frac{K^\sS_{l,n}}{|\sV|}\d(\phi_*\W^{\bullet,\sS}_{l,n})(\bm)
\end{equation}
where $K^{\sS}_{l,n}=\W^{\bullet,\sS}_{l,n}[1/|\sV|]^{-1}$. In particular,
$\W^{\sV}_{l,n}=\phi_*\W^{\bullet,\sV}_{l,n}$ for every $l$, $n$, so there is nothing more to prove for $\sS=\sV$; Theorem~\ref{sec:admiss-regul-crit} is equivalent to Theorem~\ref{sec:admiss-regul-crit-2} in this case.

\bigskip
Now suppose that $\sS\in \{\sE,\sF\}$. Then, by Propositions~\ref{sec:boutt-di-franc-1} and~\ref{sec:conv-encod-proc-2}, it
holds that, for every $\eps>0$,
$$\W^{\bullet,\sS}_{l_k,{n_k}}\left(\Big||\sV|-\frac{a_\sV}{a_\sS}\, {n_k}\,\Big|>\eps
  {n_k}\right)=\M_{l_k}\left(\Big|N^\sV+1-\frac{a_\sV}{a_\sS}\,
  {n_k}\,\Big|>\eps {n_k}\,\middle|\,
  N^\sS={n_k}\right)=\mathrm{oe}({n_k})\,.$$ (Note that a bound of the
form $o(1/{n_k})$ would suffice for the argument to work.) From this
and the fact that $1/|\sV|\leq 1$, we obtain that, for every
$\delta>0$,
$$\W^{\bullet,\sS}_{l_k,{n_k}}\left[\left|\frac{{a_\sV}\,
      {n_k}}{{a_\sS}\,|\sV|}-1\right|\right]\leq \delta +\Big(\frac{a_\sV}{a_\sS}\,
{n_k}+1\Big)\W^{\bullet,\sS}_{l_k,{n_k}}\Big(\Big|\frac{a_\sV}{a_\sS}\,
{n_k}-|\sV|\,\Big|>\delta\,|\sV|\Big)=\delta + \mathrm{oe}({n_k})\, ,$$ 
since 
\begin{align*}
  \W^{\bullet,\sS}_{l_k,{n_k}}\Big(\Big|\frac{a_\sV}{a_\sS}\,
{n_k}-|\sV|\,\Big|>\delta\,|\sV|\Big)\leq &
\W^{\bullet,\sS}_{l_k,{n_k}}\Big(|\sV|\leq
\frac{a_\sV}{a_\sS}\frac{{n_k}}{2} \Big)\\
&+
\W^{\bullet,\sS}_{l_k,{n_k}}\Big(\Big|\frac{a_\sV}{a_\sS}\,
{n_k}-|\sV|\,\Big|>\delta \frac{a_\sV}{a_\sS}\frac{{n_k}}{2}
\Big)=\mathrm{oe}({n_k})\, .
\end{align*}
From this it follows that
$K^\sS_{l_k,{n_k}}= \frac{a_\sV}{a_\sS} {n_k} (1+o(1))$ as $k\to\infty$, and then that 
$$\W^{\bullet,\sS}_{l_k,{n_k}}\left[\left|\frac{K^\sS_{l_k,{n_k}}}{|\sV|}-1\right|\right]\build\longrightarrow_{k\to\infty}^{
}0\, .$$
From this and~\eqref{eq:9}, we obtain the following result. 

\begin{lmm}
 \label{sec:de-pointing-2}
For every $\sS\in \{\sE,\sF\}$, and every $(l_k,n_k)_{k\geq 0}\in \SSL$, one has
  $$\big\|\W^{\sS}_{l_k,n_k}-\phi_*\W^{\bullet,\sS}_{l_k,n_k}\big\|\build\longrightarrow_{k\to\infty}^{}0\,
  ,$$ where $\|\cdot\|$ is the total variation norm.
\end{lmm}

Theorem \ref{sec:admiss-regul-crit} is now a direct consequence of
this statement combined with Theorem \ref{sec:admiss-regul-crit-2}.

\subsection{Proof of the convergence of Boltzmann
  maps}\label{sec:conv-boltzm-maps-1}

This section is dedicated to the proof of Theorem~\ref{thmboltz_nfree}. 
As in the proof of Theorem~\ref{sec:admiss-regul-crit}, we first focus on random maps with
distribution~$\W^\bullet_l$, which are easier to handle, since they are
directly related by Proposition~\ref{sec:boutt-di-franc-1} to random labeled forests with law~$\M_l$.

The result follows from the fact that, for any measurable and
bounded function~$\Phi$, 
\begin{align}
\W_l^\bullet[\Phi]&=\sum_{n\in \A^\sV_l}\W_l^\bullet\big(|\sV|=n+1\big)\,\W_{l,n}^{\bullet,\sV}[\Phi]\label{eq:12}\\
&=\sum_{n\in \A^\sV_l}\M_l\big(N^\sV=n\big)\,\W_{l,n}^{\bullet,\sV}[\Phi].\nonumber 
\end{align}
At this point, recall from Lemma~\ref{sec:boltzm-rand-maps} that
$\A^\sV_l= R^\sV_l  \cup (\beta^\sV
l+h^\sV\Z_+)$, where $\beta^\sV\geq 1$ and
$R^\sV_l\subseteq\{0,1,\ldots,\beta^\sV l-1\}$. 
For simplicity, let us use the notation $\beta=\beta^\sV$
and $h=h^\sV$. Therefore,
\begin{align*}
\W_l^\bullet[\Phi]&=\sum_{n\geq 0}\M_l\big(N^\sV=\beta l+h
n\big)\,\W_{l,\beta  l+h n}^{\bullet,\sV}[\Phi] +\mathcal{R}_l\\
&=\int_{\R_+}l^2Q^\sV\big(l,\beta l+h\lfloor l^2A\rfloor\big)\, \W_{l,\beta l+h\lfloor
  l^2 A\rfloor}^{\bullet,\sV}[\Phi]\,\d A +\mathcal{R}_l
\end{align*}
where $|\mathcal{R}_l|\leq \|\Phi\|_\infty \W_l^\bullet(|\sV|\leq
\beta l)$.
Recall that, under $\W_l^\bullet$, the random variable $|\sV|-1$ has
same distribution as a sum of $l$ i.i.d.\ random variables with
distribution $Q^\sV(1,\cdot)$. The proof of Lemma~\ref{sec:conv-encod-proc-1} yields 
that, under $\W_l^\bullet$
$$\frac{|\sV|}{l^2}\build\longrightarrow_{l\to\infty}^{(d)}\frac{1}{\siv^2}\mathcal{A}^\bullet$$
where $\mathcal{A}^\bullet$ is a stable random variable with density $j_1$. 
Clearly, this implies that $\mathcal{R}_l \to 0$ as $l\to\infty$.
Now assume that $\Phi=\varphi((2l/3)^{-1/2}M)$ where~$\varphi$ is a
continuous and bounded function on the Gromov--Hausdorff space. Then
one has, by Theorem~\ref{sec:admiss-regul-crit-2},
$$\W_{l,\beta l+h\lfloor
  l^2A\rfloor}^{\bullet,\sV}\big[\varphi\big((2l/3)^{-1/2}M\big)\big]\build\longrightarrow_{l\to\infty}^{}\E\Big[\varphi\big((h\siv^2 A)^{1/4}\,\bd_{(h\siv^2 A)^{-1/2}}\big)\Big]=
\E\big[\varphi\big(\bd_{1,h\siv^2 A}\big)\big]\, ,$$
where the lase equality follows from Remark~\ref{sec:brownian-disks}. 
At this point, we apply Lemma~\ref{sec:conv-encod-proc-1}, 
which implies that for every $A>0$, 
$$l^2Q^\sV\big(l,\beta l+h\lfloor l^2A\rfloor\big)\build\longrightarrow_{l\to\infty}^{} h j_{1/\siv}(hA)=h\siv^2 j_1\big(h\siv^2 A\big)\, .$$
Since we are dealing with
probability densities, the Scheff\'e Lemma implies that the latter
convergence holds in fact in $L^1(\d A)$, and we conclude that 
\begin{align*}
  \lim_{l\to\infty}\W_l^\bullet\left[\varphi\big((2l/3)^{-1/2}M\big)\right]
&= \int_{\R_+}h\siv^2j_1\big(h\siv^2 A\big)\, \E\big[\varphi\big(\bd_{1,h\siv^2 A}\big)\big]\, \d A\\
&= \int_{\R_+} j_1(A)\E[\varphi(\bd_{1,A})]\,\d A
\end{align*}
and this is equal to $\E[\varphi(\mathrm{FBD}^\bullet_1)]$.  
The second part of Theorem~\ref{thmboltz_nfree} follows. 

\bigskip
To obtain the result under $\W_l$ instead of $\W_l^\bullet$, note that~\eqref{eq:16} implies
\begin{equation}
  \label{eq:13}
  \W_l(|\sV|=n+1)=K_l\,  \frac{\W^\bullet_l(|\sV|=n+1)}{n+1}
\end{equation}
where $K_l=\W^\bullet_l[1/|\sV|]^{-1}$.
We then use the following lemma, which is certainly known, but for which we did not
find a proper reference. 

\begin{lmm}\label{sec:proof-conv-boltzm}
  Let $X_1$, $X_2$\ldots be a sequence of i.i.d.\ r.v.s with
  values in $\{1,2,3,\ldots\}$, and such that 
$$\P(X_1>k)\build\sim_{k\to\infty}^{} \frac{c}{k^\alpha}$$
for some constants $c\in (0,\infty)$ and $\alpha\in(0,1)$. Then
$$\E\left[\frac{l^{1/\alpha}}{X_1+\ldots+X_l}\right]\build\longrightarrow_{l\to\infty}^{}\E\left[\frac{1}{S}\right]$$
where~$S$ is the limit in distribution of $(X_1+\ldots+X_l)/l^{1/\alpha}$
as $l\to\infty$ (so that $S$ is a stable distribution of index
$\alpha$). 
\end{lmm}

\begin{proof}
By
hypothesis and standard facts on stable domains of attractions
\cite[Chapter~8]{BGT}, our hypotheses imply that as $s\uparrow 1$, 
$$\E[s^{X_1}]=1-c'(1-s)^\alpha(1+o(1))$$
for some constant $c'\in (0,\infty)$ depending only on $c$ and
$\alpha$. Applying this to $s=\exp(-\lambda)$ for $\lambda\geq 0$ implies that 
$$\E[\exp(-\lambda X_1)]=1-c'\lambda^\alpha(1+o(1))$$
as $\lambda\downarrow 0$, so that there exists $c''\in (0,\infty)$
such that for every $\lambda\in [0,1]$, one has 
\begin{align}
  \E[\exp(-\lambda
X_1)]&\leq 1-c''\lambda^\alpha \nonumber\\
	&\leq \exp(-c''\lambda^\alpha)\, .  \label{eq:11}
\end{align}
On the other hand, the assumption that $X_1\geq 1$ a.s.\ implies that 
$\E[\exp(-\lambda X_1)]\leq
\exp(-\lambda)$ for every $\lambda\geq 0$. This implies that, possibly
by choosing $c''$ smaller, one can assume that~\eqref{eq:11} is valid
for every $\lambda\geq 0$, as we supposed $\alpha \le 1$.

Now note that, for every $x>0$, one has (using the 
inequality $\ind_{[0,1]}(u)\leq e\, \exp(-u)$ in the first step)
\begin{align*}
  \P\big(X_1+\ldots+X_l \leq xl^{1/\alpha}\big)&\leq e\, 
\E\big[\exp\big(-(X_1+\ldots+X_l)/(xl^{1/\alpha})\big)\big]\\
&\leq e\, \E\big[\exp\big(-X_1/(xl^{1/\alpha})\big)\big]^l\\
&\leq e\, \exp(-c''x^{-\alpha})
\, ,
\end{align*} 
where we used the version of~\eqref{eq:11} valid for all $\lambda$ at
the last step. 
This stretched-exponential tail bound is uniform in $l$ and clearly implies the convergence
of all negative moments. 
\end{proof}

Since, as we observed, $|\sV|-1$ under $\W^\bullet_l$ is distributed
as a sum of~$l$ i.i.d.\ random variables satisfying the hypotheses of
Lemma~\ref{sec:proof-conv-boltzm} with $\alpha=1/2$ (by Lemma~\ref{sec:conv-encod-proc-1}), this entails that
$$\lim_{l\to\infty}\frac{K_l}{l^2}=\E\left[\frac{\siv^2}{\mathcal{A}^\bullet}\right]^{-1}=\frac{1}{\siv^2}\,
.$$ Repeating the previous argument, only changing
$\W_l^\bullet(|\sV|=n+1)$ by $\W_l(|\sV|=n+1)$ in \eqref{eq:12} and
applying \eqref{eq:13}, and then performing the same steps using the
equivalent we obtained for $K_l$, we obtain
$$\lim_{l\to\infty}\W_l\left[\varphi\big((2l/3)^{-1/2}M\big)\right]
= \int_{\R_+}\frac{h\siv^2}{h\siv^2 A} j_1\big(h\siv^2 A\big)\, \E\big[\varphi\big(\bd_{1,h\siv^2 A}\big)\big]\, \d A$$
and this is $\E[\varphi(\mathrm{FBD_1})]$, as wanted.


\bibliographystyle{abbrv}
\bibliography{biblio}

\end{document}